\documentclass[aos,preprint]{imsart}

\RequirePackage[OT1]{fontenc}
\usepackage{amsthm,amsmath,multirow,graphicx,subfigure,booktabs,url,mathtools,wrapfig,mathrsfs,bm,relsize,caption,xcolor, enumitem, float, algorithm}

\usepackage[noend]{algpseudocode}

\RequirePackage[numbers]{natbib}
\RequirePackage[psamsfonts]{amssymb}
\usepackage[colorlinks=true,
linkcolor=blue,
urlcolor=blue,
citecolor=blue]{hyperref}

\startlocaldefs
\numberwithin{equation}{section}
\newtheorem{conj}{Conjecture}
\newtheorem{thm}[conj]{Theorem}
\newtheorem{cor}[conj]{Corollary}
\newtheorem{prop}[conj]{Proposition}

\newtheorem{lemma}{Lemma}
\newtheorem{ass}{Assumption}
\newtheorem{fact}{Fact}

\newtheorem{innercustomgeneric}{\customgenericname}
\providecommand{\customgenericname}{}
\newcommand{\newcustomtheorem}[2]{%
	\newenvironment{#1}[1] 
	{%
		\renewcommand\customgenericname{#2}%
		\renewcommand\theinnercustomgeneric{##1}%
		\innercustomgeneric
	}
	{\endinnercustomgeneric}
}

\newcustomtheorem{customAss}{Assumption}

\theoremstyle{definition}\newtheorem{remark}{Remark}

\def\Cov{\text{Cov}}   
\def\PP{\mathbb{P}}
\def\EE{\mathbb{E}}
\def\RR{\mathbb{R}}
\def\wh{\widehat}
\def\eps{\varepsilon}
\def\bI{{\bm I}}
\def\wt{\widetilde}

\def\Y{{\bm Y}}
\def\y{{\bm y}}
\def\X{{\bm X}}
\def\E{{\bm E}}
\def\W{{\bm W}}

\def\Eps{{\bm \eps}}

\def\g{\gamma}
\def\diag{\textrm{diag}}
\def\C{\mathcal{C}}
\def\R{\mathcal{R}}
\def\cE{\mathcal{E}}
\def\l12{\ell_1 / \ell_2}
\def\i{\infty}
\def\tr{\textrm{tr}}
\def\P{P_{\lambda_2}}
\def\Q{Q_{\lambda_2}}
\def\1{\bm{1}}
\def\rank{\textrm{rank}}
\def\Ttheta{\Theta^*}
\def\B{B^*}
\def\Ps{\Psi^*}

\makeatletter
\def\BState{\State\hskip-\ALG@thistlm}
\makeatother

\DeclarePairedDelimiter{\floor}{\lfloor}{\rfloor}

\usepackage{tikz}
\usetikzlibrary{positioning}
\usetikzlibrary{calc,fit}
\usepgflibrary{shapes.geometric}
\tikzset{
	server/.style={circle, inner sep=0.0mm, minimum width=.5cm,
		draw=white,fill=teal!35,thick},
	buffer/.style={circle, inner sep=0.0mm, minimum width=.5cm,
		draw=white,fill=blue!35,thick},
	routing/.style={circle, inner sep=0.0mm, minimum width=.5cm,
		draw=white,fill=pink!90,thick},
	center/.style={circle,inner sep=0.0mm, minimum width=0cm,
		draw=white,fill=white,thick},
}

\endlocaldefs

\begin{document}
	\begin{frontmatter}
		\title{Adaptive Estimation in Multivariate Response Regression with Hidden Variables}
		\runtitle{Adaptive estimation with hidden variables}
		
		\begin{aug}
			\author{\fnms{Xin} \snm{Bing}\thanksref{t1} \ead[label=e1]{xb43@cornell.edu}}
			\and
			\author{\fnms{Yang} \snm{Ning} \ead[label=e2]{yn265@cornell.edu}}
			\and
			\author{\fnms{Yaosheng} \snm{Xu}	\ead[label=e3]{yx433@cornell.edu}}
			
			\thankstext{t1}{Corresponding author}
			
			\runauthor{X. Bing, Y. Ning and Y. Xu}
			
			%\affiliation{Cornell University}
			
			\address{
				%X. Bing, Y. Ning and Y. Xu\\
				Department of Statistics and Data Science\\ 
				Cornell University\\
				%Ithaca, New York 14853-3801\\
				%USA\\
				\printead{e1,e2,e3}}
			%\address{
			%	Y. Ning\\
			%	Department of Statistics and Data Science\\ 
			%	Cornell University\\
			%	Ithaca, New York 14853-3801\\
			%	USA\\
			%	\printead{e2}}
			%\address{
			%	Y. Xu\\
			%	Department of Statistics and Data Science\\
			%	Cornell University\\
			%	Ithaca, New York 14853-3801\\
			%	USA\\
			%	\printead{e3}}
		\end{aug}

		\begin{abstract}
			A prominent concern of scientific investigators is the presence of unobserved hidden variables in association analysis. Ignoring hidden variables 
			%in the analysis 
			often yields biased statistical results and misleading scientific conclusions. Motivated by this practical issue, this paper studies the 
			%estimation of the partial direct effect in 
			multivariate response regression with hidden variables, $Y = (\Ps)^TX + (B^*)^TZ  + E$, where $Y \in \RR^m$ is the response vector, $X\in \RR^p$ is the observable feature, $Z\in \RR^K$ represents the vector of unobserved hidden variables, possibly correlated with $X$, and $E$ is an independent error. The number of hidden variables $K$ is unknown and both $m$ and $p$ are allowed, but not required, to grow with the sample size $n$. 
			
			%Since only $Y$ and $X$ are observable, the regression parameter $\Ps$ is not identifiable. However, the partial direct effect $\Ttheta = \Ps P_{\B}^{\perp}$ is always identifiable and is treated as the parameter of our interest. It has practical meanings and represents the effect of $X$ on $Y$ that can not be explained by the hidden effects.
			
			Though $\Ps$ is shown to be non-identifiable due to the presence of hidden variables, we propose to identify the projection of $\Ps$ onto the orthogonal complement of the row space of $B^*$, denoted by $\Ttheta$. The quantity $(\Ttheta)^TX$ measures the effect of $X$ on $Y$ that cannot be explained through the hidden variables, and thus $\Ttheta$ is  treated as the parameter of interest. Motivated by the identifiability proof, we propose a novel %and computationally efficient 
			estimation 
			algorithm for $\Ttheta$, called HIVE, under homoscedastic errors. 
			The first step of the algorithm estimates the best linear prediction of $Y$ given $X$, in which the unknown coefficient matrix exhibits an additive decomposition of $\Ps$ and a dense matrix due to the correlation between $X$ and %the hidden variable 
			$Z$. Under the sparsity assumption on $\Ps$, we propose to minimize a penalized least squares loss by regularizing $\Ps$ and the dense matrix via group-lasso and multivariate ridge, respectively. Non-asymptotic deviation bounds of the in-sample prediction error are established. 
			Our second step estimates the row space of $B^*$ by leveraging the covariance structure of the residual vector from the first step. In the last step, we estimate $\Ttheta$ via projecting $Y$ onto the orthogonal complement of the estimated row space of $B^*$ to remove the effect of hidden variables. Non-asymptotic error bounds of our final estimator of $\Ttheta$, which are valid for any $m,p,K$ and $n$, are established. We further show that, under mild assumptions, the rate of our estimator matches the best possible rate with known $B^*$  and is adaptive to the unknown sparsity of $\Ttheta$ induced by the sparsity of $\Ps$.  
			The model identifiability, estimation algorithm and statistical guarantees are further extended to the setting with heteroscedastic errors. 
			Thorough numerical simulations and two real data examples are provided to back up our theoretical results.

		\end{abstract}
		
		
	\end{frontmatter}

	\section{Introduction}

	Multivariate response regression has been widely used to evaluate how predictors are associated with multiple response variables and is ubiquitous in many areas including genomics, epidemiology, social science and economics \citep{anderson_book}. Most of the existing research on multivariate response regression assumes that the collected predictors are sufficient to explain the responses. However, due to cost constraints or ethical issues, oftentimes there exist unmeasured hidden variables that are associated with the responses as well. Ignoring the hidden variables often leads to biased estimates.  
	
	In this paper, we consider the following multivariate response regression with hidden variables. Let $Y\in \RR^m$ denote the response vector, $X\in \RR^{p}$ denote the observable predictors and $Z\in \RR^{K}$ be the unobservable hidden variables. The multivariate response regression model postulates 
	\begin{equation}\label{model_1}
	Y = (\Psi^*)^TX + (B^*)^TZ  + E,
	\end{equation}
	where $\Psi^*\in \RR^{p\times m}$ 
	and $B^*\in \RR^{K\times m}$ are unknown deterministic matrices and  $E \in \RR^{m}$ is a stochastic error with zero mean and a diagonal covariance matrix $\Sigma_E$. The random error $E$ is independent of $(X, Z)$ and the hidden variable $Z$ is possibly correlated with $X$. The number of hidden variables, $K$, is unknown and is assumed to be less than $m$. As we can subtract means from both sides of (\ref{model_1}), we consider $Y$, $X$ and $Z$ have mean zero. Without loss of generality, we assume $\Sigma=\Cov(X)$ and $\Sigma_Z=\Cov(Z)$ are strictly positive definite and $\rank(B^*) = K$. Otherwise, one might reduce the dimensions of $X$ and $Z$ such that these conditions are met.

	Assume that we observe $n$ i.i.d. copies of $(X, Y)$ and stack them together as a design matrix $\X\in \RR^{n\times p}$ and a response matrix $\Y\in \RR^{n\times m}$. In practice, the number of response variables $m$, or the number of features $p$, or both of them, can be greater than the sample size $n$. 
	%{\red The main interest of model (\ref{model_1}) is to study quantities related with $\Ps$ in order to understand the association between the primary features $X$ and the responses $Y$ in the presence of unobserved hidden variables $Z$. }

	The proposed model unifies and generalizes the following two strands of research that emerges in a variety of applications.
	
	{\em 1. Surrogate variable analysis (SVA) in genomics}. The measurements of high-throughput genomic data are often confounded by unobserved factors. To remove the influence of unobserved confounders, surrogate variable analysis (SVA) based on model (\ref{model_1}) has been proposed for the analysis of biological data \citep{LeekStorey07, Leek2008, Teschendorff, Chakraborty, Gagnon2012, Houseman12, sun2012}. In these applications, the response vector $Y$ is often the gene expression or DNA methylation levels at $m$ sites, which is usually much larger than the sample size $n$. The covariate $X$ is  a small set of exposures (e.g., treatment variables), whose dimension $p$ is assumed to be fixed in the theoretical analysis \citep{Lee2017, wang2017, McKennan19}. Since $p$ is small, the existing SVA methods apply the ordinary least squares (OLS) $(\X^T\X)^{-1}\X^T\Y$ to estimate the main regression effect and then remove the bias of the OLS estimator, originating from the correlation between $Z$ and $X$. However, to avoid confounding issues, researchers tend to collect as many features as possible and adjust them in the regression model afterwards. In this case, $p$ can be large and even much larger than $n$, whence the existing SVA methods are not applicable as the OLS estimator may not exist. Our work extends the scope of the SVA in the sense that a unified estimation procedure and theoretical justification are developed under model (\ref{model_1}) where both $p$ and $m$ are allowed, but not required, to grow with $n$. We refer to Section \ref{sec_related_work} for detailed comparisons with existing SVA literature.

	{\em 2. Structural equation model in causal inference}. Model  (\ref{model_1}) can  be also framed as linear structural equation models \citep{hox1998introduction}. Suppose the causal structure among $(X, Z, Y)$  is represented by the directed acyclic graph (DAG) in Figure \ref{fig_dag}. As shown in this graph, both observed variables $X$ and hidden variables $Z$ are the causes of $Y$, as $(X,Z)$ are the parents of $Y$. Under the linearity assumption, the causal structure of $(X,Z)\rightarrow Y$ is modeled by  equation (\ref{model_1}). Similarly, the DAG in Figure \ref{fig_dag} also implies that $X$ is the cause of $Z$, which can be further modeled via
	\begin{equation}\label{modelZ}
	Z = D^T X + W,
	\end{equation}
	where $D\in\RR^{p\times K}$ is a deterministic matrix and $W\in\RR^K$ is a random noise, independent of $X$ and $E$. Since $Z$ is not observable, model (\ref{model_1})  and (\ref{modelZ}) are viewed as linear structural equation models with hidden variables  \citep{diaz}. Using the terminology in causal mediation analysis, the parameter $\Ps$ in (\ref{model_1}) represents the direct causal effect of $X$ on $Y$. It is worthwhile to note that the proposed framework is more general than linear structural equation models as model (\ref{modelZ}) is not imposed. In particular, we allow an arbitrary dependence structure between $X$ and $Z$, whereas the linear structural equation model assumes $X$ is a cause of $Z$ with the independence between $X$ and $W$.
	
	%\begin{figure}[ht]
	%   \centering
	%  \includegraphics[width = .4\textwidth]{DAG.pdf}
	% \caption{Illustration of DAG}
	%  \label{fig_dag}
	%\end{figure}
	
	\begin{figure}[ht]
		\centering
		\begin{tikzpicture}[inner sep=1.5mm]
		\node[server,minimum size=0.4in] at (6,4) (X)  {\textbf{X}};
		\node[routing,minimum size=0.4in] (Z) [right=0.9in of X] {\textbf{Z}};
		\node[center](C)[right=0.43in of X]{};
		\node[buffer,minimum size=0.4in]  (Y) [below=0.8in of C.north] {\textbf{Y}};

		\draw[thick,->] (X) -- ++  (Y)  ;
		\draw[thick,->] (Z) -- ++ (Y) ;
		\draw[thick,->] (X) -- ++ (Z);
		
		\path[thick]
		;
		
		\node [left,align=left] at ($(Y)+(-0.3in,.43in)$) {\textbf{$\Ps$}};
		\node [left,align=left] at ($(X)+(0.8in,.15in)$) {\textbf{$D$}};
		\node [left,align=left] at ($(Z)+(-0.02in,-.58in)$) {\textbf{$B^*$}};
		\end{tikzpicture}
		\caption{Illustration of the DAG under model (\ref{model_1}) and (\ref{modelZ})}
		\label{fig_dag}
	\end{figure}
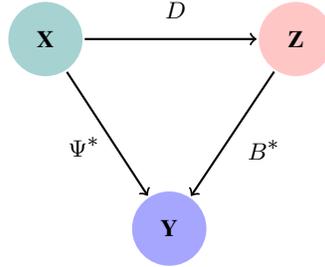
	
	\subsection{Our contributions}
	%We now state our main contributions in this paper. 
	
	\subsubsection*{\bf Identifiability}
	Our first contribution is to investigate the identifiability of model (\ref{model_1}). We show in Proposition \ref{prop_ident_nece} of Section \ref{sec_ident} that $\Ps $ in model (\ref{model_1}) is not identifiable. This motivates us to focus on an alternative estimand, the projection of $\Ps $ onto the orthogonal complement of the row space of $B^*$, which is identifiable and has desirable interpretations in practice. 
	%Our first contribution is to address a fundamental question of this model, that is, which part of $\Ps $ is identifiable.  
	
	We start by rewriting model (\ref{model_1}). Denote by $(A^*)^T X$ the $L_2$ projection of $Z$ onto the linear space of $X$ and by $W=Z-(A^*)^TX$ its residual, where
	\begin{equation}\label{def_A}
	A^*=\left\{\EE\left[XX^T\right]\right\}^{-1}\EE\left[XZ^T\right]\in\RR^{p\times K}.
	\end{equation}
	We emphasize that we do not require model (\ref{modelZ}), or equivalently, the independence between $W$ and $X$. For this reason, we use a different notation $A^*$ rather than $D$ to denote the coefficient of the $L_2$ projection. We then decompose the effect of hidden variable $Z$ as 
	$
	(B^*)^TZ=(A^*B^*)^TX+(B^*)^TW. 
	$
	Plugging this into (\ref{model_1}) yields
	\begin{align}\label{model}\nonumber
	Y &~=~ (\Ps  + A^*B^*)^TX + (B^*)^TW + E\\
	&~ = ~ (\underbrace{\Ps P_{\B}^\perp}_{\Ttheta} + \Ps  P_{\B} + \underbrace{A^*B^*}_{L^*})^TX+ \underbrace{(B^*)^T W+E}_{\eps},
	\end{align}
	where $P_{\B} = B^{*T}(B^*B^{*T})^{-1}\B \in \RR^{m\times m}$ is the projection matrix onto the row space of $B^*$, $P_{\B}^\perp = \bI_m - P_{\B}$, $L^*=A^*B^*$ and the  residual vector $\eps =(B^*)^T W+E$  satisfies $\EE[\eps] = 0$ and $\Cov(X,\eps) = 0$.	
	
	In (\ref{model}), we decompose $\Ps$ into two components $\Ps P_{\B}^\perp$ and $\Ps P_{\B}$. The former, denoted by $\Ttheta$, is the projection of $\Ps $ onto the orthogonal complement of the row space of $B^*$. %Since $(\Ps P_{\B})^TX$ shares the same row space of $B^*$, it can be rewritten as $(B^*)^T\wt Z$ with some suitable hidden variables $\wt Z$ in model (\ref{model_1}). Thus, $\Ttheta $ measures the association between $X$ and $Y$ that cannot be explained by any hidden variables.  
	Since ${\Ttheta}^T X$ is orthogonal to ${\B}^T Z$, ${\Ttheta}^T X$ measures the effect of $X$ on $Y$ in the multivariate response regression (\ref{model_1}) that cannot be explained through the effect of hidden variables.	In the mediation analysis, $\Ttheta$ is referred to as the ``partial'' direct effect of $X$ on $Y$. 
	
	%	On the other hand, despite the fact that  $\Ps$ is not identifiabile, we establish that $\Ttheta$ is always a bona fide parameter of interest.
	
	In Propositions \ref{prop_ident} and \ref{prop_ident_W} of Section \ref{sec_ident} we establish the sufficient and necessary conditions for the identifiability of $\Ttheta$  when the error $E$ is homoscedastic, that is  $\Sigma_E=\tau^2\bI_m$. 
	This covariance structure of the residual vector $\eps$ is crucial to show the results, as detailed in Section \ref{sec_ident}. 
	However, the sufficiency no longer holds in the presence of heteroscedastic error, when $\Sigma_E$ is a diagonal matrix with unequal entries. Inspired by \cite{hpca}, we introduce a mild incoherence condition on the right singular vectors of $B^*$ to identify the row space of $B^*$, which is an important intermediate step towards identifying $\Ttheta$. We show in Proposition \ref{prop_ident_hetero} of Section \ref{sec_hetero} that this incoherence condition guarantees the identifiability of $\Ttheta$ in the heteroscedastic case.

	Summarizing, the parameter $\Ttheta$ is identifiable under both homoscedastic and heteroscedastic errors. We focus on the estimation of $\Ttheta$ throughout the paper. In cases where $\Ps$ is the primary interest, the information in $\Ttheta$ is still helpful to infer $\Ps$. In particular, our analysis of $\Ttheta$ carries over to $\Ps$, under $\Ps P_{\B} \to 0$ as $m = m(n)\to \i$, a common condition of $\Ps$ in the SVA literature \citep{Lee2017}. We refer to Section \ref{sec_usage_theta} for more detailed discussions of using $\Ttheta$ to infer $\Ps$.

	\subsubsection*{\bf Estimation of $\Ttheta$}
	Our second contribution is to propose a new method for estimating $\Ttheta$. In particular, our approach can handle the case when $p>n$ and the OLS commonly used in the SVA literature does not exist. To deal with the high dimensionality of  $\Ttheta$, we assume that there exists a small subset of $X$ that are associated with $Y$ in model (\ref{model_1}). Such a row-wise sparsity assumption on the coefficient matrix has been widely used in multivariate response regression, for instance,  \cite{bv,bunea2012,Bunea2014,lounici2011,ob,yuanlin}, just to name a few. While $\Ps$ is not identifiable, we assume $\Ps$ lies in the space $\Omega$
	\begin{equation}\label{Psi_space}
	\Ps\in\Omega:= \left\{M \in \RR^{p\times m}: \|M\|_{\ell_0 / \ell_2}\le s_*\right\},
	\end{equation}
	where $s_* \le p$ and  $\|M\|_{\ell_0 / \ell_2}=\sum_{j=1}^p1_{\{\|M_{j\cdot}\|_2\ne 0\}}$ is the number of nonzero rows. As a result, (\ref{Psi_space}) implies that $\Ttheta\in\Omega$ as
	\begin{equation}\label{Theta_space}
	\|\Ttheta\|_{\ell_0 / \ell_2} = \|\Ps P_{\B}^{\perp} \|_{\ell_0 / \ell_2} \le s_*.
	\end{equation}
	Our estimation procedure	consists of three steps: first estimate the best linear prediction of $Y$ given $X$; then estimate the row space of $B^*$ and finally estimate $\Ttheta$.

	The first step is  critical but challenging especially when $p$ is large. In Section \ref{sec_est_fit}, we propose a new optimization-based approach with a combination of the group-lasso penalty \citep{yuanlin} and the multivariate ridge penalty. The group-lasso penalty aims to exploit the row-wise sparsity of $\Ps$ in (\ref{Psi_space}), while the multivariate ridge penalty regularizes the additional dense signal $L^*$ due to the hidden variables $Z$ (see model (\ref{model})). %{\color{red} Our method is a generalization of the lava approach \citep{chernozhukov2017} to the multivariate regression setting, where the coefficient matrix is a sum of a sparse matrix and a dense matrix.} 
	The proposed procedure is easy to implement and has almost the same complexity as solving a group-lasso problem. We refer to Section \ref{sec_est_fit} for detailed discussions of computational and theoretical advantages  of our estimator over other competing methods.  
	%Under the classical multivariate regression setting, this method could be used to estimate the regression coefficient matrix that has an additive decomposition of a sparse matrix and a dense matrix. This could be of its own interest. Our method can be view as a multivariate generalization of the lava approach \citep{chernozhukov2017} which is originally proposed to capture an additive decomposition of a sparse vector and a dense vector in one-dimensional regression setting. 
	
	Our second step is to estimate the row space of $B^*$ or equivalently $P_{\B}$. When the noise is homoscedastic, we directly apply 
	the principle component analysis (PCA) to the sample covariance matrix of the estimated residuals (see Section \ref{sec_est_P}). The resulting first $K$ eigenvectors are then used to estimate $P_{\B}$. However, PCA may lead to biased estimates under heteroscedastic error, especially when $m$ is fixed. To deal with heteroscedasticity, we adapt the HeteroPCA  algorithm, originally proposed by \cite{hpca}, to our setting. %The need of developing different estimation methods for $P_{\B}$ under homoscedastic and heteroscedastic errors echoes the difference of the aforementioned identifiability conditions in these two settings. 

	In the last step, we project $Y$ onto the orthogonal complement of the estimated row space of $B^*$ to remove the effect of hidden variables, and then recover $\Ttheta$ by applying group-lasso to the projected $Y$.

	Our entire procedure is summarized in Algorithm \ref{alg_1}, called HIVE, representing \underline{HI}dden  {\underline V}ariable  adjustment  {\underline E}stimation. Similarly, the algorithm tailored for the heteroscedastic error is referred to as H-HIVE in Algorithm \ref{alg_2}. For the convenience of practitioners, we also provide detailed discussions on practical implementations in Section \ref{sec_practical}, including estimation of the number of hidden variables ($K$), the consequence of overestimating/underestimating $K$, the choice of tuning parameters, data standardization and practical usage of $\Ttheta$ for inferring $\Ps$.

	\subsubsection*{\bf Statistical guarantees}
	Our third contribution is to establish theoretical properties of our procedure. In Theorem \ref{thm_pred} of Section \ref{thm_pred}, we derive non-asymptotic deviation bounds of the in-sample prediction error, which are valid for any finite $n$, $p$, $m$ and $K$. The error bounds consist of three components: bias and variance terms from the ridge regularization and an error term from the group-lasso regularization.  To understand the advantage of our estimator, we particularize to the orthonormal design and show that our estimator enjoys the optimal rate of the group-lasso when there is no hidden variable (i.e., $L^* = 0$ in model (\ref{model})) and it also achieves the optimal rate of the ridge estimator when $\Ps = 0$. 
	Thus, the rate of our estimator matches the best possible rate even if $L^*=0$ or $\Ps=0$ were known a priori. %For this reason, we refer to it as an adaptive estimator. 
	%Overall, our estimator adapts to the best possible split between the sparse and dense coefficient matrices. 
	
	We further provide theoretical guarantees for the estimation of $\Ttheta$. Specifically, we establish in Theorem \ref{thm_Theta} non-asymptotic upper bounds of the estimation error of our estimator $\wt \Theta$ based on any estimator $\wh P$ of $P_{\B}$. As expected, the estimation error of $\wt \Theta$ depends on the accuracy of $\wh P$ for estimating $P_{\B}$. When $P_{\B}$ can be estimated accurately, our estimator $\wt \Theta$ achieves the optimal rate in the oracle case with known $\B$ (see the subsequent paragraph of Theorem \ref{thm_Theta}).
	However, if the estimation error of $P_{\B}$  is relatively large, we can balance this term with the error of the group-lasso to attain a more refined rate via a suitable choice of the regularization parameter. In Theorem \ref{thm_U} of Section \ref{sec_theory_theta} and Theorem \ref{thm_U_hetero} of Section \ref{sec_hetero}, we further establish non-asymptotic error bounds for our proposed estimators of $P_{\B}$ under both homoscedastic and heteroscedastic errors. These results together with Theorem \ref{thm_Theta} provide the final upper bounds of the estimation error of $\wt \Theta$. 
	%In order to facilitate understanding, we simplify  results in Corollary \ref{cor_rate} under reasonable conditions and discuss its optimality subsequently in Remark \ref{rem_rate_Theta}. 
	To deal with heteroscedastic errors, we develop a new robust $\sin\Theta$ theorem in Appendix \ref{sec_sintheta} to control the perturbation of eigenspaces in the Frobenius norm. This theorem is essential to the proof of Theorem \ref{thm_U_hetero} and can be of its own interest.
	
	%	Finally, in Theorem \ref{thm_K} of Section \ref{sec_select_K} we provide theoretical justifications for selection of $K$. The effect of overestimating and underestimating $K$ is clarified in Lemma \ref{lem_mis_K} of Section \ref{sec_remark_K}.

	%	Finally, we corroborate our theoretical results via extensive simulations in Section \ref{sec_simulation}. Our proposed approach has best performance among competitive algorithms in all settings.

	\subsection{Related literature}\label{sec_related_work}
	
	This work is most related to the literature on surrogate variable analysis (SVA) in which the parameter of interest is $\Ps$. Though our target $\Ttheta =\Ps  P^{\perp}_{\B}$ is different from $\Ps$, our analysis of $\Ttheta$ is applicable to   $\Ps$ under the condition $\Ps P_{\B} = 0$ (or $\Ps P_{\B} \rightarrow 0$). We thus  compare our results under $\Ps P_{\B} = 0$ with the existing SVA literature. For the identifiability of $\Ps$, \cite{Gagnon2012,wang2017} assumed that there exists a {\em known }subset $J \in \{1,\ldots, m\}$ such that the $p\times |J|$ submatrix $\Ps _J=0$. This set $J$ is known as ``negative control'' in the microarray studies. However, this side information is usually unknown in other settings. 
	Another approach by \cite{wang2017, McKennan19} assumes that each row $\Ps _{j\cdot}\in\RR^m$ is sparse with $\|\Ps _{j\cdot}\|_0\leq (m-a)/2$ for some $a>K$ and any $K\times a$ submatrix of $B^*$ is of rank $K$. This assumption rules out the possibility that $B^*$ could be sparse and the resulting sparsity pattern of $\Ps $ differs from (\ref{Psi_space}), considered in this work. \cite{Lee2017} assumed the condition  $\Ps  P_{\B} \to 0$ as $m\to \i$. %However, when the error is heteroscedastic, \cite{Lee2017} implicitly required $m\rightarrow\infty$ to show the asymptotic identifiability of $\Ps$, see our Remark \ref{rem_robust_PCA} for more explanations. 
	In contrast, our identifiability result of $\Ttheta$ (which is $\Ps$ under $\Ps P_{\B} = 0$) holds for any finite $n, p, m$ and $K$. To show the estimation consistency, all existing SVA methods require that $m$ grows with $n$ and is typically much larger than $n$, meanwhile $p$ is fixed and small, whereas our method provides a more general theoretical framework in which both $p$ and $m$ are allowed, but not required, to grow with $n$.

	\cite{chandrasekaran2012latent} studied the estimation of Gaussian graphical models with latent variables. In their setting, one can rewrite their estimand as the sum of a low-rank matrix and a sparse matrix (see \cite{Hsu2011, Candes} for other related examples). The regularized maximum likelihood approach is proposed with a combination of the lasso penalty and the nuclear norm penalty. Our problem is related to theirs, because model (\ref{model}) is a regression problem where the coefficient matrix has an additive decomposition of a sparse and a low-rank matrix when $K$ is much smaller than $p$ and $m$. However, our work differs significantly from this strand of research in the following aspects. First, the parameter of interest is $\Ttheta$, and thus we do not need the identifiability assumptions in \cite{chandrasekaran2012latent}. To see this, consider a simple example based on the regression model (\ref{model}) with $p=m$ and $K=1$. Let $A^*=e_i$ and $B^*=e_i^T$, where $e_i$ is the $i$th canonical basis vector of $\RR^p$. The identifiability assumption in  \cite{chandrasekaran2012latent} does not hold,  because the low rank matrix $L^*=A^*B^*=e_ie_i^T$ is too sparse and cannot be distinguished from the sparse matrix $\Ps$. However, our target $\Ttheta \in \{\Ps(\bI_p-e_ie_i^T): \Ps\in\RR^{p\times p}\}$ is still identifiable when the error is homoscedastic. One explanation is that the covariance structure of $\eps = (\B)^TW + E$ from model (\ref{model}) can assist the identification of $\Ttheta$, whereas this information is ignored if one directly applies the approach in  \cite{chandrasekaran2012latent}.
	Second, due to the differences of identifiability, our estimation algorithm (HIVE or H-HIVE) is fundamentally different from their regularized maximum-likelihood approach. In particular, our regularized estimation in the first step of our algorithm combines the group-lasso penalty and the ridge penalty. We provide a technical comparison of the ridge penalty and the nuclear norm penalty in Appendix \ref{sec_comp_rr}.

	Recently, \cite{diaz} applied SVA to estimate the causal effect under the structural equation models with hidden variables. As discussed previously, the structural equation models assume (\ref{modelZ}), which is not needed in our modeling framework. Our model (\ref{model}) is derived without imposing any specific model between $X$ and $Z$. For instance, we allow the true dependence structure between $X$ and $Z$ to be very complicated and highly nonlinear. %However, if there exists side information implying that $X$ is the cause of $Z$ such that the structural equation model (\ref{modelZ}) holds, we then have $A^*=D$. 
	The estimation method of \cite{diaz} is adapted from the SVA literature, and therefore has the same drawback as SVA. In another recent paper, \cite{cevid2018spectral} proposed a new spectral deconfounding approach to deal with high-dimensional linear regression with hidden confounding variables. In particular, their model can be written as a perturbed linear regression $Y=X^T(\beta+b)+\epsilon$ where $\epsilon \in\RR$ is a random noise, $\beta\in\RR^p$ is an unknown sparse vector and $b\in\RR^p$ is a small perturbation vector. In order to identify $\beta$, they assumed that $\|b\|_2$ is sufficiently small. Their estimation method generalizes the lava estimator \cite{chernozhukov2017}. Unlike this work, we consider a different setting where the response $Y$ is multivariate and, consequently, both our identifiability and estimation procedures (HIVE and H-HIVE) are completely different from theirs. Our theoretical results in Corollary \ref{cor_rate} and its subsequent Remark \ref{rem_rate_Theta} imply that the convergence rate of our estimator benefits substantially from the multivariate nature of the response, which can be viewed as the blessing of dimensionality.

	\subsection{Outline} In Section \ref{sec_ident_est}, we study the identifiability and estimation of $\Ttheta$ under homoscedastic error. Sufficient and necessary conditions for the identifiability of $\Ttheta$ are established in Section \ref{sec_ident}. Section \ref{sec_estimation} contains three steps of our estimation procedure.   The estimation of $\Ps + L^*$ in model (\ref{model}) is stated in Section \ref{sec_est_fit} and the estimation of the row space of $B^*$ is discussed in Section \ref{sec_est_P}. The final step of estimating $\Ttheta$ is stated in Section \ref{sec_est_theta}. Section \ref{sec_theory_fit} is dedicated to the deviation bounds of the in-sample prediction error. The estimation errors of our estimator of $\Ttheta$ together with the errors for estimating the row space of $B^*$ are given in Section \ref{sec_theory_theta}. The extension to heteroscedastic case is studied in Section \ref{sec_hetero}. In Section \ref{sec_practical}, we discuss several practical considerations, including the selection of $K$, the consequence of overestimating and underestimating $K$, the choice of tuning parameters, data standardization and practical usage of $\Ttheta$ for inferring $\Ps$. Simulation results and real data applications are presented in Sections \ref{sec_simulation} and \ref{sec_data}. All proofs and supplementary simulation results are deferred to \ref{supp}.

	\subsection{Notation}
	For any set $S$, we write $|S|$ for its cardinality. 
	For any vector $v\in \RR^d$ and some real number $q\ge 0$, we define its $\ell_q$ norm as $\|v\|_q = (\sum_{j=1}^d |v_j|^q)^{1/q}$. For any matrix $M \in \RR^{d_1 \times d_2}$, $I \subseteq \{1,\ldots, d_1\}$ and $J\subseteq \{1, \ldots, d_2\}$, we write $M_{IJ}$ as the $|I| \times |J|$ submatrix of $M$ with row and column indices corresponding to $I$ and $J$, respectively. In particular, $M_{I\cdot}$ denotes the $|I|\times d_2$ submatrix and $M_J$ denotes the $d_1\times |J|$ submatrix. Further write $\|M\|_{\ell_p / \ell_q} = (\sum_{j=1}^{d_1} \|M_{j\cdot}\|_{\ell_q}^p)^{1/p}$ and denote by $\|M\|_{\ell_0}$, $\|M\|_{op}$, $\|M\|_F$ and $\|M\|_\infty$, respectively, the element-wise $\ell_0$ norm, the operator norm, the Frobenius norm and the element-wise sup-norm of $M$. For any symmetric matrix $M$, we write $\lambda_{k}(M)$ for its $k$th largest eigenvalue. For any two sequences $a_n$ and $b_n$, we write $a_n \lesssim b_n$  if there exists some positive constant $C$ such that $a_n \le Cb_n$. Both $a_n \asymp b_n$ and $a_n = \Omega(b_n)$ stand for $a_n = O(b_n)$ and $b_n = O(a_n)$. Denote $a\vee b=\max (a,b)$ and $a\wedge b=\min(a,b)$. Throughout the paper, we will write $\wh\Sigma = n^{-1}\X^T\X$ with non-zero eigenvalues $\sigma_1 \ge \sigma_2 \ge \cdots\ge \sigma_q$ and $q:=\rank(\X)$.

	\section{Identifiability and estimation under homoscedastic noise}\label{sec_ident_est}
	
	As seen in the Introduction, $\Ps $ is not  identifiable
	%The identifiability of $\Ttheta$ 
	under model (\ref{model_1}) 
	due to the presence of hidden variables. The following proposition formally shows this.
	
	\begin{prop}\label{prop_ident_nece}
		Under model (\ref{model_1}), or equivalently (\ref{model}), if $Z\in \RR^K$ has continuous support and $A^*\ne 0$, then $\Ps$ is not identifiable.
	\end{prop}
	%\begin{proof}
	%	The proof is deferred to Appendix \ref{sec_proof_ident_nece}.
	%\end{proof}
	
	Despite of the non-identifiability of $\Ps$, we proceed to show that $\Ttheta= \Ps P_{\B}^{\perp}$ is identifiable when the error $E$ is homoscedastic. Our identifiability procedure is constructive and leads to a computationally efficient estimation algorithm for $\Ttheta$. 
	
	\subsection{Identifiability of $\Ttheta$}\label{sec_ident}
	
	We start by describing our procedure of identifying $\Ttheta$  from model (\ref{model})
	in  three steps:
	\begin{enumerate}[noitemsep]
		\item[(1)] identify the coefficient matrix 
		$$
		F^* :=\Ps +A^*B^*= \Ps + L^*;
		$$
		\item[(2)] identify $\Sigma_{\eps}:= \Cov(\eps)$ with $\eps = (\B)^TW+E$ and use it to construct $P_{\B}$, the projection matrix onto the row space of $B^*$; 
		\item[(3)] identify $\Ttheta$ from $(\bI_m - P_{\B}) Y$. 
	\end{enumerate}
	
	Recall that $W=Z-(A^*)^TX$ is independent of $E$ with $A^*$ defined in (\ref{def_A}). In step (2), a key observation from model (\ref{model}) is that, under homoscedastic error, the covariance matrix of $\eps = (\B)^T W+E$ satisfies
	\begin{align}\label{eq_Sigma_eps}
	\Sigma_{\eps} = (B^*)^T\Sigma_W B^* + \Sigma_E= (B^*)^T\Sigma_W B^* + \tau^2\bI_m,
	\end{align}
	where $\Sigma_W=\Cov(W)$. Recall that $\rank(\B)= K<m$. Provided that $\Sigma_W$ has full rank, (\ref{eq_Sigma_eps}) implies that the row space of $B^*$  coincides with the space spanned by the eigenvectors of $\Sigma_{\eps}$ corresponding to its largest  $K$ eigenvalues. We thus propose to identify $P_{\B}$ via the eigenspace of $\Sigma_{\eps}$.  Step (3) uses
	\begin{align}\label{model_Y_comp}\nonumber
	P_{\B}^{\perp}Y & ~ = ~ (\Ps  P_{\B}^{\perp})^T X +(B^*P_{\B}^{\perp})^TZ+ P_{\B}^{\perp}E\\
	& ~ = ~ (\Ttheta)^T X + P_{\B}^{\perp}E,
	\end{align}
	where $P^{\perp}_{\B} = \bI_m - P_{\B}$. The above identity further implies
	$$\Ttheta = \bigl[\Cov(X)\bigr]^{-1}\Cov\left(X, P_{\B}^{\perp}Y\right).$$ 
	The identifiability of $\Ttheta$ under the homoscedastic error is summarized in the following proposition. % Recall that $\Sigma_W:=\Cov(W)$. 

	\begin{prop}\label{prop_ident}
		Under model (\ref{model}), $\Ttheta$ is identifiable if either of the following holds:
		\begin{enumerate}
			\item[(1)]  $\Ps P_{\B} + A^*\B = 0$;
			\item[(2)] $\rank(\Sigma_W) = K$ and $\Sigma_E = \tau^2\bI_m$.
		\end{enumerate}
	\end{prop}
	Case {\it (1)} implies that $\Ps P_{\B}$, the direct effect of $X$ on $Y$ explained by $Z$, can be exactly offset by the indirect effect $A^*\B$. In this case, $\Ttheta$ can be recovered directly by regressing $Y$ onto $X$. Since this is rarely the case in practice, we will focus on $\Ps P_{\B} + A^*\B \ne 0$. Case {\it (2)} requires $\rank(\Sigma_W) = K$ in addition to $\Sigma_E = \tau^2\bI_m$. We show, in Proposition \ref{prop_ident_W} below, that $\rank(\Sigma_W) = K$ is also  necessary for identifying $\Ttheta$ if $\EE[W|X] = 0$ and $\Ps P_{\B} + A^*\B \ne 0$. 
	
	%\begin{prop}\label{prop_ident}
	%	Under model (\ref{model}), assume $\rank(\Sigma_W) = K$ and $\Sigma_E = \tau^2\bI_m$. Then $\Ttheta$ is uniquely determined by $\Cov(X), Cov(Y)$ and $\Cov(X,Y).$
	% \end{prop}
	%	\begin{proof}
	%		The proof is deferred to Appendix \ref{sec_proof_prop_ident}.
	%	\end{proof}	
	%	Proposition \ref{prop_ident} requires $\rank(\Sigma_W) = K$. We show, in Proposition \ref{prop_ident_W} below, that $\rank(\Sigma_W) = K$ is also  necessary for identifying $\Ttheta$ if $\EE[W|X] = 0$ holds. 

	\begin{prop}\label{prop_ident_W}
		Under model (\ref{model}) with $\EE[W|X] = 0$, assume $\Ps P_{\B} + A^*\B \ne 0$. If $\rank(\Sigma_W) < K$, then $\Ttheta$ is not identifiable. 
	\end{prop}
	
	Combining Propositions \ref{prop_ident} and \ref{prop_ident_W} concludes  that, under homoscedastic error, $\Ps P_{\B} + A^*\B \ne 0$ and $\EE[W|X] = 0$, $\Ttheta$ is identifiable if and only if  $\rank(\Sigma_W) = K$. 
	The condition $\EE[W|X] = 0$ is satisfied in many interesting scenarios, such as the structured equation model (\ref{modelZ}) and the multivariate Gaussian model for $(Z, X)$.  
	In practice, recalling that $W = Z - A^{*T}X$, $\rank(\Sigma_W) = K$ is a reasonable assumption as the hidden variable $Z$ usually contains information that cannot be perfectly explained by a linear combination of the observable feature $X$.  Therefore, throughout the paper, we assume $\rank(\Sigma_W) = K$.

	\begin{remark}\label{rem_ident_Psi}
		In the SVA literature, \cite{Lee2017} assumed $\Ps P_{\B} \to 0$ as $m = m(n)\to \i$. 
		Under this condition, we obtain $\Ttheta \approx \Ps$ for sufficiently large $m$. In this case, Propositions \ref{prop_ident} and \ref{prop_ident_W} provide sufficient and necessary conditions for the identifiability of $\Ps$ for $m$ large enough. Therefore, our analysis provides complete identifiability results for SVA and it further generalizes to the setting that $\Ps P_{\B} \not\to 0$.
		
		%Recall that in the SVA literature, \cite{Lee2017} assumed $\Ps P_{\B} = 0$ for the estimation of $\Ps$. Under this condition, we obtain $\Ttheta = \Ps$. In this case, Propositions \ref{prop_ident} and \ref{prop_ident_W} provide sufficient and necessary conditions for the identifiability of $\Ps$. Therefore, our analysis provides complete identifiability results for SVA and it further generalizes to the setting that $\Ps P_{\B} \neq 0$. 
	\end{remark}	
	
	\begin{remark}
		Our identifiability results are established when the rows of $\bm {X}\in\RR^{n\times p}$ are viewed as i.i.d. realizations of a random vector $X\in \RR^p$. When ${\bm X}$ is treated as a fixed design matrix, Proposition 1 in \cite{McKennan19} provides sufficient conditions on $\Ps$ and other quantities (such as the sparsity of rows of $\Ps$ and the magnitude of $\B$) under which $\Ps$ becomes identifiable for sufficiently large $n$. 
	\end{remark}

	\subsection{Estimation of $\Ttheta$}\label{sec_estimation}
	Given the data matrices $\X \in \RR^{n\times p}$ and $\Y \in \RR^{n\times m}$, our estimation procedure follows the same steps in the analysis of the model identifiability: (1) first estimate $\X F^*$; (2) then estimate $\Sigma_{\eps}$ and $P_{\B}$; (3) finally estimate $\Ttheta$.     
	
	\subsubsection{Estimation of $X F^*$}\label{sec_est_fit}
	Recall that $F^*=\Ps + L^*$ is identifiable, where $L^*=A^*B^*$ is a dense matrix and $\Ps$ is a row-wise sparse matrix satisfying (\ref{Psi_space}). We propose to estimate $F^*$ by $\wh F = \wh \Psi + \wh L$ where  $\wh \Psi$ and $\wh L$ are obtained by solving the following optimization problem 
	\begin{align}\label{est_F}
	(\wh \Psi,~\wh L) = \arg\min_{\Psi, L} {1\over n}\left\|\Y - \X(\Psi + L)\right\|_F^2 +\lambda_1 \|\Psi\|_{\l12} +  \lambda_2\|L\|_F^2
	\end{align}
	with some tuning parameters $\lambda_1, \lambda_2 \ge 0$. Our estimator is designed to recover both the sparse matrix $\Ps$ via the group-lasso regularization \citep{yuanlin} and the dense matrix $L^*$ via the multivariate ridge regularization. %Since our goal in this step is to estimate the best linear predictor $\X F^*$, there is no need to separate $\Ps$ from $L^*$. 
	Computationally, solving (\ref{est_F}) is efficient with almost the same complexity of solving a group-lasso problem. Specifically, we have the following lemma.
	
	%Computationally, (\ref{est_F}) has a unique solution  and solving (\ref{est_F}) is computationally efficient with almost the same complexity of solving a group-lasso problem. Specifically, we have the following lemma.
	\begin{lemma}\label{lem_solution}
		Let $(\wh \Psi, \wh L)$ be any solution of (\ref{est_F}), and denote
		\begin{equation}\label{def_P_Q_lbd2}
		P_{\lambda_2} = \X\left(\X^T\X + n \lambda_2 \bI_p\right)^{-1}\X^T,\qquad Q_{\lambda_2} = \bI_n - P_{\lambda_2}
		\end{equation}
		for any $\lambda_2\ge 0$ such that $\P$ exists. 
		Then $\wh \Psi$ is the solution of the following problem
		\begin{equation}\label{crit_Theta}
		\wh \Psi = \arg\min_{\Psi} {1\over n}\left\|Q_{\lambda_2}^{1/2}(\Y - \X\Psi)
		\right\|_F^2 + \lambda_1 \|\Psi\|_{\l12},
		\end{equation}
		and $\wh L=(\X^T\X + n \lambda_2 \bI_p)^{-1}\X^T(\Y - \X\wh\Psi)$, where $Q_{\lambda_2}^{1/2}$ is the principal matrix square root of $Q_{\lambda_2}$. Moreover, we have
		\begin{equation}\label{fit}
		\X\wh F= \X(\wh \Psi + \wh L) = P_{\lambda_2}\Y + Q_{\lambda_2}\X\wh \Psi. 
		\end{equation}
	\end{lemma}
	Lemma \ref{lem_solution} characterizes the role of the regularization parameters $\lambda_2$ and $\lambda_1$. 
	%When $\lambda_2\to 0$, we have $\wh F\approx\wh\Psi+(\X^T\X)^{+}\X^T(\Y - \X\wh\Psi)=(\X^T\X)^{+}\X^T\Y$, where $(\X^T\X)^{+}$ is the pseudo inverse of $\X^T\X$. Thus, $\wh F$ reduces to the generalized least squares estimator. 
	When $\lambda_2\to 0$, we have $\P \approx \X (\X^T \X)^+\X^T$ and $\Q \X \approx 0$ with $(\X^T\X)^{+}$ being the Moore-Penrose inverse of $\X^T\X$. Thus, $\X\wh F\approx \X(\X^T\X)^{+}\X^T\Y$ and $\wh F$ reduces to the minimum norm ordinary least squares estimator (see, for instance, \cite{bunea2020interpolation}).
	On the other hand, when $\lambda_2\to \i$, we have $Q_{\lambda_2} \approx \bI_n$ and $\wh L\approx 0$ whence $\wh F\approx \wh \Psi$ essentially becomes the group-lasso estimator. Later in Remark \ref{rem_orth_design} of Section \ref{sec_theory_fit}, we will take a closer look at this phenomenon in terms of the convergence rates of $\|\X\wh F - \X F^*\|_F$ under the orthonormal design. The tuning parameter $\lambda_1$ only appears in (\ref{crit_Theta}) and its magnitude controls the sparsity level of the group-lasso estimator $\wh \Psi$. 
	Lemma \ref{lem_solution} also implies that the estimator $(\wh \Psi, \wh L)$ is unique if and only if the solution of the group-lasso problem (\ref{crit_Theta}) is unique. Even if (\ref{crit_Theta}) has multiple solutions, we can define $(\wh \Psi, \wh L)$ to be any of the solutions and the resulting best linear predictor $\X\wh F=\X(\wh \Psi+\wh L)$ satisfies the desired deviation bounds stated in Theorem \ref{thm_pred} of Section \ref{sec_theory_fit}.

	In applications when both $m$ and $p$ are large while $K$ is small, $L^*$ can be also viewed as a low-rank matrix with rank $K$. One common approach of estimating a low-rank matrix is to either impose a rank constraint on the matrix, known as the reduced-rank approach \citep{i08}, or regularize its nuclear norm. %The latter is known as the convex relaxation of the reduced-rank approach. 
	We emphasize that, under model (\ref{model}), our approach with the ridge penalty has both theoretical and computational advantages over these two methods. We defer to Appendix \ref{sec_comp_rr} for both theoretical and numerical comparisons.

	Finally, we comment that our method (\ref{est_F}) can be viewed as the multivariate generalization of the lava approach proposed by \cite{chernozhukov2017}; see also \cite{cevid2018spectral}. Lava estimates the sum of a sparse vector $\beta$ and a dense vector $b$ in linear regression problem $\y=\X(\beta+b)+{\bm \epsilon}$ by minimizing the least squares loss plus the penalty $\lambda_1\|\beta\|_1+\lambda_2\|b\|_2^2$. 
	%It is shown that lava has better prediction performance than many other estimators such as ridge, lasso, ordinary least squares and elastic net. 
	As explained in \cite{chernozhukov2017},  lava
	is intrinsically different from the elastic net as lava penalizes both $\beta$ and $b$ and the estimate of $(\beta+b)$ is non-sparse, whereas elastic net uses the penalty $\lambda_1\|\beta\|_1+\lambda_2\|\beta\|_2^2$ and typically yields a sparse estimate of $\beta$. These differences naturally extend to our multivariate setting. 
	
	%The closed-form of $\X \wh F$ in (\ref{fit}) is also critical  to derive theoretical guarantees for the fit $\|\X \wh  F - \X F^*\|_F^2$. For any pair $(\Theta_0, L_0)$ such that $\Theta_0 + L_0 = F^*$, using (\ref{fit}) yields the following decomposition of the fit 
	%\begin{align}\label{error_decomp}\nonumber
	%\X\wh F - \X F^* &= P_{\lambda_2}\Y + Q_{\lambda_2}\X\wh \Theta - \X \Theta_0 - \X L_0\\
	%&= \underbrace{P_{\lambda_2}(\Y- \X\Theta_0) - \X L_0}_{Error(ridge)}+ \underbrace{Q_{\lambda_2}(\X\wh \Theta - \X \Theta_0)}_{Error(lasso)},
	%\end{align}
	%where $Error(ridge)$ is the estimation error of the mutivariate ridge regression while $Error(lasso)$ is that of the group-lasso in (\ref{crit_Theta}). 

	\subsubsection{Estimation of $P_{\B}$}\label{sec_est_P}
	
	In this section, we discuss how to estimate the projection matrix $P_{\B}$. Consider the singular value decomposition $B^*=VDU^T$, where $V\in\RR^{K\times K}$ and $U\in\RR^{m\times K}$ are, respectively, the left and right singular vectors of $B^*$, and $D$ is the diagonal matrix of non-increasing singular values. It is easily seen that $P_{\B}=UU^T$. Recall that, from  (\ref{eq_Sigma_eps}), $U$ also coincides with the first $K$ eigenvectors of $\Sigma_{\eps}$ up to an orthogonal matrix. We thus propose to  first estimate $\Sigma_{\eps}$  by
	\begin{equation}\label{def_Sigma_eps_hat}
	\wh \Sigma_{\eps} = {1\over n}\left(\Y - \X \wh F\right)^T\left(\Y - \X \wh F\right),
	\end{equation}
	with $\wh F$ obtained from (\ref{est_F}), and then estimate $P_{\B}$ by $\wh P_{\B}=\wh U\wh U^T$, where $\wh U$ consists of the eigenvectors of $\wh\Sigma_{\eps}$ corresponding to the $K$ largest eigenvalues.  We assume $K$ is known for now and defer to Section \ref{sec_select_K} for detailed discussions of selecting $K$.

	\subsubsection{Estimation of $\Ttheta$}\label{sec_est_theta}
	After estimating $P_{\B}$ by $\wh P_{\B}$, motivated by (\ref{model_Y_comp}),  we propose to estimate $\Ttheta$ by
	\begin{equation}\label{est_Theta}
	\wt\Theta = \arg\min_{\Theta} {1\over n}\left\|\Y \left(\bI_m - \wh P_{\B}\right)-  \X\Theta\right\|_F^2 + \lambda_3 \|\Theta\|_{\l12}
	\end{equation}
	with some tuning parameter $\lambda_3>0$. Solving the problem in (\ref{est_Theta}) is equivalent to solving a group-lasso problem with the projected response matrix $\Y (\bI_m - \wh P_{\B})$.

	For the reader's convenience, we summarize our procedure, \underline{HI}dden  {\underline V}ariable  adjustment  {\underline E}stimation (HIVE), in Algorithm \ref{alg_1}.
	
	{\begin{algorithm}[ht]
			\caption{The HIVE procedure for estimating $\Ttheta$.}\label{alg_1}
			\begin{algorithmic}[1]
				\Require Data $\X\in \RR^{n\times p}$ and $\Y \in \RR^{n\times m}$, rank $K$, tuning parameters $\lambda_1$, $\lambda_2$ and $\lambda_3$. 
				\State Estimate $\X\wh F$ with $\wh F = \wh\Psi + \wh L$ by solving (\ref{est_F}).
				\State Obtain $\wh \Sigma_{\eps}$ from (\ref{def_Sigma_eps_hat}). 
				\State Compute $\wh P_{\B} = \wh U\wh U^T$ where $\wh U$ are the first $K$ eigenvectors of $\wh \Sigma_{\eps}$.
				\State Estimate $\Ttheta$ by $\wt \Theta$ obtained from (\ref{est_Theta}).
			\end{algorithmic}
	\end{algorithm} }

	\section{Statistical guarantees}\label{sec_theory}
	In this section, we provide theoretical guarantees for our estimation procedure. 
	In our theoretical analysis, the design matrix $\X$ is considered to be deterministic and the analysis can be done similarly for random designs by first conditioning on $\X$. 
	Recall from model (\ref{model}) that $W$ is uncorrelated with $X$. To simplify the analysis under the fixed design scenario, we assume the independence between $X$ and $W$ in order to derive the deviation bounds of their cross product. We expect that the same theoretical guarantees hold under $\Cov(X,W)=0$ by using more tedious arguments. We start with the following assumptions on the error matrices $\W\in\RR^{n\times K}$ and $\E\in\RR^{n\times m}$. 
	\begin{ass}\label{ass_error}
		Let $\g_w$ and $\g_e$ denote some positive constants.
		
		(1) Assume $\bigl\{\Sigma_W^{-1/2}\W_{i\cdot}\bigr\}_{i=1}^n$ are  i.i.d. $\g_w$ sub-Gaussian random vectors\footnote{A random vector $X$ is $\g$ sub-Gaussian if $\langle u, X\rangle$ is $\g$ sub-Gaussian for any $\|u\|_2=1$.}, where $\Sigma_W=\Cov(\W_{i\cdot})$.
		
		(2) For any fixed $1\le j\le m$, $\bigl\{\E_{ij}\bigr\}_{i=1}^n$ are i.i.d. $\g_e$ sub-Gaussian\footnote{A centered random variable $X$ is $\g$ sub-Gaussian if it satisfies 
			$
			\EE[\exp(tX)] \le \exp(\g^2t^2/2)
			$ for all $t\ge 0$.}. For any fixed $1\le i\le n$, $\bigl\{\E_{ij}\bigr\}_{j=1}^m$ are independent.
	\end{ass}
	Since part $(2)$ of Assumption \ref{ass_error} does not assume $\E_{ij}$ are identically distributed across $1\le j\le m$, this assumption is applicable to both homoscedastic and heteroscedastic errors, provided that $\max_{1\le j\le p}\textrm{Var}(\E_{ij}) \le \g_e^2$. 
	We assume $\Sigma_E = \Cov(E) = \tau^2\bI_m$ throughout this section and defer the dicussion of the heteroscedastic case to Section \ref{sec_hetero}. Finally, recall from (\ref{Psi_space}) that $\|\Ps\|_{\ell_0 / \ell_2} \le s_*$.

	\subsection{Statistical guarantees of estimating $XF^*$}\label{sec_theory_fit}
	To establish theoretical properties for $\X\wh F$ obtained from (\ref{crit_Theta}), we first generalize the design impact factor of $\X$ in \cite{chernozhukov2017} to multivariate response regression settings. Denote $\wt \X = Q_{\lambda_2}^{1/2}\X$, where $\Q$ is defined in (\ref{def_P_Q_lbd2}). For notational simplicity, we suppress the dependence of $\wt \X$ on $\lambda_2$. For any constant $c>0$ and matrix $\Psi_0 \in \RR^{p\times m}$, define the design impact factor as
	\begin{equation}\label{RE_pred}
	\kappa_1(c, \Psi_0, \lambda_1, \lambda_2) := \inf_{\Delta \in \R(c, \Psi_0, \lambda_1, \lambda_2)}{\|\wt \X\Delta\|_F/\sqrt{n} \over 
		\|\Psi_0\|_{\l12} - \|\Psi_0+\Delta\|_{\l12} + c\|\Delta\|_{\l12}},
	\end{equation}
	where 
	\begin{align}\label{def_R}
	& ~~\R(c, \Psi_0, \lambda_1, \lambda_2)  = \Big\{
	\Delta \in \RR^{p\times m}\setminus \{0\}: \\\nonumber
	&\hspace{1cm}{\|\wt \X \Delta\|_F/ \sqrt n}\le 2\lambda_1\left(	\|\Psi_0\|_{\l12} - \|\Psi_0+\Delta\|_{\l12} + c\|\Delta\|_{\l12}\right)
	\Big\}.
	\end{align}
	
	It is well known that when $p>n$ the matrix $\wt \X^T\wt \X$ is singular and the least squares loss is not strictly convex. The design impact factor $\kappa_1(c, \Psi_0, \lambda_1, \lambda_2)$ is used to characterize the minimum curvature of the least squares loss in (\ref{crit_Theta}) when the matrix $\Delta$ is restricted in a feasible set $\R(c, \Psi_0, \lambda_1, \lambda_2)$. It generalizes the widely used Restricted Eigenvalue (RE) condition in high-dimensional regression \citep{bickel2009} and is more suitable for prediction \citep{belloni2014, chernozhukov2017}.  We refer to Remark \ref{rem_design_impact_factor} for its connection with the RE condition.
	
	Define the following quantity which characterizes the total variation of the multivariate response regression in (\ref{model}), \begin{equation}\label{def_V_eps}
	V_\eps = \tr(\Gamma_\eps),\qquad \text{with}\quad \Gamma_{\eps} := \g_w^2\  B^{*T}\Sigma_W \B + \g_e^2\ \bI_m,
	\end{equation}
	where $\tr(\cdot)$ stands for the trace. 
	Let $r_e(\Gamma_\eps) = \tr(\Gamma_{\eps}) /\|\Gamma_{\eps}\|_{op}$ denote the \emph{effective rank} of $\Gamma_\eps$. Write $M = n^{-1}\X^T Q_{\lambda_2}^2\X$ and $\wh\Sigma = n^{-1}\X^T\X$. Recall that  $\P$ and $\Q$ are defined in (\ref{def_P_Q_lbd2}). The following theorem provides the deviation bounds of $\|\X\wh F - \X F^*\|_F$.  
	
	\begin{thm}\label{thm_pred}
		Under model (\ref{model}) and Assumption \ref{ass_error},  choose 
		\begin{equation}\label{rate_lbd1}
		\lambda_1 = 4\sqrt{\max_{1\le j\le p}M_{jj}}\left(1+\sqrt{2\log(p / \epsilon') \over r_e(\Gamma_{\eps})}\right)\sqrt{V_\eps \over n}
		\end{equation}
		for any $\epsilon'>0$ and choose any $\lambda_2 \ge 0$
		in (\ref{est_F}) such that $\P$ exists. 
		With probability $1-\epsilon - \epsilon'$, 
		\begin{align*}
		{1\over n}\left\|\X \wh F - \X F^*\right\|_F^2 
		&\le 
		\!\!\! \inf_{\substack{(\Psi_0, L_0):\\
				\Psi_0 + L_0 = F^*}}  \!\!\! \left[
		{2\over n}\left\|\X(\wh L - L_0)\right\|_F^2 + {2\over n}\left\|Q_{\lambda_2} \X(\wh \Psi - \Psi_0)\right\|_F^2\right]\\
		&\le \!\!\! \inf_{\substack{(\Psi_0, L_0):\\
				\Psi_0 + L_0 = F^*}}  \!\!\! \Bigl[4Rem_1 + 36 \|Q_{\lambda_2}\|_{op} Rem_2(L_0) + 8 \|Q_{\lambda_2}\|_{op} Rem_3(\Psi_0)\Bigr],
		\end{align*}
		where $\|\Q\|_{op}\le1$ and 
		\begin{align*}
		&Rem_1 = \left(\sqrt{{\rm tr}(P_{\lambda_2}^2)} + \sqrt{2\log (m/\epsilon)\|P_{\lambda_2}^2\|_{op}} \right)^2{V_\eps \over n},\\
		&Rem_2(L_0) = \lambda_2~  {\rm tr} \left[L_0^T \wh \Sigma(\wh \Sigma + \lambda_2 \bI_p)^{-1}L_0\right],\\
		&Rem_3(\Psi_0) = \lambda_1^2~  \left[\kappa_1(1/2, \Psi_0, \lambda_1, \lambda_2)\right]^{-2}.
		\end{align*}
	\end{thm}
	%\begin{proof}
	%	The proof is deferred to Appendix \ref{sec_proof_thm_pred}.
	%\end{proof}
	
	Since $\Ps$ is not identifiable,  neither  $\Ps$ nor $L^*$ can be identified individually. Nevertheless, our estimator $\wh F$ minimizes the error over all possible combinations of $\Psi_0$ and $L_0$ satisfying $\Psi_0 + L_0 = F^*$. %In particular, it holds for $(\Ps, L^*)$ whenever they are identifiable. 
	As expected from (\ref{est_F}), the prediction error in Theorem \ref{thm_pred} comes from two sources:  estimating $L^*$ from the multivariate ridge regression and estimating $\Ps$ from the group-lasso. Specifically, $Rem_1$ and $Rem_2(L_0)$ are, respectively, the variance and bias terms from the ridge regression while $Rem_3(\Psi_0)$ corresponds to the estimation error of the group-lasso. In the following, we provide more insights on these three terms.

	\begin{remark}[Design impact factor and $\lambda_1$]\label{rem_design_impact_factor}
		The remainder term $Rem_3(\Psi_0)$ depends on the design impact factor  $\kappa_1(c, \Psi_0, \lambda_1, \lambda_2)$ and the tuning parameter $\lambda_1$. 
		We first discuss the connection of $\kappa_1(c, \Psi_0, \lambda_1, \lambda_2)$ and the Restricted Eigenvalue (RE) condition of $\X$ defined as 
		\begin{equation}\label{RE_X}
		\kappa(s, \alpha) = \min_{S\subseteq \{1, 2, \ldots, p\},|S|\le s}~\min_{\Delta \in \C(S, \alpha)}{\|\X \Delta\|_F \over \sqrt{n}\|\Delta_{S\cdot}\|_F},
		\end{equation}
		where $\alpha \ge 1$ is a constant, $1\le s\le p$ is some integer and $\C(S, \alpha) := \{\Delta\in \RR^{p\times m} \setminus \{0\}: \alpha\|\Delta_{S\cdot}\|_{\ell_1/\ell_2}\ge \|\Delta_{S^c \cdot}\|_{\ell_1/\ell_2}\}$. Similarly, denote by $\wt \kappa(s, \alpha)$ the RE condition of $\wt \X = \Q^{1/2}\X$.
		In Lemma \ref{lem_RE} of Appendix \ref{sec_proof_auxiliary}, we show that, for any constant $c\in (0,1)$, 
		\begin{equation}\label{kappa1_kappa}
		[\kappa_1(c, \Psi_0, \lambda_1, \lambda_2)]^2 \ge{[\wt\kappa(s_0, \alpha_{c})]^2 \over  (1+c)^2 s_0}  \ge
		{\lambda_2 \over  \sigma_1+ \lambda_2}\cdot {[\kappa(s_0, \alpha_{c})]^2\over (1+c)^2 s_0},
		\end{equation} 		
		where $\alpha_{c} = (1+c)/(1-c)$, $s_0 = \|\Psi_0\|_{\ell_0/\ell_2}$ and $\sigma_1$ is the leading eigenvalue of $\wh\Sigma = n^{-1}\X^T\X$.
		The first inequality is proved in \cite{chernozhukov2017} for $m = 1$. Here we extend it to $m\ge 2$, and we further establish the second inequality which characterizes the relation between $\kappa_1(c, \Psi_0, \lambda_1, \lambda_2)$ and $\kappa(s_0, \alpha_c)$. It is well known that  $\kappa(s_0, \alpha_c)$ is lower bounded by a positive constant with high probability when the rows of $\X$ are i.i.d. sub-Gaussian vectors with $\lambda_{\min}(\Sigma)>c'$ for some constant $c'>0$ and $s_0 = O(n)$ \citep{rz13}. Together with the second inequality in (\ref{kappa1_kappa}), we obtain $[\kappa_1(c, \Psi_0, \lambda_1, \lambda_2)]^2\gtrsim \lambda_2/[s_0(\sigma_1+ \lambda_2)]$. Thus, when $\lambda_2$ dominates $\sigma_1$, $\kappa_1(c, \Psi_0, \lambda_1, \lambda_2)$ scales as $1/\sqrt{s_0}$.
		
		%As discussed in \cite{chernozhukov2017},  $\kappa_1(c, \Theta_0, \lambda_1, \lambda_2)$  scales as $1/\sqrt{s_0}$ when $\wt \kappa(s_0, \alpha_{c})$ is lower bounded away from zero. On the other hand, even if  $\wt\kappa(s_0, \alpha_{c}) = 0$, $ \kappa_1(c, \Theta_0, \lambda_1, \lambda_2)$ could still be positive such that $[\kappa_1(c, \Theta_0, \lambda_1, \lambda_2)]^{-1}$ is well-defined, for instance, when two regressors are identical \citep{belloni2014,chernozhukov2017}. On the other hand, the connection of $\kappa_1(c, \Theta_0, \lambda_1, \lambda_2)$ and $\kappa(s_0, \alpha_c$) depends on $\lambda_2$ and is not this transparent. Only when $\lambda_2$ is relatively large comparing to $\sigma_1$, $\kappa_1(c, \Theta_0, \lambda_1, \lambda_2)$ scales as $\kappa(s_0, \alpha_{c})/\sqrt{s_0}$.
		
		Note that $Rem_3(\Psi_0)$ also depends on the tuning parameter $\lambda_1$ which is further related to the choice of $\lambda_2$ via the diagonal entries of $M = n^{-1}\X^T Q_{\lambda_2}^2\X$. To simplify $Rem_3(\Psi_0)$, in Lemma \ref{lem_RE} of Appendix \ref{sec_proof_auxiliary} we prove
		\begin{equation}\label{Mjj}
		\max_{1\le j\le p}M_{jj} \le \max_{1\le j\le p}\wh\Sigma_{jj}  \left(\lambda_2 \over \sigma_q + \lambda_2\right)^2,
		\end{equation}
		where $\sigma_q$ is the smallest non-zero eigenvalue of $\wh\Sigma$. Combining (\ref{kappa1_kappa}) and (\ref{Mjj}), we obtain that
		$$
		Rem_3(\Psi_0)\lesssim {\lambda_2(\sigma_1 + \lambda_2) \over (\sigma_q + \lambda_2)^2}\max_{1\le j\le p} \wh \Sigma_{jj}{s_0 \over  [\kappa(s_0, 3)]^2}\left(1 + {\log(p/\epsilon') \over r_e(\Gamma_\eps)}\right)\frac{V_\eps}{n}.
		$$
	\end{remark}
	\medskip
	
	The first two remainder terms $Rem_1$ and $Rem_2(L_0)$ depend on the choice of $\lambda_2$ in a more complicated way. To make the remainder terms more transparent, we can bound them from above via the eigenvalues of $\wh\Sigma$. The simplified deviation bounds of $\|\X \wh F - \X F^*\|_F^2$ are summarized in the following corollary. 
	Recall that $\sigma_1 \ge \cdots \ge \sigma_q$ denote the non-zero eigenvalues of $\wh\Sigma$ with $q = \rank(\X)$. 
	
	\begin{cor}\label{cor_pred}
		Under model (\ref{model}) and Assumption \ref{ass_error}, with probability $1-\epsilon - \epsilon'$, one has
		\begin{align*}
		{1\over n}\left\|\X \wh F - \X F^*\right\|_F^2
		&\lesssim  \!\!\! \inf_{\substack{(\Psi_0, L_0):\\
				\Psi_0 + L_0 = F^*}}  \!\!\!  \bigg\{ {\sigma_1 \lambda_2 \over (\sigma_1 + \lambda_2)(\sigma_q + \lambda_2)}~\lambda_2\|L_0\|_F^2\\
		+  & \left[
		\sum_{k=1}^q\left({\sigma_k \over \sigma_k + \lambda_2}\right)^2+
		\left( {\sigma_1 \over \sigma_1 + \lambda_2}\right)^2\log(m/\epsilon)
		\right]{V_\eps \over n}\\
		+ &
		{\sigma_1 + \lambda_2 \over \sigma_q + \lambda_2} \left({\lambda_2 \over \sigma_q + \lambda_2}\right)^2 \!\!\! \max_{1\le j\le p} \wh \Sigma_{jj} \left(1 + {\log(p/\epsilon') \over r_e(\Gamma_\eps)}\right){s_0\over  [\kappa(s_0, 3)]^2} {V_\eps  \over n}\bigg\}
		\end{align*} 
	where $\kappa(s_0, 3)$ is defined in (\ref{RE_X}) with $s_0 = \|\Psi_0\|_{\ell_0/\ell_2}$.
	\end{cor}
	
	%\begin{proof}
	%	The proof is deferred to Appendix \ref{sec_proof_cor_thm_pred}.
	%\end{proof}
	
	%We further illustrate the prediction error when the design matrix is orthonormal. 
	%One thing to note is that $\kappa(s_*,3)>0$ implies $\wt \kappa(s_*,3) > 0$ from (\ref{kappa1_kappa}) albeit the latter depends on the choice of $\lambda_2$. 
	
	\begin{remark}[Orthonormal design]\label{rem_orth_design}
		To draw connections with existing results on the group-lasso and ridge estimators, we consider the orthonormal design $\wh\Sigma = \bI_p$.  The deviation bounds in Corollary \ref{cor_pred} reduce to (after ignoring the logarithmic factors) 
		\begin{align}\label{rate_fit_orth_design}
		{1\over n}\left\|\X \wh F - \X F^*\right\|_F^2 
		&\lesssim  \left({1\over 1+\lambda_2}\right)^2{pV_{\eps} \over n} +\left(\lambda_2 \over 1+\lambda_2\right)^2\|L_0\|_F^2 + \left(\lambda_2 \over 1+\lambda_2\right)^2{s_0V_\eps \over n},
		\end{align}
		for any $(\Psi_0, L_0)$ satisfying $\Psi_0 + L_0 = F^*$. 
		The first two terms are the variance and bias due to the ridge penalty while the third term is the error of the group-lasso.  As $\lambda_2$ increases, the variance term of the ridge decreases whereas the bias term of the ridge and the error of group-lasso increase. Optimizing the right hand side of (\ref{rate_fit_orth_design}) over $\lambda_2$ yields 
		\begin{equation}\label{rate_lam2}
		\lambda_2 = {pV_{\eps} / n \over \|L_0\|_F^2 + s_0V_{\eps}/n}.
		\end{equation}
		
		 (a) When $L^* = 0$, model (\ref{model}) reduces to $Y = (\Ps)^T X+\eps$. By choosing $(\Psi_0, L_0)=(\Ps, 0)$, we have $\lambda_2 = {p / s_*}$ from (\ref{rate_lam2}) and $\max_jM_{jj} \asymp p^2/ (p +s_*)^2$ from (\ref{Mjj}). Consequently, the choice of $\lambda_1$ in (\ref{rate_lbd1}) satisfies 
		$$
		\lambda_1 \asymp \left({p \over s_*+p}\right)^2\left(1 + \sqrt{\log(p / \epsilon') \over r_e(\Gamma_{\eps})}\right) \sqrt{V_\eps \over n},$$ 
		and (\ref{rate_fit_orth_design}) reduces to 
		$$
		{1\over n}\left\|\X \wh F - \X F^* \right\|_F^2\lesssim \left(\frac{s_*}{p+s_*}\right)^2{pV_\eps \over n}+\left(\frac{p}{p+s_*}\right)^2{s_*V_\eps \over n}\lesssim {s_*V_\eps \over n},
		$$
		which is the optimal rate of the group-lasso estimator. 
		%When $s_* = o(p)$, we have $\lambda_2 \to \infty$, $\lambda_1^2 \asymp V_\eps\log (pm) / n$ and  $\wh F$ from (\ref{est_F}) essentially becomes the group-lasso estimator.
		
		 (b) When $\Psi^* = 0$,  model (\ref{model}) reduces to $Y = (L^*)^TX +\eps$. By choosing $(\Psi_0, L_0)=(0, L^*)$, we have $\lambda_2 = pV_{\eps} / (n\|L^*\|_F^2)$ from (\ref{rate_lam2}). After simple calculations, (\ref{rate_fit_orth_design}) yields
		\begin{equation}\label{rate_ridge}
		{1\over n}\left\|\X \wh F - \X F^*\right\|_F^2  \lesssim \min \left(\frac{pV_{\eps}}{n},\|L^*\|_F^2\right) \lesssim \sqrt{pV_{\eps}\over n}\|L^*\|_F,
		\end{equation}
		which is the optimal rate of the ridge regression \citep{Hsu2014}. 
		
		Combining scenarios (a) and (b), we conclude that the convergence rate (\ref{rate_fit_orth_design}) of our estimator $\wh F$ with the optimal tuning parameters $\lambda_1$ and $\lambda_2$ matches the best possible rate even if $L^*=0$ or $\Ps=0$ were known a priori. For this reason, we refer to our estimator $\wh F$ as an adaptive estimator. 
		
		When $\Ps = 0$ and $K$ is much smaller than both $p$ and $m$, model (\ref{model}) is a multivariate response regression with the coefficient matrix $L^*$ exhibiting a low-rank structure. A natural approach is to use a reduced-rank estimator to estimate $L^*$. In Appendix \ref{sec_comp_rr}, we show that our ridge-type estimator could have a faster rate than the reduced-rank estimator in our problem.
	\end{remark}

	\subsection{Statistical guarantees  of estimating $\Ttheta$}\label{sec_theory_theta}
	
	It is seen from (\ref{est_Theta}) that $\wt\Theta$ depends on $\wh P_{\B}$, the estimator of $P_{\B}$.
	In the following, we first state a general theorem which establishes non-asymptotic upper bounds of $\|\wt \Theta - \Ttheta\|_{\l12}$ for $\wt \Theta$ obtained from (\ref{est_Theta}) by using any estimator $\wh P$ of $P_{\B}$ in lieu of $\wh P_{\B}$. Let $\Lambda_1$ denote the largest eigenvalue  of $B^{*T}\Sigma_W B^*$.
	\begin{thm}\label{thm_Theta}
		Under model (\ref{model}) and Assumption \ref{ass_error}, assume $\kappa(s_*, 4) > 0$. Let  $\wt \Theta$ be any solution of problem (\ref{est_Theta}) by using estimator $\wh P\in \RR^{m\times m}$ in place of $\wh P_{\B}$.  Choose any $\lambda_3 \ge  \bar\lambda_3$ in (\ref{est_Theta}) with
		\begin{align}\label{def_event_lbd3}
		&\bar\lambda_3 =4 \g_e\sqrt{\max_{1\le j\le p} \wh\Sigma_{jj}}{\sqrt{m}+\sqrt{2\log(p/\epsilon)}\over \sqrt n}.
		\end{align}
		On the event 
		%\begin{equation}\label{def_event_P_B}
		$   \{
		\|\wh P - P_{\B}\|_F \lesssim \xi_n
		\}
		$
		%\end{equation}
		for some proper sequence $\xi_n$, 
		with probability $1-\epsilon - 2e^{-cK}$ for some constant $c>0$, one has
		\begin{equation}\label{rate_Theta_td}
		\|\wt \Theta - \Ttheta\|_{\l12} 
		\lesssim \max\left\{{\lambda_3},  ~ {(\wt \lambda_3)^2 \over \lambda_3 }\right\}{s_*\over \kappa^2(s_*,4)},
		\end{equation}
		where
		%, for some constant $c>0$,
		\begin{equation}\label{def_lbd3_td}
		\wt\lambda_3 =   \left\{{1\over \sqrt n}\|\X F^*\|_{op} +  \sqrt\Lambda_1\left(1+ \sqrt{K \over n}\right)\right\}{\kappa(s_*,4) \over \sqrt{s_*}}~ \xi_n.
		\end{equation}
		%with  $Rem(P_{\B})$ defined in (\ref{def_RU}).
	\end{thm}
	%\begin{proof}
	%	The proof is deferred to Appendix \ref{sec_proof_thm_Theta}.
	%\end{proof}
	
	If $K$ is small, one can replace $\sqrt{K/n}$ in (\ref{def_lbd3_td}) by $\sqrt{K\log(n)/n}$ and the resulting probability of (\ref{rate_Theta_td}) will become $1-\epsilon - 2n^{-cK}$ which, by choosing $\epsilon= n^{-1}$, tends to one as $n\to \i$. The same argument is applicable to the subsequent theorems. 
	
	Theorem \ref{thm_Theta} holds for any estimator $\wh P$ of $P_{\B}$ with convergence rate $\|\wh P - P_{\B}\|_F \lesssim \xi_n$. The effect of $\wh P$ on the estimation error of $\wt\Theta$ is characterized by the term $(\wt \lambda_3)^2s_*/[\lambda_3 \kappa^2(s_*,4)]$ in (\ref{rate_Theta_td}) via the choice of $\lambda_3$. When $P_{\B}$ can be estimated very accurately, for instance when $\B$ is known,  $\xi_n$ is fast enough such that $\wt \lambda_3 \le \bar \lambda_3$. We can take $\lambda_3 = \bar \lambda_3$ to obtain the convergence rate $\bar\lambda_3s_*/\kappa^2(s_*,4)$. We refer to this as the oracle rate since it is the optimal rate for estimating $\Ttheta$ from $\Y P_{\B}^{\perp} = \X \Ttheta + \E P_{\B}^{\perp}$ when $B^*$ is known (cf. \cite{lounici2011}). 
	On the other hand, when $\wh P$ has a slow rate such that $\wt \lambda_3 > \bar \lambda_3$, one needs to take a larger $\lambda_3$ to achieve the best trade-off between the two terms in (\ref{rate_Theta_td}). It is easy to see that in this scenario the optimal $\lambda_3$  is equal to $\wt \lambda_3$ and the resulting convergence rate is $\wt \lambda_3 s_*/\kappa^2(s_*,4)$.
	
	\begin{remark}[On the benefit of group-lasso]
		As seen above, when $\wt \lambda_3 \leq \bar \lambda_3$, the convergence rate (\ref{rate_Theta_td}) reduces to the oracle rate $\bar\lambda_3s_*/\kappa^2(s_*,4)$. If $\log p = o(m)$ and $\max_j\wh \Sigma_{jj} = O(1)$ hold, (\ref{def_event_lbd3}) implies $\bar\lambda_3 = O(\sqrt{m/ n})$ by choosing $\epsilon=p^{-1}$. As a result, provided that $[\kappa(s_*,4)]^{-1}=  O(1)$, the average error per response satisfies $\sum_{j=1}^p[m^{-1}\sum_{\ell=1}^m (\wt \Theta_{j\ell} - \Ttheta_{j\ell})^2]^{1/2}=m^{-1/2}\|\wt \Theta - \Ttheta\|_{\l12}= O(s_*/\sqrt{n})$, which does not depend on logarithmic factors of the feature dimension $p$ and is faster than the standard rate of the lasso applied separately to each column of $\Y$. Such a phenomenon is known as the benefit of the group-lasso \citep{lounici2011}. Recall that we also use the group-lasso in our first step (\ref{est_F}) for estimating $\X F^*$. This benefit of the group-lasso remains and can be seen from the choice of $\lambda_1$ in (\ref{rate_lbd1}). Indeed, if $\log p = o(r_e(\Gamma_\eps))$ holds, by choosing $\epsilon = p^{-1}$ in (\ref{rate_lbd1}), the $\log p$ term in $\lambda_1$ is negligible. The quantity $r_e(\Gamma_\eps)$ is the effective rank of $\Gamma_{\eps}$ and it depends on the interplay of $B^{*T}\Sigma_W \B$ and $\Sigma_E$. If $\lambda_1(B^{*T}\Sigma_W \B)$ is small (e.g., upper bounded by a constant), then $r_e(\Gamma_\eps) \asymp m$, whereas if $\lambda_1(B^{*T}\Sigma_W \B) \asymp \lambda_K(B^{*T}\Sigma_W \B)\asymp m$, we have $r_e(\Gamma_\eps) \asymp K$ and $\log p = o(r_e(\Gamma_\eps))$ reduces to $\log p = o(K)$.
		
	\end{remark}

	In the following theorem, we establish non-asymptotic upper bounds of  the estimation error of our  estimator $\wh P_{\B}$ obtained from Section \ref{sec_est_P}. 
	The proof is based on a variant of the Davis-Kahan theorem \citep{Davis-Kahan-variant} together with careful control of the estimation error of $\wh\Sigma_{\eps}$. 
	Let $\Lambda_K$ denote the $K$th largest eigenvalue of $(B^*)^T\Sigma_W B^*$.
	
	\begin{thm}\label{thm_U}
		Under model (\ref{model}) and Assumption \ref{ass_error}, assume $ m \le e^n$. For some constants $c,c'>0$, one has
		\[
		\PP\left\{\|\wh P_{\B}- P_{\B}\|_F \le  c \cdot Rem(P_{\B})\right\} \ge 1-\epsilon'-5m^{-c'},
		\]
		where, with $V_{\eps}$ and $\Gamma_{\eps}$ defined in (\ref{def_V_eps}), 
		\begin{align}\label{def_RU}\nonumber
		Rem(P_{\B})  &=  \!\!\! \inf_{\substack{(\Psi_0, L_0):\\
				\Psi_0 + L_0 = F^*}}  {1\over \Lambda_K} \bigg \{ V_\eps\sqrt{\log m\over n} +  {\lambda_2\sigma_1 \over \lambda_2 + \sigma_1}{\|L_0\|_F^2}+\sum_{k=1}^q {\sigma_k \over \sigma_k + \lambda_2}{V_\eps \over n}\\
		&\hspace{1.2cm}    + {\lambda_2(\sigma_1 + \lambda_2) \over (\sigma_q + \lambda_2)^2} \max_{1\le j\le p} \wh \Sigma_{jj}\left(1 + {\log(p/\epsilon') \over r_e(\Gamma_{\eps})}\right){s_0\over [\kappa(s_0,4)]^2}{V_\eps \over n}
		\bigg \}.
		\end{align}
	\end{thm}
	%\begin{proof}
	%	The proof is deferred to Appendix \ref{sec_proof_thm_U}.
	%\end{proof}
	Recall that $\wh P_{\B}$ relies on the estimates of both $\Sigma_{\eps}$ and $\X F^*$. The first term $V_{\eps} \sqrt{\log m/n}$ in $Rem(P_{\B})$ is the oracle error of estimating $\Sigma_{\eps}$ in Frobenius norm even if $\X F^*$ were known. The other three terms in $Rem(P_{\B})$ originate from the errors of estimating $\X F^*$ in Corollary \ref{cor_pred}. 
	
	When $\wt\Theta$ is obtained from (\ref{est_Theta}) by using $\wh P_{\B}$, combining Theorem \ref{thm_Theta} and Theorem \ref{thm_U} yields the final rate of $\|\wt\Theta - \Ttheta\|_{\l12}$ with explicit dependency on all quantities. To simplify its expression, we introduce the parameter space 
	\begin{equation}\label{def_par}
	(\Ps, L^*) \in \left\{
	(\Psi, L): \Psi + L = F^*,  \|\Psi\|_{\ell_0/\ell_2} \le s_*, \|L\|_F^2 \le R_*
	\right\}
	\end{equation}
	for some $R_*\geq 0$. Without loss of generality, we standardize the design matrix such that $\wh\Sigma_{jj} = 1$ for $1\le j\le p$.  We assume the following conditions.

	%Recall that $\sigma_1$ is the largest eigenvalue of $\wh\Sigma$.
	\begin{ass}\label{ass_rates}\mbox{}
		\begin{enumerate}
			\item[(a)] $[\kappa(s_*, 4)]^{-1} =O(1)$, $n^{-1}\|\X F^*\|_{op}^2 = O(m + s_*)$;
			\item[(b)] $\Lambda_1 \asymp \Lambda_K \asymp m$ with $\Lambda_1$ and $\Lambda_K$ being the first and $K$th largest eigenvalues of $(B^*)^T\Sigma_W B^*$.
		\end{enumerate}
	\end{ass}
	
	The verification of Assumption \ref{ass_rates} is deferred to Section \ref{sec_dis_cond}. Under Assumption \ref{ass_rates}, the following Corollary \ref{cor_rate}  simplifies the rates of $\|\wt \Theta - \Ttheta\|_{\l12}$ by combining Theorem \ref{thm_Theta} with Theorem \ref{thm_U}. For two sequences $a_n$ and $b_n$, we write  $a_n \lessapprox b_n$ for $a_n = O(b_n)$  up to multiplicative logarithmic factors of $m$ or $p$. Recall that $q = \rank(\X)$.
	
	\begin{cor}\label{cor_rate}
		Under model (\ref{model}) and Assumptions \ref{ass_error} \& \ref{ass_rates}, assume $\kappa(s_*, 4) > 0$ and $K = O(n)$. For any $(\Ps, L^*)$ satisfying (\ref{def_par}), there exists a suitable choice of $\lambda_2$ in (\ref{est_F}) such that the following holds with probability tending to one, 
		\begin{align}\label{rate_theta_simp}
		{1\over \sqrt m}\|\wt \Theta - \Ttheta\|_{\l12} 
		~ \lessapprox ~ \max\left\{  {s_* \over \sqrt{n}}, ~
		\sqrt{s_*(m + s_*)\over  m}\cdot {\rm Err}(P_{\B})\right\},
		\end{align}
		where 
		\begin{align*}
		{\rm Err}(P_{\B}) = \min\Bigg\{
		&{\sigma_1 R_* \over m} + {Ks_*\over n}, {qK\over n} ,\sqrt{{(p+\sigma_1s_*)K R_*\over nm}} + {Ks_*\over n}
		\Bigg\} + {K\over \sqrt n}.
		\end{align*}
	\end{cor}
	%\begin{proof}
	%	The proof is deferred to Appendix \ref{sec_proof_cor_rate}.
	%\end{proof}
	In view of (\ref{rate_theta_simp}), $s_*/\sqrt{n}$ is the oracle rate for estimating $\Ttheta$ as discussed after Theorem \ref{thm_Theta}. The term ${\rm Err}(P_{\B})$  quantifies the minimum price to pay for estimating $P_{\B}$ by $\wh P_{\B}$ over all choices of $\lambda_2$. 
	%and it is the minimum of three error terms which are related to the estimation of $\X F^*$.  
	Recalling that ${\rm Rem}(P_{\B})$ in Theorem \ref{thm_U} depends on the choice of the tuning parameter $\lambda_2$, the derivation of  ${\rm Err}(P_{\B})$ minimizes ${\rm Rem}(P_{\B})$ with respect to $\lambda_2$. The three error terms in ${\rm Err}(P_{\B})$ correspond to different choices of $\lambda_2$ depending on the interplay of the terms in ${\rm Rem}(P_{\B})$ (see the proof of Corollary \ref{cor_rate} for more details).	To facilitate understanding, we further simplify (\ref{rate_theta_simp}) in low- and high-dimensional settings.

	\begin{remark}[Further simplified rates of $\|\wt \Theta - \Ttheta\|_{\l12}$]\label{rem_rate_Theta}\mbox{}
		
		(i) Suppose $p  < n$, $p \asymp s_*$, $\sigma_1  = O(1)$ and $K = O(\sqrt{p} \wedge \sqrt{n/p}\wedge m)$. Then (\ref{rate_theta_simp}) becomes
		\begin{alignat*}{2}
		{1\over \sqrt m}\|\wt \Theta - \Ttheta\|_{\l12} 
		&~\lessapprox ~ {p\over \sqrt n},\qquad&&\text{when }p = O(m);\\
		{1\over \sqrt m}\|\wt \Theta - \Ttheta\|_{\l12} 
		&~\lessapprox ~  {p \over \sqrt  n} + \left({p \over \sqrt  n}\right)^{2} ,\qquad&&\text{when $m=O(1)$}. 
		\end{alignat*}
		Recall that the oracle rate in this case is $p/\sqrt n$. As long as $p/\sqrt n = o(1)$ which is the minimum requirement for consistent estimation of $\Ttheta$ in $\l12$ norm, our estimator $\wt \Theta$ achieves the oracle rate. If $m$ grows, we also allow $K$ to grow but no faster than $\sqrt{p} \wedge \sqrt{n/p}\wedge m$.
		
		(ii) Suppose $p \ge n$, $s_*<n$ and $K=O(\sqrt{s_*} \wedge \sqrt{n/s_*} \wedge m)$. The upper bound in (\ref{rate_theta_simp}) becomes
		\begin{alignat*}{2}
		%&{1\over \sqrt m}\|\wt \Theta - \Ttheta\|_{\l12}  \lessapprox &&\\ 
		&\quad     {s_* \over \sqrt{n}} + \sqrt{s_*R_* \over m} \min\left\{\sigma_1\sqrt{R_* \over m}, ~\sqrt{{(p+\sigma_1s_*)K\over n}}\right\}, \quad &&\text{if }s_* = O(m);\\
		&\quad  {s_* \over \sqrt{n}} + \left({s_* \over \sqrt{n}}\right)^2 +s_*\sqrt{R_*}\min\left\{\sigma_1\sqrt{R_*}, ~\sqrt{{p+\sigma_1s_*\over n}}\right\}, \quad &&\text{if $m=O(1)$}. 
		\end{alignat*}
		In high-dimensional case, the dimension $m$ plays a more significant role. When $m$ is fixed, one needs 
		$\sigma_1 s_* R_* = o(1)$ or $s_*\sqrt{R_*(p+\sigma_1s_*)}= o(\sqrt n)$ for estimation consistency. This requirement is much more relaxed when $s_* = O(m)$ as $m$ tends to infinity. The benefit of a large $m$ can be viewed as the blessing of dimensionality. 
		In the sequel, we focus on $s_*=O(m)$, which leads to the following sub-cases:    
		\begin{alignat*}{2}
		{1\over \sqrt m}\|\wt \Theta - \Ttheta\|_{\l12} 
		&\lessapprox   {s_* \over \sqrt{n}} + {\sqrt{s_*} \sigma_1R_* \over m},\quad &&\text{if} ~~{R_* \over m}\leq {(p+\sigma_1s_*)K\over n \sigma_1^2}\\
		{1\over \sqrt m}\|\wt \Theta - \Ttheta\|_{\l12} 
		&\lessapprox    {s_* \over \sqrt{n}} + \!\sqrt{s_*(p+\sigma_1s_*)KR_*\over nm},\quad  &&\text{if}~~{R_* \over m}\geq {(p+\sigma_1s_*)K\over n \sigma_1^2}. 
		\end{alignat*}
		Intuitively, the first case is more likely to occur if $\sigma_1$, the largest eigenvalue of $\wh\Sigma$, has moderate magnitude, such as $\sigma_1= O(p/n)$.  We refer to Section \ref{sec_dis_cond} for more comments on this order of $\sigma_1$. In this case, assuming $m \asymp n^{\alpha}$ for some constant $\alpha \ge  1/2$, the rate matches the oracle rate $s_*/\sqrt{n}$ if $\sigma_1R_* = O(\sqrt{s_*n^{2\alpha-1}})$. The larger $\alpha$ is, the weaker the requirement on $R_*$ becomes. On the other hand, when $\wh\Sigma$ has spiked eigenvalues, for instance $\sigma_1 \asymp p$, the second case in the display above is more likely to hold. In this case, assuming $\sigma_1 \asymp p$ and $K=O(1)$, our estimator $\wt \Theta$ achieves the oracle rate $ s_*/\sqrt{n}$ if $R_* = O(m/p)$. In Section \ref{sec_dis_cond}, we provide examples under which $R_* = O(m/p)$ holds.
		
	\end{remark}

	\subsection{A special case under the condition $\Ps P_{\B} = 0$}
	
	As seen in Remark \ref{rem_ident_Psi}, if $\Ps P_{\B} = 0$ holds, then $\Ps=\Ttheta$ is identifiable. The estimator $\wh \Psi$ obtained in (\ref{est_F}) can be viewed as an initial estimator of $\Ps$. The convergence rate of $\|\wh \Psi - \Ps\|_{\l12}$ is shown in Lemma \ref{lem_pred_theta} of Appendix \ref{sec_proof_lem_pred_theta}. Because of $\Ps=\Ttheta$, $\wt \Theta$ can be viewed as a refined estimator of $\Ps$ and its rate of convergence has been analyzed in Theorem \ref{thm_Theta} and Corollary \ref{cor_rate}. 
	In the following remark, we elaborate the improvement of  $\wt \Theta$ over the initial estimator $\wh\Psi$ in terms of their convergence rates. Empirical comparisons of these two estimators are presented in Section \ref{sec_simulation}. 
	
	\begin{remark}[Comparison of $\wh\Psi$ and $\wt\Theta$]\label{rem_rate_theta_hat}
		Assume conditions of Corollary \ref{cor_pred} and
		Assumption \ref{ass_rates} hold. With suitable choices of $\lambda_1$ and $\lambda_2$, $\wh\Psi$ obtained from (\ref{est_F}) satisfies  
		\begin{equation}\label{rate_theta_hat}
		{1\over \sqrt m} \|\wh\Psi - \Ps\|_{\l12} \lessapprox  {s_*\sqrt{K} \over \sqrt n } + \sqrt{s_*}\sqrt{{\sigma_1R_* \over m}}.
		\end{equation}
		Comparing this rate to (\ref{rate_theta_simp}) (notice that $\Ps = \Ttheta$), the advantage of $\wt\Theta$ over the initial estimator $\wh\Psi$ is substantial. For instance, in the low-dimensional case (i) of Remark \ref{rem_rate_Theta},  we have $m^{-1/2}\|\wt\Theta - \Ps\|_{\l12} \lessapprox p/\sqrt{n}$ provided that $p/\sqrt{n}=o(1)$. In contrast, $m^{-1/2}\|\wh\Psi - \Ps\|_{\l12} \lessapprox p\sqrt{K/n}$ which has an extra $\sqrt{K}$ factor even if $\sigma_1R_* / m$ is sufficiently small. 
		In the high-dimensional case (ii), one has 
		\[
		{1\over \sqrt m}\|\wt \Theta - \Ps\|_{\l12} 
		\lessapprox   {s_* \over \sqrt{n}} + \sqrt{s_*(m+s_*)\over m} \left({\sigma_1R_* \over m} + \sqrt{s_*\over n}\right)
		\]
		which is always faster than (\ref{rate_theta_hat}) provided that $\sigma_1R_* = o(m / s_*)$ and $s_* = O(mK)$. Note that in (\ref{rate_theta_hat}) we need $\sigma_1R_* = o(m/s_*)$ for the consistency of $\wh \Psi$. For further illustration, suppose $s_*=O(m)$, $m\asymp n^\alpha$ for some $\alpha \ge 1/2$ and $\sigma_1R_*  = O(\sqrt{s_* n^{2\alpha - 1}})$, then 
		$m^{-1/2}\|\wt \Theta - \Ps\|_{\l12}  \lessapprox {s_* / \sqrt{n}}$ corresponds to the oracle rate, whereas (\ref{rate_theta_hat}) becomes $m^{-1/2}\|\wh \Psi - \Ps\|_{\l12}  \lessapprox s_*(K/n)^{1/2}+\sqrt{s_*}(s_*/n)^{1/4}$. 
	\end{remark}
	
	\subsection{Validity of Assumption \ref{ass_rates} and conditions in Remark \ref{rem_rate_Theta}}\label{sec_dis_cond}
	In this section we provide theoretical justifications for Assumption \ref{ass_rates} as well as some conditions on $\sigma_1$ and the radius $R_*$ that we mentioned in Remark \ref{rem_rate_Theta}.
	%though it is weaker than commonly assumed conditions in the literature of the surrogate variable analysis (SVA), see, for instance, \cite{wang2017, Lee2017, McKennan19}. 

	Part $(a)$ of Assumption \ref{ass_rates} contains standard conditions on the design matrix. Suppose the rows of $\X\Sigma^{-1/2}$ are i.i.d. sub-Gaussian random vectors with bounded sub-Gaussian constant. The validity of $[\kappa(s_*, 4)^{-1} =O(1)$ is already discussed in Remark \ref{rem_design_impact_factor}. Regarding condition $n^{-1}\|\X F^*\|_{op}^2 = O(m + s_*)$, suppose $s_* = O(n)$, $K = O(n)$ and $\lambda_1(\Sigma_Z) = O(1)$. We show in Lemma \ref{lem_cond} of Appendix \ref{sec_proof_auxiliary} that $n^{-1}\|\X F^*\|_{op}^2 = O_p(m + s_*)$ provided that $\|\Ps\|_{op}^2 = O(s_*+m)$, $\|B^*\|^2_{op} = O(m)$ and $\|\Sigma_{S_*S_*}\|_{op} = O(1)$ where $S_*$ is the set of non-zero rows of $\Ps$. Recall that $\|\Ps\|_{\ell_0 / \ell_2} \le s_*$ and $\B \in \RR^{K\times m}$ with $K< m$, $\|\Ps\|_{op}^2 =O(m + s_*)$ and $\|B^*\|^2_{op} = O(m)$ hold when either $\|\Ps\|_{\i} = O(1)$ and $\|B^*\|_{\i} = O(1)$ or entries of $\Ps$ and $\B$ are i.i.d. samples from a mean-zero distribution with bounded fourth moment \citep{Bai-Yin}.

	Condition $(b)$ is standard when $m$ (and also $K$) is fixed. When $m$ grows with $n$, we note that $\Eps = \W B^*+\E$ follows a factor model where $\W$ is the matrix of $K$ stochastic factors and $B^*$ is the factor loading matrix. Condition $(b)$ is known as the  pervasiveness assumption in the factor model literature for identification and consistent estimation of the row space of the factor loading $B^*$ \citep{Bai-factor-model-03, fan2011, fan2013large, fan2017}. In particular, condition $(b)$ holds if $c\le \lambda_K(\Sigma_W) \le \lambda_1(\Sigma_W)\le C$ for some constants $c,C>0$, and the columns of $B^*$ are i.i.d. copies of a $K$-dimensional sub-Gaussian random vector whose covariance matrix has bounded eigenvalues. %We refer to the aforementioned literature for further discussion of this condition.  
	It is worth mentioning that this assumption is only used to simplify the order of $\wt\lambda_3$ in (\ref{def_lbd3_td}) and $Rem(P_{\B})$ in (\ref{def_RU}). If $\Lambda_1$ and $\Lambda_K$ have different rates, we can replace them by the corresponding rates and simplify the error bounds of $\wt \Theta$ accordingly. We also verify the empirical performance of our procedure in  Appendix \ref{sec_supp_sim} when some of $\Lambda_1,\ldots,\Lambda_K$ are moderate or small.
	%This condition is also commonly assumed in the aforementioned SVA literature. 

	Regarding conditions on $\sigma_1$ in part (ii) of Remark \ref{rem_rate_Theta}, when $\|\Sigma\|_{op} = O(1)$, one has $\sigma_1 = O_p(p/n)$ by $\| \wh\Sigma - \Sigma\|_{op} =O_p(\sqrt{p/n} \vee (p/n))$ from \cite{vershynin_2012}. To see when $\sigma_1 \asymp p$ holds, suppose $\|\Sigma\|_{op} \asymp p$ and $\log p = o(n)$. Since $\|\wh\Sigma - \Sigma\|_{op} \le \|\wh \Sigma - \Sigma\|_F = O_p(p\sqrt{\log p / n})$ (for instance, see the argument in Lemma \ref{lem_eps_eps} of Appendix \ref{sec_proof_auxiliary_lemma}), one can deduce that $\sigma_1 \asymp p$ with high probability.

	In the end, we comment on the magnitude of $R_*$, the upper bound of $\|L^*\|^2_F$ in (\ref{def_par}). % Since $L^* = A^*B^*$ with $A^* = \{\EE[XX^T]\}^{-1}\EE[XZ^T]$, we can interpret $L^*$ as the effect of the hidden variables $Z$ on the response that can be explained by a linear combination of $X$. 
	In the  mediation analysis via structural equation models, $L^*= A^*B^*$ is known as the indirect effect of $X$ on $Y$. Under model (\ref{model}), the estimation of the non-sparse coefficient matrix $\Ps+L^*$ becomes challenging in high dimension when $\|L^*\|_{F}^2$ is large, and the estimation error is further accumulated in the rates of the final estimator $\wt \Theta$ as shown in Corollary \ref{cor_rate}. Thus, intuitively, $\|L^*\|_{F}^2$ cannot grow too fast in order to guarantee the consistency of $\wt \Theta$. This can be compared to the standard results in linear regression. For instance, in linear regression $ \y = \X\beta + {\bm \epsilon}$ where $\y\in\RR^n$ and $\beta\in \RR^p$ is dense, one needs $\|\beta\|_2^2= o(1)$ for consistent estimation when $p>n$; see \cite{Hsu2014, dicker2016} for the minimax lower bound.

	In the following, we discuss under what conditions $\|L^*\|_{F}^2$ is small and how small it can be. Provided that $\|\B\|^2_{op} = O(m)$, we first have $\|L^*\|_F^2/m \le \|A^*\|_F^2 \|B^*\|_{op}^2/m =O(\|A^*\|_{F}^2)$. Since 
	$A^* = \Sigma^{-1}\Cov(X,Z)$, intuitively $\|A^*\|_{F}$ is small when either (1) $\Cov(X,Z)$ is close to zero or (2) $\Sigma$ is ``large". 
	
	To show when case (1) holds, suppose the smallest eigenvalue of $\Sigma$ is bounded away from zero. When 
	$\Cov(X,Z)$ is sparse with $\|\Cov(X,Z)\|_{\ell_0}=O(1)$ and $\max_{j,k}|\Cov(X_j, Z_k)| \lesssim \xi $, one has $\|A^*\|_F^2 = O(\xi^2)$ which vanishes if $\xi = o(1)$. %When $\Cov(X,Z)$ is dense with $\|\Cov(X,Z)\|_{\ell_0} \ge c'(pK)$ for some small constant $c'>0$, if the range of the nonzero entries of $\Cov(X,Z)$ is bounded, an application of P{\'o}lya-Szeg{\"o}’s inequality (see, for instance, \cite{reverse_CS}) yields $\|\Cov(X,Z)\|_F \lesssim \|\Cov(X,Z)\|_{\ell_1/\ell_1} /\sqrt{pK}$. Therefore, provided that	$\|\Cov(X,Z)\|_{\ell_1/\ell_1}=O(1)$, one has $\|L^*\|_F^2/m = O(\|A^*\|_{F}^2) =O(1/(pK)).$

	To show when case (2) holds, we consider the setting that $X$ follows an approximate factor model $X = \Gamma F + W'$, where the noise $W'$ and the factor $F$ are independent, $\Cov(F)$ and $\Cov(W')$ have bounded eigenvalues and the loading matrix $\Gamma \in \RR^{p\times \bar K}$ satisfies the pervasiveness assumption $\lambda_{\bar K}(\Gamma \Gamma^T) \gtrsim p$. %We refer to the fourth paragraph of this section for discussions of this pervasiveness assumption. 
	The number of factors, $\bar K$, is often much smaller than $p$. In this scenario, $\Sigma = \Gamma \Cov(F) \Gamma^T + \Cov(W')$ has $\bar K$ spiked eigenvalues with order at least $p$. To show the order of $\|L^*\|_F^2$, we consider the eigen-decomposition of $\Sigma = \sum_{j=1}^p   d_j v_j v_j^T$ with $d_1 \ge \cdots \ge d_p$. Further write $V_{(\bar K)} = (v_1, \ldots, v_{\bar K})\in \RR^{p\times \bar K}$ and $V_{(-\bar K)} = (v_{\bar K+1}, \ldots, v_p)\in \RR^{p\times (p-\bar K)}$.  Provided that 
	\begin{equation}\label{eqVK}
	\left\|V_{(\bar K)}^T \Cov(X,Z)\right\|_{op} \le c\sqrt{p} ~~\textrm{and}~~\left\|V_{(-\bar K)}^T \Cov(X,Z)\right\|_{op} \le c/\sqrt{p},
	\end{equation}
	for some sufficiently small constant $c>0$, we obtain 
	\begin{align*}
	\|A^*\|_{op} &= \left\|\Sigma^{-1}\Cov(X,Z)\right\|_{op}\\
	&\le {1\over d_{\bar K}}\left\|V_{(\bar K)}^T \Cov(X,Z)\right\|_{op} + {1\over d_p}\left\|V_{(-\bar K)}^T \Cov(X,Z)\right\|_{op} = O(1/\sqrt{p}).
	\end{align*}
	We thus have$\|L^*\|_F^2 / m =O(\|A^*\|_F^2)= O(K/p)$. Also notice that this setting does not conflict with part (b) of Assumption \ref{ass_rates}. Indeed, provided that $c\le \lambda_K(\Sigma_Z) \le \lambda_1(\Sigma_Z)\le C$ and $cm \le \lambda_K(\B B^{*T})\le \lambda_1(\B B^{*T}) \le C m$ for some constants $c,C>0$, one can deduce $c/2\le \lambda_K(\Sigma_W) \le \lambda_1(\Sigma_W)\le C$ from  $\Sigma_W=\Sigma_Z-\Cov(Z,X)\Sigma^{-1}\Cov(X,Z)$. 
	
	Condition (\ref{eqVK}) requires that: (1) the order of $\|\Cov(X,Z)\|_{op}$ cannot be greater than $\sqrt{p}$; (2) the columns of $\Cov(X,Z)$ and $V_{(-\bar K)}$ are approximately orthogonal. From a practical perspective, under the structural equation model (\ref{modelZ}), condition (\ref{eqVK}) implies that the causal effect of $X$ on $Z$ (i.e., the matrix $A^*$) is weak due to the spiked eigenvalues of $\Sigma$ in high dimension. However, in view of the factor model $X = \Gamma F + W'$, the association between the hidden variable $Z$ and the low dimensional factor $F$ can be strong. 
	%Thus, requiring $\|L^*\|^2_F$ or $\|A^*\|_{op}$ small is reasonable for many practical situations and does not rule out the effect of hidden variables in model (\ref{model_1}). 
	Though $\|L^*\|_F$ or $\|A^*\|_{op}$ is required to be small when $p$ is large, our analysis allows nontrivial dependence between $Z$ and $F$, which could be reasonable in many practical situations.

	%As a result, the indirect effect of $X$ on $Y$ via $Z$ cannot be ignored when estimating $\Ttheta$. 

	\section{Extension to heteroscedastic noise}\label{sec_hetero}
	We have discussed the identifiability and estimation in model (\ref{model_1}) when the errors are homogeneous. In practice, the multivariate response $Y$ may correspond to measurement of different properties (e.g., phenotypes) whose values could differ in scales. To deal with this problem, in this section we extend the model by allowing heteroscedastic errors, $\Sigma_E = \diag(\tau_1^2, \ldots, \tau_m^2)$, and discuss how to modify our approach correspondingly.
	
	\subsection{Identifiability}
	
	From the identifiability in Section \ref{sec_ident}, we observe that the heteroscedasticity only affects the identification of $P_{\B}$ in step (2). When $\Sigma_E = \diag(\tau_1^2, \ldots, \tau_m^2)$, one has 
	\begin{equation}\label{eq_Sigma_eps_hetero}
	\Sigma_{\eps} = (B^*)^T\Sigma_W B^* + \diag(\tau_1^2, \ldots, \tau_m^2).
	\end{equation}
	In contrast to the homoscedastic case, the eigenspace of $\Sigma_{\eps}$ corresponding to the first $K$ eigenvalues, in general, no longer coincides with the row space of $B^*$. Consequently, one cannot identify $P_{\B}$ via the eigenspace of $\Sigma_{\eps}$ as in Section \ref{sec_ident}. To overcome this difficulty, we resort to a newly developed procedure called HeteroPCA proposed by \cite{hpca}. For completeness, we restate their procedure in Algorithm \ref{alg_hpca}. The main idea is to iteratively perform the singular value decomposition (SVD) on the estimates of $\Sigma_{\eps}$ to impute its diagonal. Under a mild incoherence condition on the row space of $B^*$, $P_{\B}$ can be recovered by applying Algorithm \ref{alg_hpca} to $\Sigma_{\eps}$. Thus $\Ttheta$ is identifiable from (\ref{model_Y_comp}). We  summarize the identifiability below in Proposition \ref{prop_ident_hetero}. 
	
	Recall that $P_{\B} = UU^T$ with $U := U(K)\in \RR^{m\times K}$ being the first $K$ right singular vectors of $B^*$, and $\Lambda_1$ and $\Lambda_K$ are the first and $K$th eigenvalues of $(B^*)^T\Sigma_W B^*$, respectively. Let $\{e_j\}_{j=1}^m$ denote the canonical basis of $\RR^m$. 
	\begin{prop}\label{prop_ident_hetero}
		Under model (\ref{model}), assume $\Sigma_E = \diag(\tau_1^2, \ldots, \tau_m^2)$ and $\rank(\Sigma_W) = K$. Further assume 
		\begin{equation}\label{cond_ic}
		{\Lambda_1 \over \Lambda_K} \max_{1\le j\le m}\|e_j^T U\|_2^2
		\le  C_U 
		\end{equation}
		for some constant $C_U >0$. Then $P_{\B}$ can be uniquely determined via Algorithm \ref{alg_hpca} with input $\wh\Sigma = \Sigma_{\eps}$, $r = K$ and some sufficiently large number of iterations $T$. As a result,	$\Ttheta$ is identifiable. 
	\end{prop}
	An application of Theorem 3 in \cite{hpca} guarantees the recovery of $P_{\B}$ from $\Sigma_{\eps}$ and the  
	rest of the proof follows the same lines as the proof of Proposition \ref{prop_ident}. Compared to the homoscedastic case, we need an extra condition (\ref{cond_ic}) for identifying $P_{\B}$, which can be viewed as the price to pay for allowing heteroscedasticity. Inherent from the HeteroPCA algorithm, this condition prevents matrices $UQ$, for any orthogonal matrix $Q\in \RR^{K\times K}$, being well aligned with canonical basis vectors. Otherwise, one cannot separate $(B^*)^T\Sigma_W B^*$ from a diagonal matrix with unequal entries. We also note that $(m/K)\max_{1\le j\le m} \|e_j^TU\|_2^2$ is known as the \emph{incoherence constant} in the matrix completion literature \citep{ct,Candes}. When $\Lambda_1 \asymp \Lambda_K$, (\ref{cond_ic}) requires $\max_{1\le j\le m} \|e_j^TU\|_2^2 = O(1)$ which is much weaker than the typical incoherence condition $\max_{1\le j\le m} \|e_j^TU\|_2^2 = O(K/m)$, assumed in the matrix completion literature. Finally, Proposition 3 in \cite{hpca} implies that condition (\ref{cond_ic}) in general cannot be further relaxed in order to recover $P_{\B}$ from $\Sigma_{\eps}$. %We refer to \cite{hpca} for the details.
	
	\begin{remark}[Identification via PCA when $m\rightarrow\infty$]\label{rem_robust_PCA}
		We propose to use HeteroPCA to identify $P_{\B}$ in the presence of heteroscedasticity since it guarantees the identifiability of $\Ttheta$ for any $m> K$ under  condition (\ref{cond_ic}). Directly applying PCA to $\Sigma_{\eps}$ as in Section \ref{sec_ident} may not recover $P_{\B}$ hence not identify $\Ttheta$. 
		However, we remark that PCA is robust against the departure from homoscedasticity, and  even from the diagonal structure of $\Sigma_E$, when $\Lambda_K$, the $K$th eigenvalue of $(B^*)^T\Sigma_W B^*$, diverges fast enough as $m\to\infty$. Specifically, at the population level, applying PCA to $\Sigma_{\eps}$ identifies $P_{\B}$ asymptotically provided that
		$\sqrt{K}\|\Sigma_E\|_{op} = o(\Lambda_K)$, as $m\to\infty$. 
		This phenomenon is known as the blessing of dimensionality in the factor model literature  \citep{Bai-factor-model-03,fan2013large,fan2017}. Most of the SVA methods, for instance \cite{Lee2017, McKennan19}, rely on this robustness of PCA. Their methods thus only guarantee the asymptotic identifiability when $m\to\infty$, and are not applicable if $m$ is fixed.  
	\end{remark}
	
	{\begin{algorithm}[ht]
			\caption{HeteroPCA$(\wh\Sigma, r, T)$}\label{alg_hpca}
			\begin{algorithmic}[1]
				\State Input: matrix $\wh\Sigma$, rank $r$, number of iterations $T$.
				\State Set $N^{(0)}_{ij} = \wh \Sigma_{ij}$ for all $i\ne j$ and $N^{(0)}_{ii} = 0$. 
				\For{$t = 0,1,\ldots, T$}
				\State Calculate SVD: $N^{(t)} = \sum_i \lambda_i^{(t)}u_i^{(t)}(v_i^{(t)})^T$, where $\lambda_1^{(t)}\ge \lambda_2^{(t)}\ge \cdots \ge 0$.
				\State Let $\wt N^{(t)} = \sum_{i=1}^r\lambda_i^{(t)}u_i^{(t)}(v_i^{(t)})^T$.
				\State Set $N^{(t+1)}_{ij} = \wh\Sigma_{ij}$ for all $i\ne j$ and  $N^{(t+1)}_{ii} = \wt N^{(t)}_{ii}$. 
				\EndFor
				\State Output $U^{(T)} = [u_1^{(T)}, \ldots, u_r^{(T)}]$.
			\end{algorithmic}
	\end{algorithm} }
	
	\subsection{Estimation}\label{sec_hetero_est}
	
	Our estimation procedure under heteroscedasticity remains the same except estimating $U$ by HeteroPCA in Algorithm \ref{alg_hpca}. To be specific, we consider the estimator $\wt P_{\B} =\wt U \wt U^T$, where $\wt U$ is obtained from Algorithm \ref{alg_hpca} with the input $\wh\Sigma=\wh\Sigma_{\eps}$, $r=K$ and a large $T$ for the algorithm to converge. Our simulation reveals that $T=5$ usually yields satisfactory results. %Alternatively one can  keep iterating until the algorithm converges. 
	We still assume $K$ is known and defer the discussion of selecting $K$ to Section \ref{sec_select_K}.	We state the modified algorithm in Algorithm \ref{alg_2}, named as \underline{H}eteroscedastic \underline{HI}dden \underline{V}ariable adjustment \underline{E}stimation (H-HIVE). 
	
	{\begin{algorithm}[ht]
			\caption{The H-HIVE procedure for estimating $\Ttheta$.}\label{alg_2}
			\begin{algorithmic}[1]
				\Require Data $\X\in \RR^{n\times p}$ and $\Y \in \RR^{n\times m}$, rank $K$, number of iterations $T$, tuning parameters $\lambda_1$, $\lambda_2$ and $\lambda_3$.
				\State Estimate $\X\wh F$ with $\wh F = \wh \Psi + \wh L$ by solving (\ref{est_F}).
				\State Obtain $\wh \Sigma_{\eps}$ from (\ref{def_Sigma_eps_hat}). 
				\State Compute $\wt P_{\B} = \wt U\wt U^T$ with $\wt U$ obtained from HeteroPCA($\wh \Sigma_{\eps}, K,T$) in Algorithm \ref{alg_hpca}.
				\State Estimate $\Ttheta$ by solving (\ref{est_Theta}) with $\wt P_{\B}$ in lieu of $\wh P_{\B}$.
			\end{algorithmic}
	\end{algorithm} }
	
	\subsection{Statistical guarantees}
	
	Our estimation algorithm enjoys similar statistical guarantees as in Section \ref{sec_theory}. First, since $\wh F$ is the same estimator obtained from (\ref{est_F}), the deviation bounds of $\|\X \wh F - \X F^*\|_F$ in Theorem \ref{thm_pred} and Corollary \ref{cor_pred} still hold under Assumption \ref{ass_error}. Second, Theorem  \ref{thm_U_hetero} below provides non-asymptotic upper bounds for $\|\wt P_{\B} - P_{\B}\|_F$ with $\wt P_{\B} = \wt U \wt U^T$ and $\wt U$ obtained from Algorithm \ref{alg_hpca}.  Finally, since $\wt \Theta$ is obtained from the same criterion in (\ref{est_Theta}) by using $\wt P_{\B}$ in place of $\wh P_{\B}$,  the convergence rate of $\|\wt \Theta - \Ttheta\|_{\l12}$ immediately follows from the following theorem in conjunction with Theorem \ref{thm_Theta}.

	\begin{thm}\label{thm_U_hetero}
		Under the same conditions of Theorem \ref{thm_U},  assume condition (\ref{cond_ic}) holds and 
		$Rem(P_{\B}) \le c\sqrt{K}$ for some constant $c>0$ with $Rem(P_{\B})$ defined in (\ref{def_RU}). For some constants $c',c''>0$, the estimator $\wt P_{\B} =\wt U \wt U^T$ with $\wt U$ obtained from Algorithm \ref{alg_hpca} satisfies 
		\[
		\PP\left\{\|\wt P_{\B} - P_{\B}\|_F \le c'\cdot Rem(P_{\B})\right\} \ge 1-\epsilon' - 5m^{-c''}.
		\]	
	\end{thm}
	%\begin{proof}
	%	The proof is deferred to Appendix \ref{sec_proof_thm_U_hetero}.
	%\end{proof}
	
	The proof of Theorem \ref{thm_U_hetero} mainly relies on a new robust $\sin\Theta$ theorem stated in Appendix \ref{sec_sintheta}, which provides upper bounds for the Frobenius norm of $\sin\Theta(\wt U,U):=\wt U^T_\perp U$, where $\wt U$ is the output of Algorithm \ref{alg_hpca} and $\wt U_\perp$ is its orthogonal complement. The new $\sin\Theta$ theorem complements Theorem 3 in \cite{hpca} which controls the operator norm of $\sin\Theta(\wt U,U)$. In order to establish the rate of $\wt \Theta$, we need this new  result to control the Frobenius norm of the estimated eigenspace. This technical tool can be of its own interest and is potentially useful for many other problems.
	
	The validity of Theorem \ref{thm_U_hetero} also hinges on the condition $Rem(P_{\B}) \le c\sqrt{K}$. 
	Under conditions of Corollary \ref{cor_rate} and Remark \ref{rem_rate_Theta}, by inspecting their proofs in Appendix \ref{sec_proof_cor_rate}, one can verify that $Rem(P_{\B})= O(\sqrt{K})$  holds for a suitable choice of $\lambda_2$, provided that, up to a multiplicative logarithmic factor, $p\sqrt{K}=O(n)$ in the low-dimensional case or  $(s_*\vee \sqrt{K})\sqrt{K} = O(n)$ and $\sigma_1R_* = O(m\sqrt{K})$ in the high-dimensional case.

	\begin{remark}[Effect of heteroscedasticity on estimating $\Theta$]
		Heteroscedasticity affects the estimation error of $\wt\Theta$ implicitly via $V_{\eps}$  defined in (\ref{def_V_eps}) %(through $Rem(P_{\B})$) 
		and $\bar\lambda_3$ in (\ref{def_event_lbd3}). For simplicity of presentation, we assumed $\{\E_{ij}\}_{j=1}^m$ shares the same sub-Gaussian constant $\g_e$ in Assumption \ref{ass_error}. To illustrate the effect of heteroscedasticity,  one could instead assume $\E_{ij} / \tau_j$ is $\g_e$ sub-Gaussian for $1\le j\le m$. Then by inspecting the proof and using modified arguments in Lemmas \ref{lem_eps_subgaussian} -- \ref{lem_XE} in Appendix \ref{sec_proof_auxiliary_lemma}, it is straightforward to show that the same results in Theorems \ref{thm_pred}, \ref{thm_Theta}, \ref{thm_U} and \ref{thm_U_hetero} hold with $V_{\eps}$ and $\bar\lambda_3$ replaced by 
		\[
		V_{\eps}' = \g_w^2 \tr\left(B^{*T}\Sigma_W \B\right) + \g_e^ 2m \bar\tau^2,~~ 
		\bar\lambda_3' = 4\g_e \bar \tau \sqrt{\max_{1\le\ell\le p}\wh\Sigma_{\ell\ell}}{\sqrt{m} + \sqrt{2\log(p/\epsilon)} \over \sqrt{n}}.
		\]
		where $\bar\tau^2 = m^{-1}\sum_{j=1}^m\tau_j^2$. The quantity $\bar\tau^2$ reduces to $\tau^2$ in the homoscedastic case. But in the presence of strong heteroscedasticity, $\bar\tau^2$ can be of order different from $O(1)$.
	\end{remark}

	%To conclude this section, we further provide statistical guarantees for the robustness of using PCA to estimate $P_{\B}$. As we have mentioned in Remark \ref{rem_robust_PCA}, at the population level, PCA can identify $P_{\B}$ asymptotically  even in the presence of heteroscedasticity provided that the signal $(B^*)^T\Sigma_W\B$ dominates the noise $\Sigma_E$. The following theorem shows that this robustness of PCA remains at the sample level. Recall that $\Sigma_E = \diag(\tau_1^2, \ldots, \tau_p^2)$ and $\wh P_{\B}$ denotes the estimate of $P_{\B}$ by using PCA obtained in Section \ref{sec_est_P}. 
	To conclude this section, we compare the estimation errors of $\wt P_{\B}$ and the PCA-based estimator $\wh P_{\B}$ in Section \ref{sec_est_P} in the presence of heteroscedasticity. Recall that $\Lambda_K$ denotes the $K$th eigenvalue of $B^{*T}\Sigma_W \B$.
	
	\begin{thm}\label{thm_robust_PCA}
		Suppose the same conditions of Theorem \ref{thm_U} hold. Then 
		\[
		\PP\left\{\|\wh P_{\B}- P_{\B}\|_F \le  c\cdot Rem^{(h)}(P_{\B})\right\} \ge 1-\epsilon'-5m^{-c'}
		\]
		for some constants $c, c'>0$, where 
		\begin{align}\label{def_RU_hetero}
		Rem^{(h)}(P_{\B}) =   Rem(P_{\B}) +  {1 \over \Lambda_K}\left[
		\sum_{j=1}^m\left(\tau_j^2-\bar\tau^2\right)^2\right]^{1/2}
		\end{align}
		with $Rem(P_{\B})$ defined in (\ref{def_RU}).
	\end{thm}
	%\begin{proof}
	%	The proof is deferred to Appendix \ref{sec_proof_thm_robust_PCA}.
	%\end{proof}
	Comparing (\ref{def_RU_hetero}) with (\ref{def_RU}), the last term in (\ref{def_RU_hetero}) is the bias of PCA due to heteroscedasticity and  it is exactly zero when the error $E$ is homoscedastic.
	In general, this bias term could vanish if $\Lambda_K$ is large and the degree of heteroscedasticity is small, such as $\Lambda_K \gtrsim m$ and $\sum_{j=1}^m(\tau_j^2-\bar\tau^2)^2 = O(m)$ as $m\to \infty$. This can be viewed as the sample analog of the robustness of PCA in Remark \ref{rem_robust_PCA}. However, we note that, even if the bias term converges to 0, it may have a slower rate than $Rem(P_{\B})$ which renders the rate of $\wh P_{\B}$ slower than that of $\wt P_{\B}$ from HeteroPCA. 
	
	%We remark that similar robustness result of using PCA to estimate $P_{\B}$ could be established even when $\Sigma_E$ is non-diagonal by following the outline of current arguments. However, we omit these results here for ease of presentation, as  it requires several modifications of the current results, including Theorem \ref{thm_pred}, Corollary \ref{cor_pred} and their subsequent remarks. 

	\section{Practical considerations}\label{sec_practical}
	
	In this section, we address several practical concerns. First, we consider how to select $K$, the number of hidden variables. Then, we discuss the effect of overestimating/underestimating $K$ on the estimation of $\Ttheta$.  Selection of tuning parameters and recommendation of standardization are discussed subsequently. In the end, we discuss in details the practical usage of our estimator of $\Ttheta$ for inferring $\Ps$.
	
	\subsection{Selection of  $K$}\label{sec_select_K}
	Recall that $\Eps = \W B^* + \E$ and $K$ corresponds to the rank of the unknown coefficient matrix $B^*$. When $\Eps$ and $\W$ are both observable, estimation of the rank of $\B$ has been studied by \cite{bunea2011, bunea2012, giraud2011, bing2019} in the framework of multivariate response regression. However, since $\Eps$ and $\W$ are both unobserved, we view $\Eps = \W B^* + \E$ as a factor model with $K$ being the number of factors. \cite{Bai-Ng-K} proposed information based criterion to select $K$. However, both this approach and the aforementioned ones in the regression setting require to know the noise level quantified by $\|\Sigma_E\|_{op}$. While it might be possible to estimate $\Sigma_E$ in view of (\ref{eq_Sigma_eps_hetero}), the theoretical justification of this class of methods is unclear under our model.

	%	rely on choosing the regularization parameter based on the noise level. More specifically, the chosen tuning parameter needs to be proportional to the noise level quantifed as $\|\E\|_{op}$ when $\Eps$ is observable. However, since $\Eps$ is not observed and has to be estimated in our case, the selection of tuning parameter should also take this estimation error into account. What is worse is that the estimation error of $\wh \Eps - \Eps$ is hard to quantify and can only be upper bounded. This makes the selection of tuning parameter problematic. 
	
	In the following we consider an eigenvalue ratio approach originally developed by \cite{lam2012, Ahn-2013} for factor models. Specifically, we estimate $\Eps$ by $\wh\Eps = \Y - \X\wh F$ with $\wh F$ obtained from (\ref{est_F}) and construct $\wh\Sigma_{\eps}$ as (\ref{def_Sigma_eps_hat}). We then propose to estimate $K$ by
	\begin{equation}\label{select_K}
	\wh K = \arg\max_{j\in \{1,2,\ldots, \bar K\}}   \wh \lambda_{j} / \wh \lambda_{j+1},
	\end{equation}
	where $\wh\lambda_1 \ge \wh \lambda_2 \ge \cdots$ are the eigenvalues of $\wh \Sigma_{\eps}$ and $\bar K$ is a pre-specified number, for example, $\bar K =\floor{(n\wedge m)/2}$ \citep{lam2012} with $\floor{x}$ standing for the largest integer that is no greater than $x$. This procedure does not require the knowledge of any unknown quantity, such as the noise level $\|\Sigma_E\|_{op}$. 	The following theorem provides theoretical justification for the above procedure. %We allow the noise to be heteroscedastic $\Sigma_E = \diag(\tau_1^2, \ldots, \tau_m^2)$.
	%Recall that $\Lambda_1 \ge \Lambda_2 \ge \cdots \ge \Lambda_K$ denote the first $K$ eigenvalues of $(B^*)^T\Sigma_W B^*$, and we allow the noise to be heteroscedastic $\Sigma_E = \diag(\tau_1^2, \ldots, \tau_m^2)$.
	\begin{thm}\label{thm_K}
		Under model (\ref{model_1}) or equivalently (\ref{model}) with heteroscedastic noise $\Sigma_E = \diag(\tau_1^2, \ldots, \tau_m^2)$, suppose condition (b) in Assumption \ref{ass_rates} holds.  Assume $\max_{1\le j\le m} \tau_j^2 = O(1)$, $Rem(P_{\B}) = o(1)$ with $Rem(P_{\B})$ defined in (\ref{def_RU}). 
		Then with probability $1-\epsilon' - 5m^{-c''}$ for some constant $c''>0$, 
		\[
		{\wh \lambda_{j}\over \wh \lambda_{j+1}} \asymp 1,~\text{ for $1\leq j\leq K-1$}, \text{ and}\quad {\wh \lambda_{K+1} \over \wh \lambda_{K}} = O\left(Rem(P_{\B}) + m^{-1}\right).
		\]
	\end{thm}
	Under Assumption \ref{ass_rates}, $K = O(1)$, $s_*=o(n)$ and $\sigma_1 R_* = o(m)$, one can deduce from the proof of Corollary \ref{cor_rate} that $Rem(P_{\B}) = o(1)$ for a suitable choice of $\lambda_2$. In addition, if $m\rightarrow\infty$,  we obtain $\wh \lambda_K / \wh \lambda_{K+1} \to \i$. Thus, the maximizer of $\wh \lambda_{j} / \wh \lambda_{j+1}$ is no smaller than $K$ asymptotically, i.e., $\wh K\geq K$, which partially justifies the criterion in (\ref{select_K}). 
	
	The criterion (\ref{select_K}) is also related to the ``elbow" approach, which is often used to determine the number of principle components in PCA. If we plot the ratio $\wh \lambda_{j} / \wh \lambda_{j+1}$ against $j$, we expect that the curve has a sharp increase at $j=K$ (since $\wh \lambda_K / \wh \lambda_{K+1} \to \i$), giving an angle in the graph. We can then select this value $j$ as an estimate of $K$. In our simulation, this simple elbow approach and the criterion (\ref{select_K}) usually yield the same results. In Section \ref{sim_2}, we conduct extensive simulations to compare our criterion (\ref{select_K}) with some other existing methods for selecting $K$ \cite{PA}.

	%	Another popular method for selecting $K$ is proposed in \cite{PA} based on the permutation test.  This method has good empirical performance as observed by \cite{Lee2017}, though it becomes computationally expensive for large $m$. We compare this method with our proposed criterion (\ref{select_K}) in our simulation study.

	\subsection{Consequence of overestimating or underestimating $K$}\label{sec_remark_K}
	It is of interest to understand the effect of selecting an incorrect $K$ on the estimation of $\Ttheta$. Recall that, after estimating $K$ by $\wh K$, we construct $\wh P_{\wh K} = \wh U_{\wh K} \wh U_{\wh K}^T$  and use it in lieu of $\wh P_{\B}$ in (\ref{est_Theta}) to estimate $\Ttheta$. For illustration purpose, we consider the case that $\wh K=r$ for some fixed integer $1\le r\le m$. At the population level, suppose we know the semi-orthogonal matrix $ U_r = (u_1, \ldots, u_r)$ such that when $r<K$, $U_r$ contains the first $r$ columns of $U:=U_K$,  the right singular vectors of $B^*$, and when $r\ge K$, the first $K$ columns of $U_r$ align with those of $U_K$ and the rest of $r-K$ columns are arbitrary but orthogonal to $U_K$. Similar to $P_{B^*}=U_KU_K^T$, $P_r=U_rU_r^T$ is also a projection matrix. The following lemma demonstrates that the effect of using 
	$P_r$ to estimate $\Ttheta$ 
	%in (\ref{est_Theta}) 
	is characterized by the difference of two projection matrices $P_r$ and $P_{B^*}$. 
	
	%	The following lemma shows the effect of using $P^{\perp}_r$ in (\ref{est_Theta}).
	\begin{lemma}\label{lem_mis_K}
		Under model (\ref{model}), $P^{\perp}_r Y$ is equal to
		\begin{alignat*}{2}
		&\left[\Ttheta+(\Ps P_{\B}+A^*B^*)\left(P_{B^*}-P_r\right)\right]^TX  + P^{\perp}_r[(B^*)^TW+E], ~&& \text{ if }r < K;\\
		&(\Ttheta)^TX-(P_r-P_{B^*})(\Ttheta)^TX+P^{\perp}_r E,&&\text{ if }r > K;\\
		&(\Ttheta)^TX+P^{\perp}_r E,&&\text{ if }r = K.
		\end{alignat*}
	\end{lemma}
	
	As we can see, if $r < K$, the estimand of (\ref{est_Theta}) is $\Ttheta+(\Ps P_{\B}+A^*B^*)(P_{B^*}-P_r)=\Ttheta + [\Ps B^{*T}(B^*B^{*T})^{-1}+A^*](B^*)_{(-r)}$ 
	where we apply SVD to $B^* = \sum_{j} d_j u_j v_j^T$ with $d_j$ being non-increasing singular values and $(B^*)_{(-r)} = \sum_{j > r} d_ju_j v_j^T$. Thus, the estimator in (\ref{est_Theta}) has bias $[\Ps B^{*T}(B^*B^{*T})^{-1}+A^*](B^*)_{(-r)}$. Intuitively, if the last $K-r$ singular values of $B^*$, $d_{r+1},...,d_K$, are relatively small and close to zero, we expect the bias to be negligible. In this case, underestimating $K$ may still lead to a reasonably accurate estimate of $\Ttheta$. On the other hand, if $r > K$, our estimator is also biased, and the bias is equal to $-(P_r-P_{B^*})(\Ttheta)^T=-P_r(\Ttheta)^T$ (the equality holds by the orthogonality between $P_{\B}$ and $P_{\B}^{\perp}$). Its magnitude depends on the angle between rows of $\Ttheta$ and the last $r-K$ columns of $U_r$.

	\subsection{Choosing tuning parameters $\lambda_1,\lambda_2$ and $\lambda_3$}\label{sec_tuning}
	Recall that our procedure (Algorithms \ref{alg_1} and \ref{alg_2}) requires three tuning parameters $\lambda_1,\lambda_2$ and $\lambda_3$. Since the first two parameters $(\lambda_1,\lambda_2)$ and the third one $\lambda_3$ appear in two optimization problems (\ref{est_F}) and (\ref{est_Theta}), respectively, we propose to select $(\lambda_1,\lambda_2)$ and $\lambda_3$ separately by cross validation. When estimating $F^*$ in (\ref{est_F}), we can search $\lambda_1$ and $\lambda_2$ over a two-way grid to minimize the mean squared prediction error via $k$-fold cross validation.\footnote{When both $p$ and $m$ are large, searching ($\lambda_1, \lambda_2$) over a fine two-way grid could be computationally intensive. We offer an alternative way of selecting $\lambda_1$ and $\lambda_2$ in Appendix \ref{sec_supp_sim} which costs less computation.} Similarly, when estimating $\Ttheta$ in (\ref{est_Theta}), we can tune $\lambda_3$ by  $k$-fold cross validation over a grid of $\lambda_3$. %We set $k=10$ in our simulation. 

	\subsection{Standardization}\label{sec_standardize}
	
	In steps (\ref{est_F}) and (\ref{est_Theta}) of our estimation procedure, the tuning parameters $\lambda_1$ and $\lambda_3$ depend on $\max_{1\le j\le p}\wh \Sigma_{jj}$ from Theorems \ref{thm_pred} and \ref{thm_Theta}. This dependency comes from the union bounds argument for controlling $\max_{1\le j\le p} \|\X_j^T\P\Eps\|_2$. To tighten the bound in practice, we recommend  standardizing the columns of $\X$ to unit variance. Since the means of $\Y$ and $\X$ do not affect the estimation of $\Ttheta$, one can center both $\X$ and $\Y$ before fitting the model.
	
	%In multivariate regression, standardizing columns of the response matrix $\Y$ such that different responses have unit variance is recommended in practice. In our procedure, such standardization could be beneficial for estimating the conditional mean in (\ref{est_F}) and for estimating $\Ttheta$ in (\ref{est_Theta}), especially when different responses exhibit clearly dissimilar variances. 

	\subsection{Practical usage of $\Ttheta$ for inferring $\Ps$}\label{sec_usage_theta}
	When the parameter $\Ps$ is of primary interest, the information in the parameter $\Ttheta = \Ps P_{\B}^{\perp}$ is still helpful to infer $\Ps$. 
	In the following, we discuss this usage of $\Ttheta$ in two scenarios. The first scenario corresponds to $\Ttheta \approx \Ps$ whence one can use $\wh \Theta$ to estimate $\Ps$. We also provide sufficient conditions for $\Ttheta \approx \Ps$.
	We then suggest further usage of $\Ttheta$ when $\Ttheta \not\approx \Ps$ in the second scenario. Finally, we offer our recommendations to practitioners. 
	
	{\bf Case (1): $\Ttheta \approx \Ps$.} In this case, our estimator of $\Ttheta$ also estimates $\Ps$ consistently. To see when $\Ttheta \approx \Ps$ holds, recall that $\Ps - \Ttheta =  \Ps P_{\B}$. Then $\Ps \approx \Ttheta$ is implied by $\Ps P_{\B}\to 0$ as $m\to \i$. %Condition $\Ps P_{\B}\to 0$ is commonly assumed in the existing SVA literature \citep{Lee2017,McKennan19,wang2017}. 
	The following two lemmas provide different sets of sufficient conditions for $\|\Ps P_{\B}\|_{\ell_\i/\ell_2} =\max_{1\le j\le p}\|P_{\B}\Ps_{j\cdot}\|_2 = o(\sqrt{m})$ and $\|\Ps P_{\B}\|_\i=o(1)$ whence
	\[
	{1\over \sqrt m}\|\Ps - \Ttheta \|_{\ell_\i/\ell_2} = o(1),\qquad \|\Ps  - \Ttheta \|_\i = o(1).
	\]
	Their proofs can be found in Appendix \ref{sec_proofs_lemma_Theta}. 
	Recall that $U\in\RR^{m\times K}$ contains the right singular vectors of $\B\in \RR^{K\times m}$. 
	\begin{lemma}\label{lem_unif_B}
		Suppose the columns of $U$ are uniformly distributed over the families of $K$ orthonormal vectors. Provided that $K = o(m)$, for any $1\le j\le p$ and $1\le \ell \le m$, one has 
		\[
		\left\|P_{\B}\Ps_{j\cdot}\right\|_2^2 = O_p\left(\left\|\Ps_{j\cdot}\right\|_2^2{K\over m}\right),\quad 
		\left|e_\ell^T P_{\B}\Ps_{j\cdot}\right| = O_p\left(\left\|\Ps_{j\cdot}\right\|_2{K\over m}\right).
		\]
		If $\|\Ps\|_\i =O(1)$ holds additionally, then 
		\[
		{1\over m}\left\|P_{\B}\Ps_{j\cdot}\right\|_2^2 = O_p\left({K\over m}\right),\qquad \left|e_\ell^T P_{\B}\Ps_{j\cdot}\right| = O_p\left({K\over \sqrt m}\right).
		\]
	\end{lemma}
	
	Lemma \ref{lem_unif_B} states that when the directions of columns of $U$ are random enough (more specifically, uniformly distributed) and $K=o(m)$, the matrix $P_{\B}$ is incoherent to $\Ps_{j\cdot}$. The uniformity assumption of $U$ is commonly made in the matrix completion literature, under which \cite{Candes,ct} prove that $\max_{1\le j\le m}\|P_{\B}e_j\|_2^2 = O_p(K/m)$ where $\{e_j\}_{1\le j\le m}$ is the canonical basis of $\RR^m$. Our result in Lemma \ref{lem_unif_B} reduces to this existing result when $\Ps_{j\cdot}$ is aligned with canonical vectors.

	\begin{lemma}\label{lem_sparse_psi}
		Suppose $\max_{1\le j\le p}\|\Ps_{j\cdot}\|_{0} \le d$ and $\|\Ps\|_\i \le c$ for some constant $c>0$ and integer $1\le d\le m$. Further assume
		$\lambda_K(m^{-1}\B{\B}^T) \ge c'$ and $\|\B\|_\i\le c''$ for some constants $c', c''>0$.  Then, for any $1\le j\le p$ and $1\le \ell \le m$, one has
		\[
		{1\over m}\left\|P_{\B}\Ps_{j\cdot}\right\|_2^2 \le {(cc'')^2 \over c'}{d^2K \over m^2},\qquad \left|e_\ell^T P_{\B}\Ps_{j\cdot}\right| \le {c(c'')^2\over c'}{K d \over m}.
		\]
	\end{lemma}
	
	From Lemma \ref{lem_sparse_psi}, under certain regularity conditions on $\Ps$ and $\B$, one has $m^{-1}\|\Ps - \Ttheta\|_{\ell_\i/\ell_2}^2 = o(1)$ provided that $d^2K/m^2= o(1)$. This holds when either the rows of $\Ps$ are sufficiently sparse or $K$ is much smaller than $m$. 
	Sparsity of rows of $\Ps$ is commonly assumed in the SVA literature and is practically meaningful in many biological applications, see, for instance, \cite{Gagnon2012,wang2017,McKennan19}. In particular, similar sets of sufficient conditions for $m^{-1}\|\Ps - \Ttheta\|_{\ell_\i/\ell_2}^2 = o(1)$ are given in \cite{McKennan19,wang2017}. When $K = O(1)$ and $d = o(m)$, Lemma \ref{lem_sparse_psi} also ensures that
	$
	    \|\Ps  - \Ttheta\|_{\i} = o(1).
	$
	%This rate tends to zero as long as $m\to\i$ even when the rows of $\Ps$ are dense, that is, $d\asymp m$. Therefore, when $m$ is large and $K$ is small, coupled with conditions in Lemma \ref{lem_sparse_psi}, one should always expect $\Ps \approx \Ttheta$.
	\\

	{\bf Case (2): $\Ttheta \not\approx \Ps$.} We provide two ways of using the estimator of $\Ttheta$ to infer $\Ps$ in this scenario. 
	\begin{enumerate} 
		\item[(i)] Suppose we are interested in $\Ps C$ for some known constraint matrix $C\in\RR^{m\times q}$. Then provided that $P_{\B} C$ is small, one could use $\wh \Theta C$ to estimate $\Ps C$ because 
		$\Ttheta C = \Ps C - \Ps P_{\B} C \approx \Ps C$. In practice, since $P_{\B}$ can be estimated by $\wh P_{\B}$ or $\wt P_{\B}$ (see Section \ref{sec_est_P} for homoscedastic error and Section \ref{sec_hetero_est} for heteroscedastic error), researchers could empirically decide whether or not using $\wh \Theta C$ to estimate $\Ps C$ by comparing the magnitude of $\|\wh P_{\B} C\|_F / \sqrt{mq}$ to a small tolerance level. 
		
		As a simple yet important example, suppose $\B = [B_1 ~ {\bm 0}]$ for some $B_1 \in \RR^{K\times m_1}$ and $1\le m_1 \le m$. From model (\ref{model_1}), this structure of $\B$ implies that there are $m_1$ responses affected by the hidden variables $Z$. Let $S\subseteq \{1,\ldots,m\}$ denote the index set of columns of $B_1$ and write $S^c = \{1,\ldots,m\}\setminus S$. Since the set $S$ can be estimated from the sparsity pattern of $\wh P_{\B}$, we assume $S$ is known for simplicity. 
		Then it is easy to see that, for any $1\le j\le p$,
		\[
		\Ttheta_{j\cdot} = P^{\perp}_{\B} \Ps_{j\cdot} = \begin{bmatrix}
		P_{B_1}^{\perp}\Ps_{jS} \\  \Ps_{jS^c}
		\end{bmatrix},
		\]
		which further implies $\Ttheta_{j\ell} = \Ps_{j\ell}$ for any $\ell \in S^c$ and $1\le j\le p$. Intuitively, since the $\ell$th response is not associated with hidden variables, the parameter $\Psi^*_{\cdot\ell}\in\RR^p$ in the multivariate response regression is identifiable and is indeed identical to our estimand $\Ttheta_{\cdot\ell}$. In this case, for any given constraint matrix $C=[{\bm 0}~C_1]^T\in \RR^{m\times q}$ with $C_1\in \RR^{q\times (m-m_1)}$ and the index set of rows of $C_1^T$ in $C$ being $S^c$, we can use our estimator $\wh \Theta C$ to infer $\Ps C$. 
		
		%In this case, the estimator of $\Theta_{\cdot S^c}$ can be directly used to infer $\Ps_{\cdot S^c}$. 
		
		%Notice that the set $S$ can be estimated from $\wh P_{\B}$. This also suggests to use the estimation of $P_{\B}$ to find responses that are not affected by the hidden variables. 
		
		\item[(ii)] Another usage of $\Ttheta$ is to further infer the non-zero rows of $\Ps$ based on the fact that $\Ttheta_{j\cdot}\ne 0$ implies $\Ps_{j\cdot}\ne 0$. Specifically, for any $j$ such that $\Ttheta_{j\cdot}\ne 0$, the first display in Section \ref{sec_ident} yields 
		$
		F^*_{j\cdot} = \Ps_{j\cdot} + {\B}^T A^*_{j\cdot}.
		$
		When $\Ps_{j\cdot}$ is sufficiently sparse, one could resort to the robust regression 
		%by regression $F_{j\cdot}^*$ on the $K$ right singular vectors of ${\B}^T$ 
		to estimate $\Ps_{j\cdot}$ (see details in \cite{wang2017}). Our procedure yields the index of non-zero rows of $\Ps$ as well as the estimates of $F^*$ and the row space of $B^*$. A full exploration of this approach is beyond the scope of this work and is left for future investigation.\\ 
		
		%(For this usage, our procedure yields the index of non-zero rows of $\Ps$ as well as the critical estimates of both $F^*$ and $B^*$.)
	\end{enumerate}

	In practice, we suggest to first check whether the conditions in Lemmas \ref{lem_unif_B} and \ref{lem_sparse_psi} are reasonable. If this is the case, our estimator $\wh\Theta$ can be directly used for estimating $\Ps$. For example, the uniformity of $U$ in Lemma \ref{lem_unif_B} can be verified by comparing $\|\wh P_{\B} v\|_2^2$ with $\wh K/m$ for some chosen unit vector $v\in \RR^m$. Here $\wh P_{\B}$ is the estimate of $P_{\B}$ and $\wh K$ is the estimated number of hidden variables. For conditions in Lemma \ref{lem_sparse_psi}, one could use $\Lambda_K$, the $K$th eigenvalue of ${\B}^T \Sigma_W \B$, as a surrogate of $\lambda_K(\B {\B}^T)$. The former can be estimated from the $\wh K$th eigenvalue of the residual matrix $\wh \Sigma_{\eps}$, see Section \ref{sec_est_P} for details. Even if there is no prior information on the sparsity of rows of $\Ps$, Lemma \ref{lem_sparse_psi} may still hold if $\wh K$ is much smaller than $m$. 
	%In practice, the sparsity assumption on rows of $\Ps$ can only be judged based on prior knowledge in specific data application.  
	
	When conditions in Lemmas \ref{lem_unif_B} and \ref{lem_sparse_psi} seem questionable, we recommend to apply the procedure in (i) of {\bf Case (2)} to check if $\wh\Theta$ could be used to infer $\Ps C$ for some constraint matrix $C$ with scientific interest. If this is not the case either, one may possibly apply the robust regression in (ii) of {\bf Case (2)} to estimate certain rows of $\Ps$.

	\section{Simulation study}\label{sec_simulation}
	
	In this section, we conduct simulations to verify our theoretical results. As mentioned in the Introduction, the SVA methods such as \citep{Lee2017} require the condition $\Ps P_{\B} \to 0$. To compare with \citep{Lee2017} and other competing methods introduced below, we force $\Ps P_{\B} = 0$ so that $\Ttheta = \Ps$ throughout this section.
	
	%We first compare our proposed approach with other existing methods in different settings for known $K$. We then show the performance of selecting $K$ by using the criterion in (\ref{select_K}) and the permutation test in \cite{PA}. 

	\paragraph{Methods} We consider both HIVE and H-HIVE in Algorithms \ref{alg_1} and \ref{alg_2}. 
	All tuning parameters $\lambda_1$, $\lambda_2$ and $\lambda_3$ are chosen via 10-fold cross validation as described in Section \ref{sec_tuning}. We set the number of iterations $T = 5$ for H-HIVE, as the algorithm converges quickly in our simulation. 
	
	Depending on the setting, we compare our method with competitors from the following list:
	\begin{itemize}[itemsep = 0mm]
		\item Oracle: the estimator from (\ref{est_Theta}) by using $P_{B} =  B^T(B^TB)^{-1}B$  with the true $B$.
		\item Lasso: the group-lasso estimator from R-package \textsf{glmnet}.
		\item Ridge: the multivariate ridge estimator from R-package \textsf{glmnet}.
		\item HIVE-init: $\wh\Psi$ obtained from solving (\ref{est_F}) in step (1) of Algorithm \ref{alg_1}. 
		\item SVA: the surrogate variable analysis summarized in the following three steps: (i) compute $\wh \Theta_{LS} = (\X^T\X)^{-1}\X^T \Y$; (ii) obtain $\wh P$ by the first $K$ right singular vectors of $\Y - \X\wh \Theta_{LS}$; (iii) estimate $\Ttheta$ by $\wh \Theta_{LS}(\bI_m - \wh P)$.\footnote{This procedure is based on \cite{Lee2017}. There are other variants of SVA in the literature, for instance, \cite{wang2017, McKennan19}. Since they have similar performances in our setting, we only consider the aforementioned one. A detailed comparison with other SVA-related procedures is given in Appendix \ref{sec_supp_sim}.} 
		\item OLS: the ordinary least squares estimator $\wh \Theta_{LS} = (\X^T\X)^{-1}\X^T \Y$.
		%\item HS-ratio: $\wt \Theta$ obtained in step(4) by using selected $\wh K$ from (\ref{select_K}).
		%\item HS-PA: $\wt \Theta$ obtained in step(4) by using selected $\wh K$ from the permutation test in \cite{PA}.
	\end{itemize}
	
	The Oracle estimator requires the knowledge of true $B$ and is used as a benchmark to show  the effect of estimating $P_{\B}$ on the estimation of $\Ttheta$ in (\ref{est_Theta}). We also consider HIVE-init, which is used as an initial estimator in Algorithms \ref{alg_1} and \ref{alg_2}, to illustrate the improvement of HIVE (H-HIVE) via (\ref{est_Theta}) (see more discussions in Remark \ref{rem_rate_theta_hat}). 
	
	To make fair comparison, we provide the true $K$ for SVA, HIVE and H-HIVE in Section \ref{sim_1}. We then show the performance of selecting $K$ by using the criterion  (\ref{select_K}) and the permutation test by \cite{PA} in Section \ref{sim_2}. 
	
	\paragraph{Data generating mechanism} We set $K = s_* = 3$ throughout the simulation settings. The design matrix is sampled from $\X_{i\cdot} \sim N_p(0, \Sigma)$ for $1\le i\le n$ where 
	$\Sigma_{j\ell} = (-1)^{j+\ell}\rho^{|j-\ell|}$ for all $1\le j, \ell \le p$.  Under $Z = A^TX + W$, to generate $A$ and $B$, we sample $A_{jk} \sim \eta \cdot N(0.5, 0.1)$ and $B_{k\ell}\sim N(0.1, 1)$ independently for all $1\le j\le p$, $1\le k\le K$ and $1\le \ell \le m$. We use $\eta$ to control the magnitude of $A$ hence the dense matrix $L = AB$. We generate the first $s_*$ rows of $\Theta_{raw}$ by sampling each entry independently from $N(\mu_{\Theta}, \sigma_{\Theta}^2)$ and set the rest rows to $0$. The final $\Theta$ is chosen as $\Theta_{raw}(\bI_m - B^T(B^TB)^{-1}B)$ which has the same row sparsity as $\Theta_{raw}$ and satisfies $\Theta P_{B}= 0$.  For the error terms, we independently generate $\W_{ik}\sim N(0, 1)$  for all $1\le i\le n$ and $1\le k\le K$. For homoscedastic case, $\E_{ij}$ for $1\le i\le n$ and $1\le j\le m$ are i.i.d. realizations of $N(0, 1)$. For heteroscedastic case, we independently generate $\E_{ij}\sim N(0, \tau_j^2)$, where, to vary the degree of heterogeneity, we follow the simulation setting in \cite{hpca} and choose $\tau_j^2 = {m v_j^{\alpha} / \sum_j v_j^{\alpha}}$, where $v_1,\ldots, v_m$ are i.i.d. $\textrm{Unif}[0,1]$.
	This choice of $\tau_j^2$ guarantees $\sum_{j=1}^m\tau_j^2 / m =1$ and $\alpha$ controls the degree of heterogeneity: a larger $\alpha$ corresponds to more heterogeneity.

	\subsection{Comparison with existing methods}\label{sim_1}
	In this section, we compare the performance of Oracle, Lasso, Ridge, SVA, HIVE-init, HIVE and H-HIVE in three different settings: (1) small $p$ and small $m$ ($m = p = 20$); (2) small $p$ and large $m$ ($m=150$, $p=20$); (3) large $p$ and small $m$ ($m = 20$, $p = 150$).  For each setting, we fix $n=100$ and consider both homoscedastic and heteroscedastic cases.  
	
	%Specifically, we set $n = 100$, $m = p = 20$ in the first setting, $n=100$, $m=150$, $p=20$ in the second one and $n = 100$, $m = 20$, $p = 150$ in the third one.  
	We choose $\mu_{\Theta} = 3$ and $\sigma_{\Theta} = 0.1$ and vary $\rho \in \{0, 0.5\}$ across all settings. For the homoscedastic case we vary $\eta \in \{0.1, 0.3, 0.5, \ldots, 1.1, 1.3\}$, while for the heteroscedastic case we vary $\alpha \in \{0, 3, 6, \ldots, 12, 15\}$ and fix $\eta = 0.5$. Within each combination of $\eta$ and $\rho$ (or $\alpha$ and $\rho$), we generate $\X$, $A$, $B$ and $\Theta$ once and generate $100$ replicates of the stochastic errors $\W$ and $\E$. For each method with their estimator $\wh \Theta$ and the prediction $\X\wh F$ (if available), we record the averaged Root Sum Squared Error (RSSE) $\|\wh \Theta - \Theta\|_F$ and the averaged Prediction Mean Squared Error (PMSE)  $\|\X \wh F - \X F\|_F^2 / (nm)$. We only report the results for $\rho = 0.5$ as the ones for $\rho = 0$ are similar.

	\subsubsection{RSSE}
	The averaged RSSE of all methods are reported in Figure \ref{fig_MSE} for homoscedastic cases and Figure \ref{fig_MSE_hetero} for heteroscedastic cases. To illustrate the difference, we take the $\log_{10}$ transformation. 
	
	{\bf Homoscedastic cases:} HIVE dominates the other methods and has the closest performance to the Oracle across all settings. H-HIVE is the second best and has similar performance to HIVE when $p$ is small. This is expected since H-HIVE also works when the errors are homoscedastic. However, when $p$ is large, its performance deteriorates comparing to HIVE as $\eta$ increases such that the dense matrix $L$ has larger magnitude. The reason is that the condition $Rem(P_{\B}) \le c\sqrt{K}$ in Theorem \ref{thm_U_hetero} becomes restrictive for large $p$, small $m$ and large $\eta$ (say $\eta \ge 0.8$), since in this scenario the prediction error gets larger and so does $Rem(P_{\B})$. 
	
	Among the competing methods, when $n > p$ (the first two panels of Figure \ref{fig_MSE}), SVA also has good performance 
	but is still outperformed by HIVE since SVA does not adapt to the sparsity structure of $\Ttheta$. OLS is comparable to Ridge. Lasso has clear advantage over Ridge when the signal is sparse enough, that is, when $\eta$ is small. HIVE-init outperforms both Lasso and Ridge. When $n<p$, SVA and OLS are not well defined and become infeasible in the third panel of Figure \ref{fig_MSE}. HIVE-init has similar performance as Lasso but has larger error when $\eta$ increases. HIVE and H-HIVE dramatically reduce the error of the initial estimator HIVE-init in all setting. This agrees with the theoretical results in Remark \ref{rem_rate_theta_hat}. 
	
	\begin{figure}[ht]
		\centering
		\includegraphics[width = \textwidth]{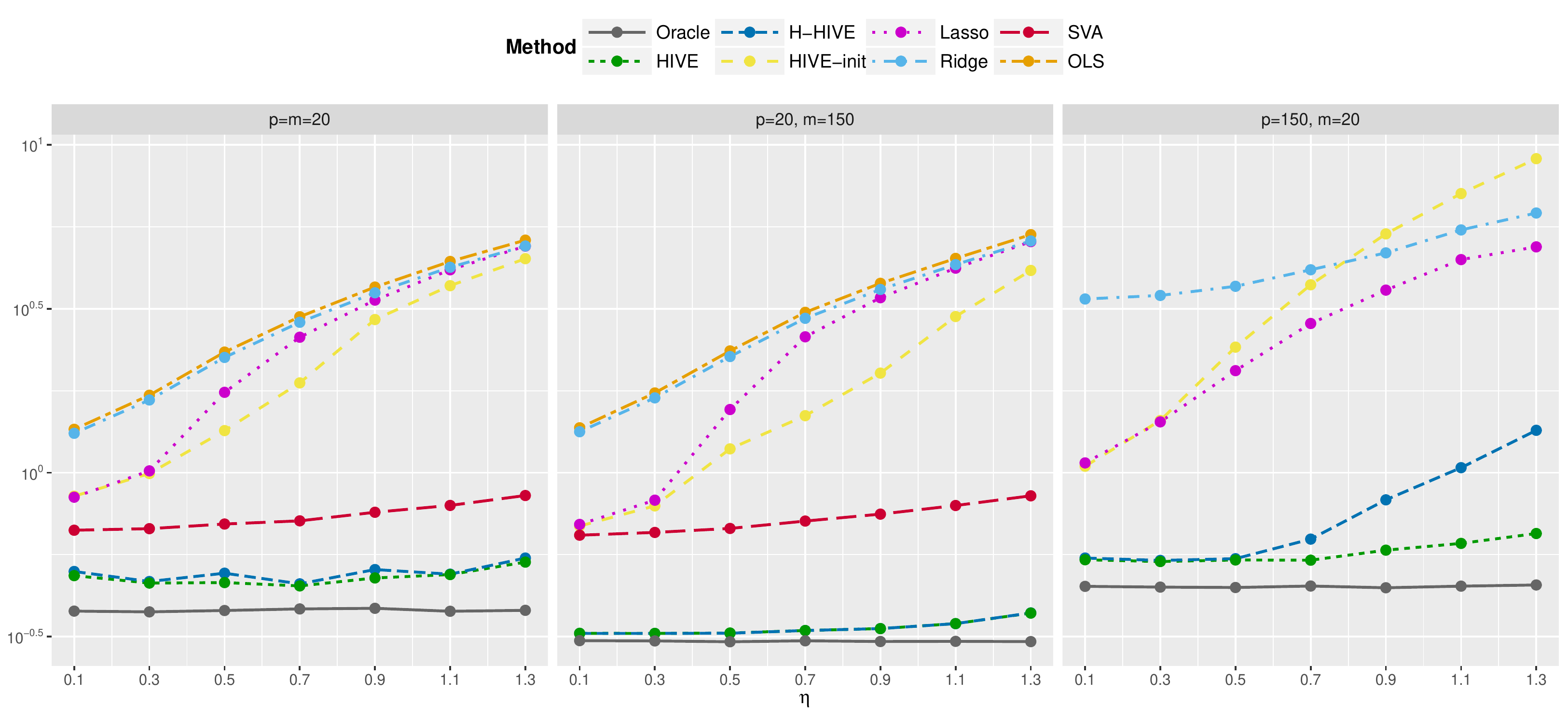} 
		\caption{RSSE under the homoscedastic settings with $n = 100$. }
		\label{fig_MSE}
	\end{figure}

	{\bf Heteroscedastic cases:} Figure \ref{fig_MSE_hetero} shows that H-HIVE, tailored for the heteroscedastic error, has the smallest RSSE among all the methods and its advantage over the second best method, HIVE, becomes evident when $m$ is small (see the first and third panels) and the degree of heteroscedasticity is moderate or large (i.e., $\alpha \ge 9)$. This agrees with our theoretical analysis that the HIVE estimator may not be consistent when $m$ is finite under heteroscedastic errors. It is worth mentioning that when $m$ is large (see the second panel), HIVE is nearly identical to H-HIVE suggesting similar performance between PCA and HeteroPCA. This is expected in light of Remark \ref{rem_robust_PCA}. Finally, Ridge, Lasso and HIVE-init are robust to the degree of heteroscedasticity in all cases, whereas SVA shows inflated RSSE  as the degree of heteroscedasticity ($\alpha$) increases when $m$ is small (see the first panel).

	\begin{figure}[ht]
		\centering
		\includegraphics[width = \textwidth]{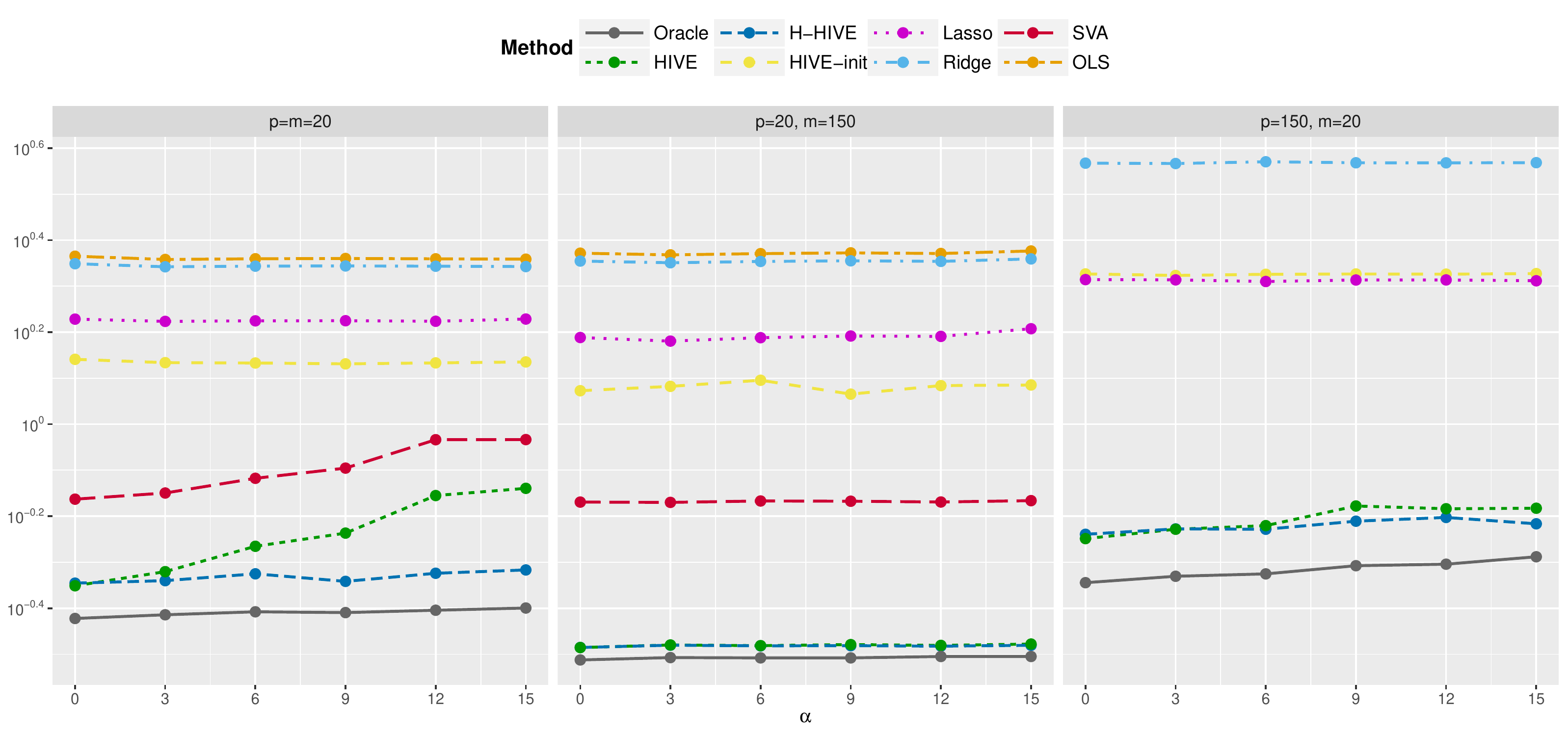} 
		\caption{RSSE under the heteroscedastic settings with $n = 100$. }
		\label{fig_MSE_hetero}
		\vspace{-3mm}
	\end{figure}

	\subsubsection{PMSE}
	
	The PMSE for different methods are reported in Figure \ref{fig_PMSE} for both homoscedastic cases and heteroscedastic cases. Notice that OLS and SVA have the same PMSE and so do HIVE-init, HIVE and H-HIVE. 
	
	{\bf Homoscedastic cases:}  As seen in the top row of Figure \ref{fig_PMSE}, when $n<p$ (the third panel), HIVE has much smaller PMSE than both Lasso and Ridge. This demonstrates the advantage of the proposed procedure in (\ref{est_F}) for prediction.  When $p<n$ (the first two panels), HIVE and Lasso have comparable performance and clearly outperform OLS and Ridge for small $\eta$ (i.e. the signal $\Theta + L$ is approximately sparse). These findings are in line with Theorem \ref{thm_pred} and its subsequent remarks.
	
	{\bf Heteroscedastic cases:} The bottom row of Figure \ref{fig_PMSE} shows that all methods have robust prediction performance under the heteroscedastic cases and the advantage of HIVE (H-HIVE) becomes more evident when $p>n$ (the last panel). %The robustness of HIVE to heteroscedasticity corroborates our results in Theorem \ref{thm_pred}. 
	
	\begin{figure}[ht]
		\centering
		\begin{tabular}{cc} 	\hspace{-3mm}	
			\includegraphics[width = .65\textwidth, height=0.28\textheight]{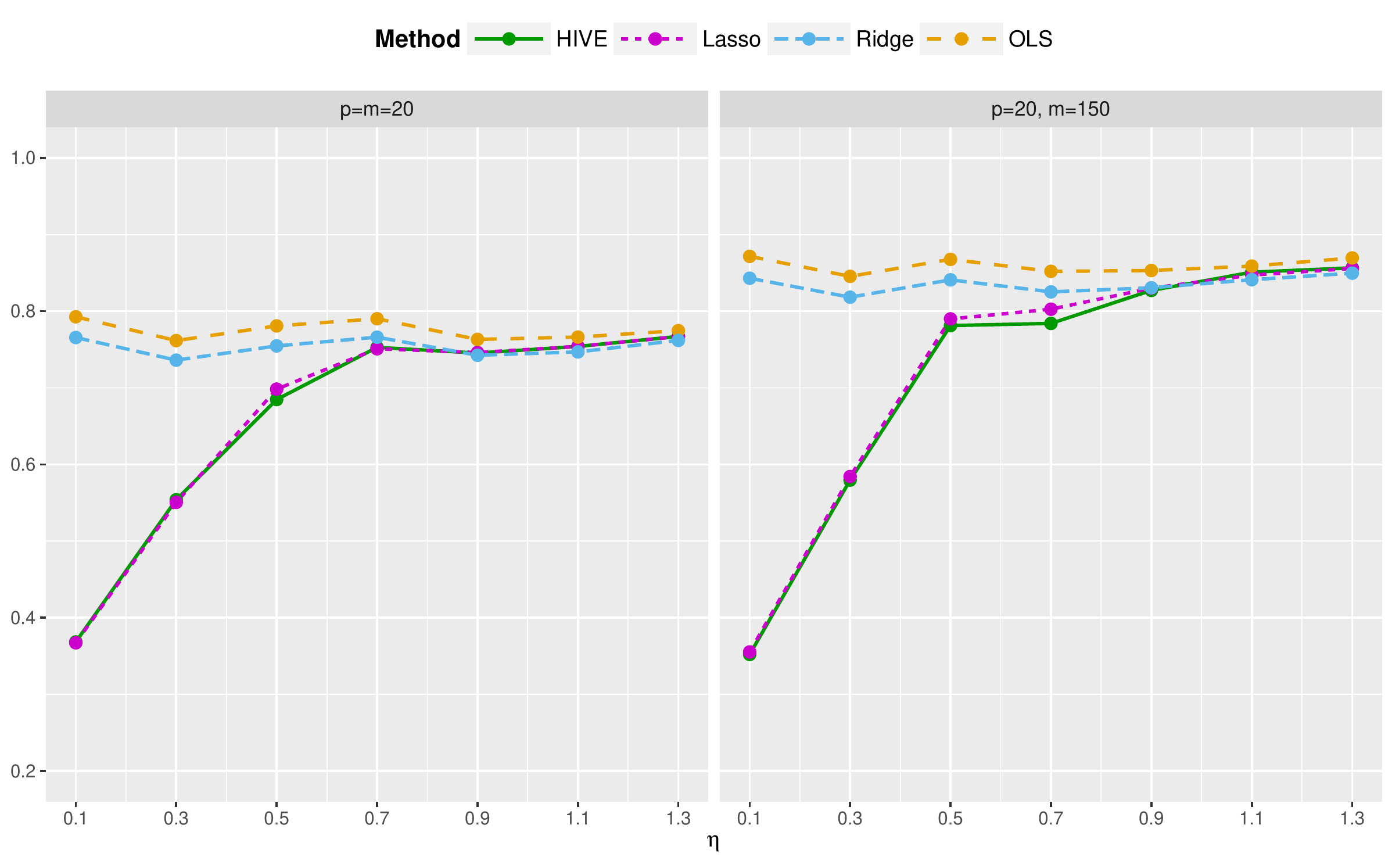} &\hspace{-5mm} 	\includegraphics[width =.35\textwidth, height=0.28\textheight]{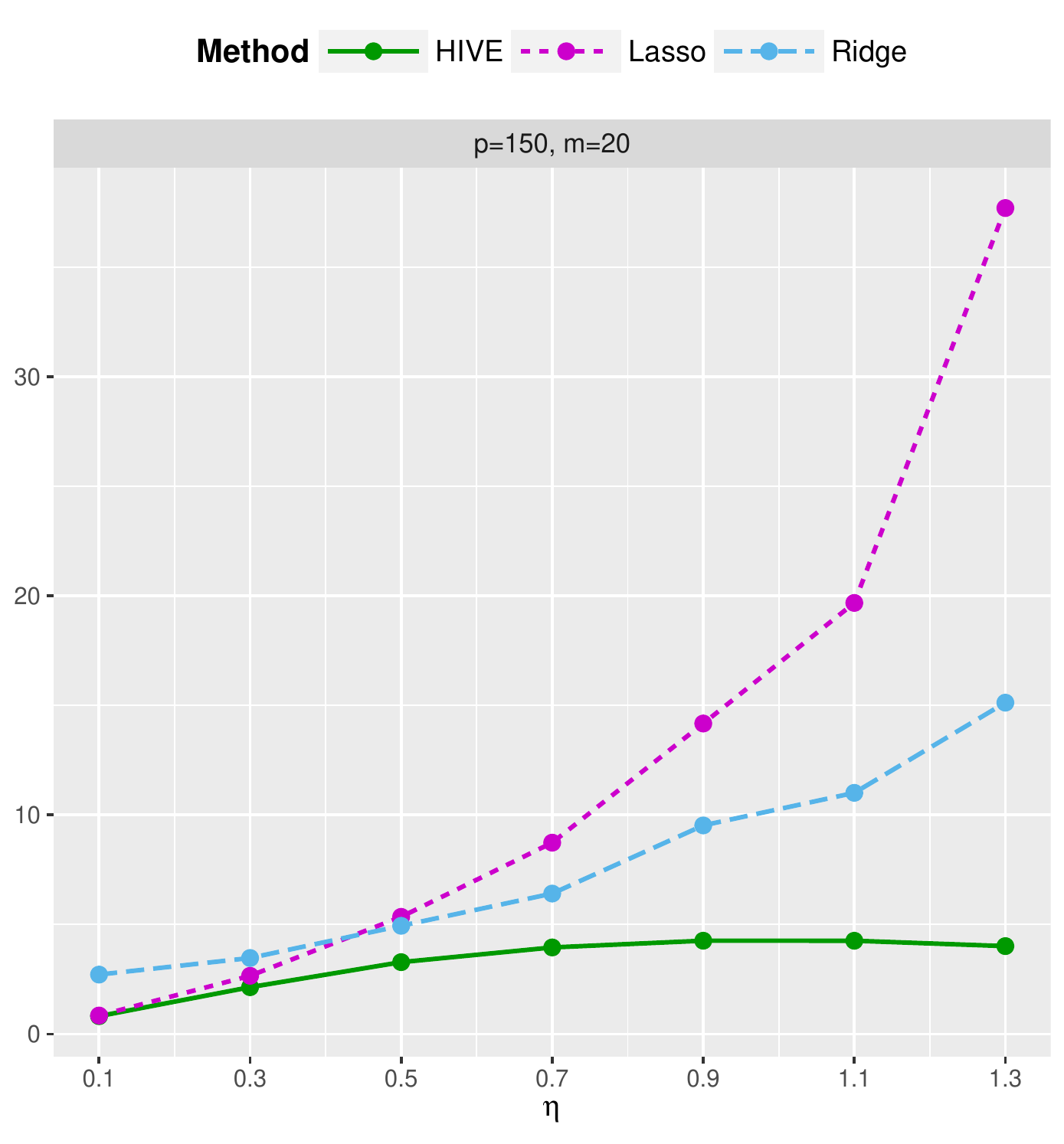}\\
			\includegraphics[width = .65\textwidth, height=0.28\textheight]{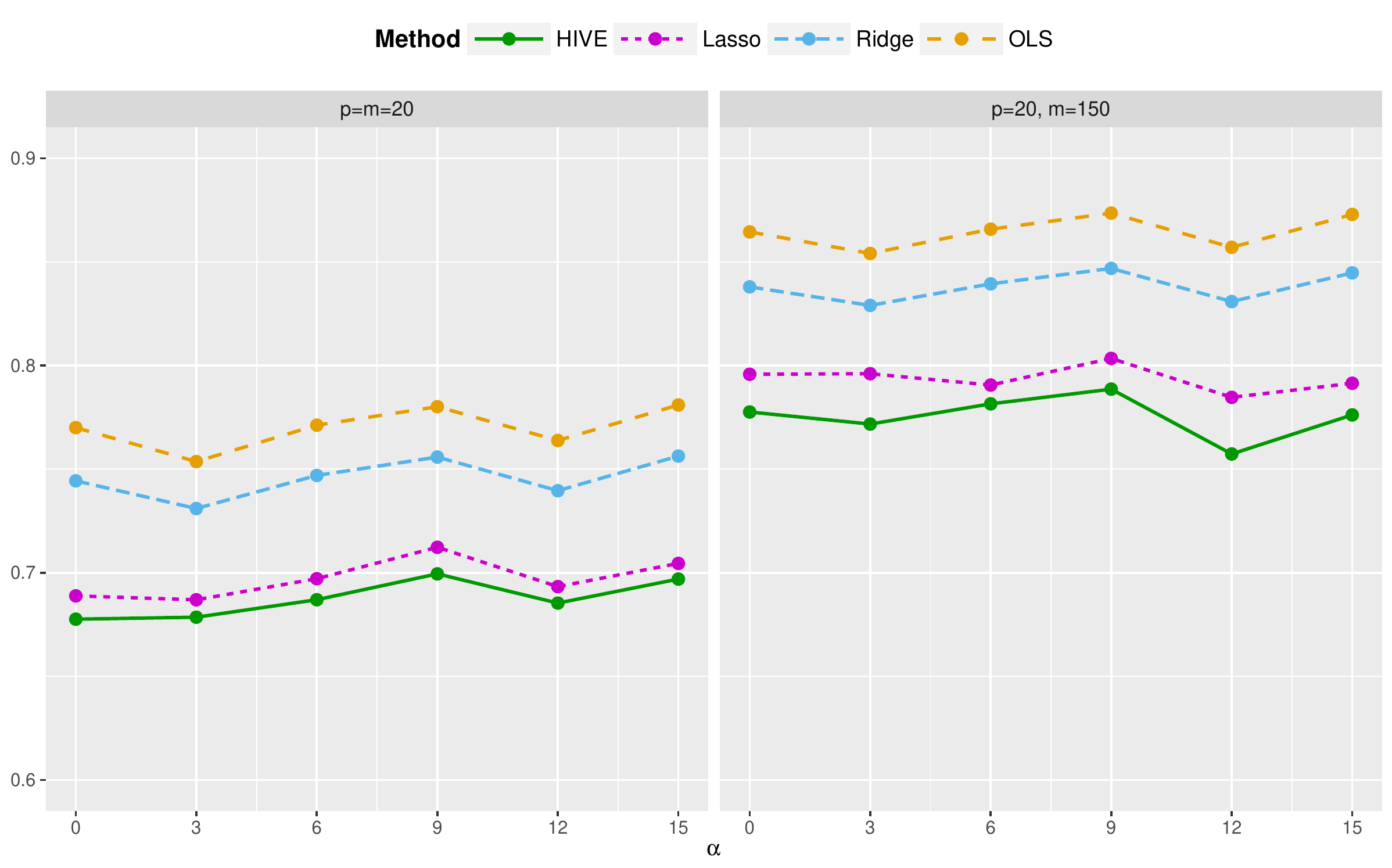} &\hspace{-5mm} 	\includegraphics[width =.35\textwidth, height=0.28\textheight]{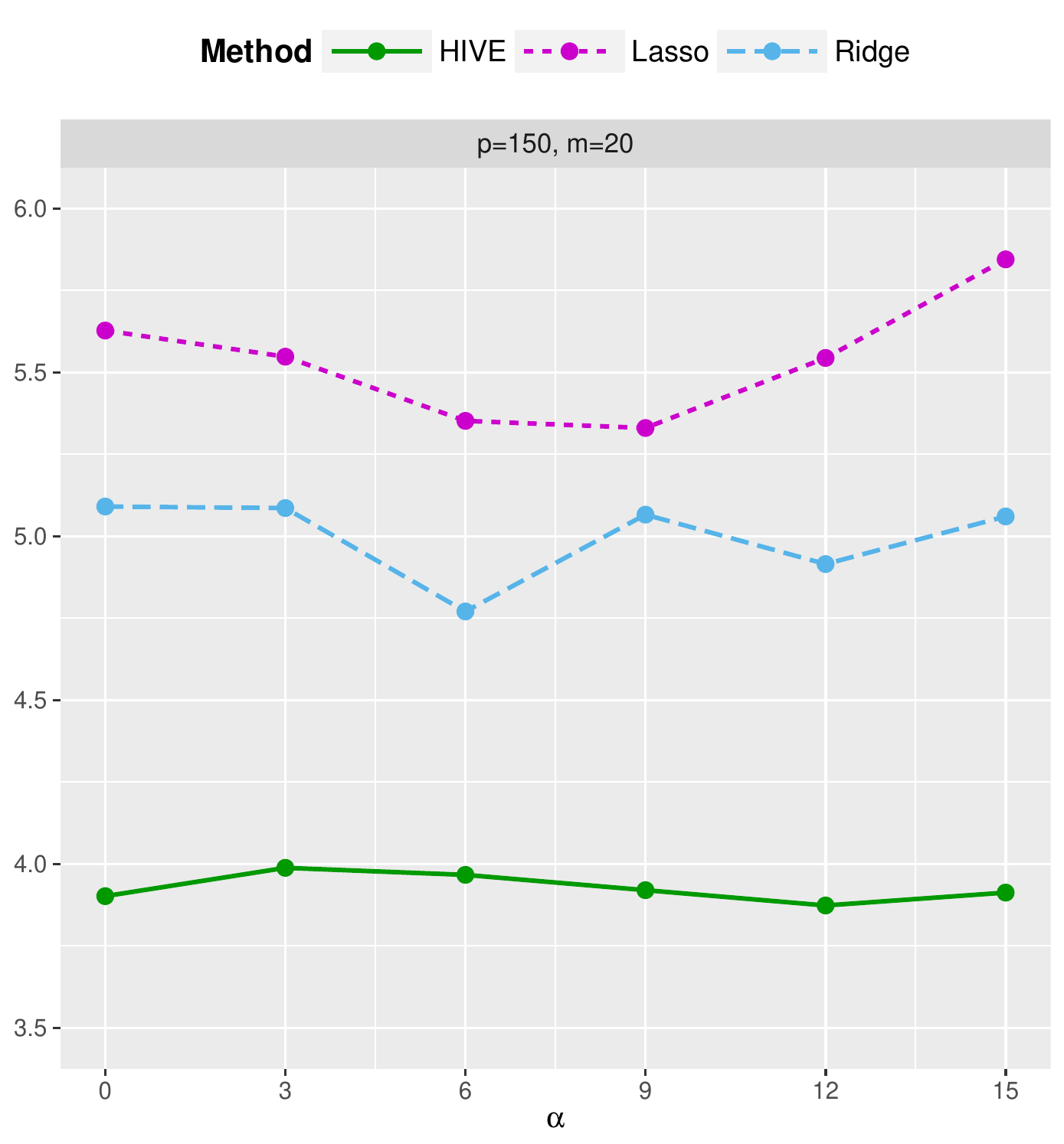}
		\end{tabular}
		\caption{PMSE under the  homoscedastic (the top row) and heteroscedastic (the bottom row) settings with $n = 100$.}
		\label{fig_PMSE}
	\end{figure}

	\subsection{Performance of selecting $K$}\label{sim_2}
	We report our simulation results of selecting $K$ by using (\ref{select_K}) (Ratio) and the permutation test (PA) in \cite{PA}. In the setting of $n = 100$, $m=150$, $p = 20$, both methods select $K$ consistently, which is expected for large $m$ and small $p$. %In particular, this agrees with the results in Theorem \ref{thm_K} for the Ratio method. %Recall that in Theorem \ref{thm_K}, in order to consistently estimate $K$, one needs two conditions  $m\rightarrow\infty$ and the prediction error $R_F / m$ converges to $0$, which often hold  when $m$ is large and $p$ is small. 
	
	We mainly investigate the selection of $K$ in two settings: $n = 100$, $m = 20$, $p = 20$ and $n=100$, $m = 20$, $p = 150$. For each setting,  we fix $\mu_{\Theta}=\sigma_{\Theta} = 1$, $\eta = 0.3$, $\rho = 0.3$  and vary  the signal-to-noise ratio (SNR) defined as 
	$\lambda_{K}(B^T\Sigma_W B) / (m\tau^2)$, where $\tau^2 = 1$ and $B_{kj} \sim N(0.1, 1)$. We choose $\Sigma_W =  \sigma_W^2\bI_K$ with  $\sigma_W \in \{0.1, 0.3, 0.5, \ldots, 1.3, 1.5\}$ such that SNR $\approx \sigma_W^2$.  Recall that the true $K$ is equal to $3$. Figure \ref{fig_K} shows the boxplot of the selected $K$ by using Ratio and PA over $100$ simulations. It is clear that as long as the SNR is large enough, both methods consistently select $K$. By comparing the two panels, we can see that when $p$ is large, we need stronger SNR in order to consistently select $K$. 
	
	In practice, we recommend using PA when $m$ is small, say around $20$. When $m$ is large or moderate, PA becomes computationally expensive due to the implementation of SVD on the permuted data. For this reason, we recommend Ratio for moderate or large $m$.

	\begin{figure}[ht]
		\centering
		\includegraphics[width = .8\textwidth]{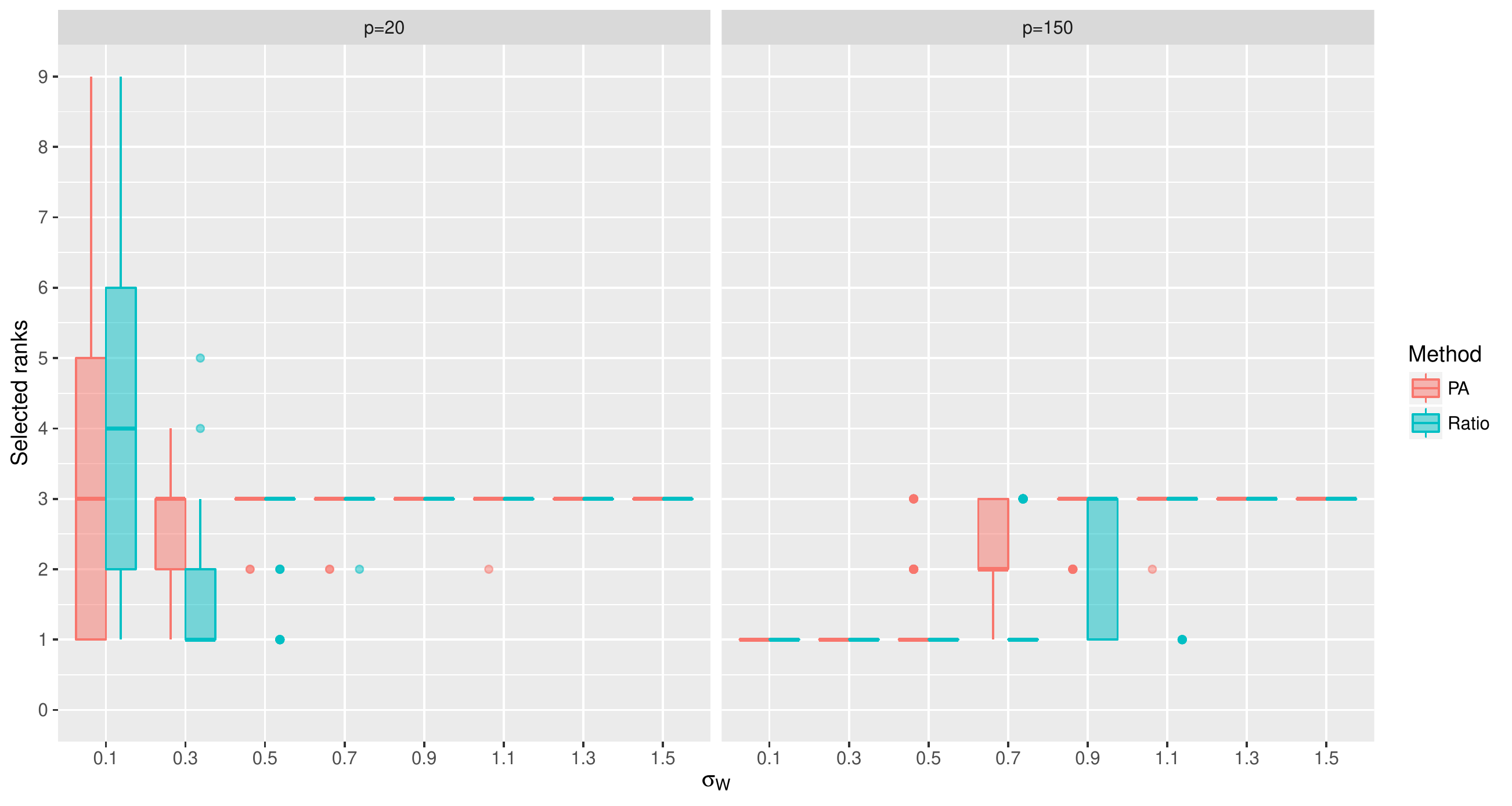}
		\caption{Boxplots of the selected $K$ using PA and Ratio in two settings. }
		\label{fig_K}
		\vspace{-3mm}
	\end{figure}

	\section{Real data application}\label{sec_data}
	
	We apply our procedure, Algorithms \ref{alg_1} and \ref{alg_2}, to two real world datasets: the Norwegian dataset and the yeast cross dataset. Due to the lack of knowledge of the ground truth in real data applications, we only compare the prediction performance of our procedure with several competing methods. %, though prediction is not the main focus of our procedure,  

	\paragraph{Norwegian dataset} This dataset available in \cite{i08} was collected to study the effect of three variables $X_1$, $X_2$ and $X_3$ on the quality of the paper from a Norwegian paper factory. The quality of the paper is measured by $13$ continuous responses while all $X_i$ taking values in $\{-1, 0, 1\}$ represent the location of the design point. In addition to the main effect terms ($X_1, X_2, X_3$), six second order interaction terms ($X_1^2, X_2^2, X_3^2$, $X_1X_2$, $X_1X_3$, $X_2X_3$) were also considered as predictors. In total, the dataset consists of $ n =29$ fully observed observations with $m = 13$ responses and $p = 9$ predictors. The design matrix is centered and standardized to unit variance while the response matrix is centered. 
	
	\cite{bunea2011} showed that the data may exhibit a low-rank structure with estimated rank $K = 3$ via reduced-rank regression. This finding is consistent with \cite{aldrin} based on the smallest leave-one-out cross-validation (LOOCV) error, which is 326.2 (total sum of squared errors), over all possible ranks. A later analysis of \cite{bunea2012} via the sparse reduced-rank regression (SSR) further reduces the LOOCV error to 304.5. Specifically, \cite{bunea2012} estimated the coefficient matrix of the multivariate linear regression by 
	\begin{equation}\label{procedure_bunea}
	\wh F_{k} = \min_{\rank(F)\le k}\|\Y - \X F\|_F^2 + 2\lambda \|F\|_{\l12}
	\end{equation}
	with $k = 3$ and $\lambda > 0$ selected from cross validation. The resulting estimator $\wh F_{k}$ is both low-rank and row-sparse, and selects 6 predictors by excluding the following three terms $X_1^2$, $X_1X_2$ and $X_2X_3$. 
	
	To compare the prediction performance, we applied HIVE in Algorithm \ref{alg_1} to this dataset and the permutation test in Section \ref{sec_select_K} for estimating $K$ as $m$ is small. Our procedure yields $\wh K = 3$. The results from the H-HIVE algorithm are similar and thus omitted. For comparison, we also applied Ridge, group-lasso (Lasso) and SVA to this dataset (note that the prediction of SVA is the same as OLS). The LOOCV errors for all methods are summarized in Table \ref{tab_loocv}. The HIVE algorithm has the smallest LOOCV error among all methods. Thus, our approach yields the most accurate prediction.

	\begin{table}[ht]
		\centering
		\caption{LOOCV errors of HIVE, ridge, group-lasso (Lasso), SVA, reduced-rank regression (RRR) and sparse reduced-rank regression in (\ref{procedure_bunea}) (SRR) on Norwegian dataset}
		\label{tab_loocv}
		\begin{tabular}{l|cccccc}
			\hline 
			Method &  HIVE  &  Ridge &  Lasso & SVA & RRR & SRR\\\hline
			LOOCV error &    288.9 & 324.3 & 317.3 & 338.1  & 326.2 & 304.5\\ 
			\hline
		\end{tabular}
	\end{table}
	
	In addition, the results in Table \ref{tab_loocv} imply that the low-rank structure of the coefficient matrix in RRR and SRR may not be sufficient to model the association between the predictors and responses. 	\cite{bunea2011} showed that the reduced-rank regression by using all $p$ predictors can explain $86.9\%$ of the total variation of $\Y$ quantified by $\tr(\Y^T\X(\X^T\X)^{-1}\X^T\Y)$ (see \cite{i08} for the definition). Note that we can rewrite model (\ref{model}) as a reduced-rank regression of $\Y - \X \Psi$ on $\X$. By replacing $\Psi$ with our estimator $\wh\Psi$, we can show that our model (\ref{model}) can explain $98.6\%$ of the total variation, a much higher percentage than the reduced-rank regression. This implies that our model may provide a better fit to the data than the reduced-rank regression. The main reason is that model (\ref{model}) is able to capture the sparse signal that cannot be explained by the low rank structure. To quantify this statement, we calculate, in Table \ref{tab_1}, the $\ell_2$ norm of rows of our estimator $\wh \Psi +\wh L $ corresponding to all the predictors. As a comparison, we also compute the reduced-rank estimator $\wt L$ by regressing $\Y$ on $\X$ directly. The results are also shown in Table \ref{tab_1}. Similar to the results from the SRR, the estimator $\wt L$ corresponding to the three predictors $X_1^2$, $X_1X_2$ and $X_2X_3$ has small $\ell_2$ norm. However, the association between the three predictors and responses is indeed strong as shown by our estimator $\wh \Psi + \wh L$ when the sparse signal $\wh \Psi$ is taken into account. This suggests that model (\ref{model_1}) can successfully capture both the low rank signal and the sparse signal, whereas the latter is omitted in the (sparse) reduced-rank regression. % which may yield misleading results.

	%Recall that SRR in (\ref{procedure_bunea}) shrinks the coefficients of $X_1^2$, $X_1X_2$ and $X_2X_3$ to 0. Our estimate of the low rank matrix $\wh L$ demonstrates the similar pattern, as the three predictors have smaller $\ell_2$ norm than the other predictors. However, the association between the three predictors and responses becomes more evident in $\wh \Theta + \wh L$, if the sparse signal $\wh \Theta$ is taken into account. This suggests that model (\ref{model_1}) can successfully capture both the low rank signal and the sparse signal, whereas the latter is omitted in the (sparse) reduced-rank regression.

	\begin{table}[ht]
		\centering
		\caption{$\ell_2$ norms of rows of $\wt L$ and $\wh \Psi + \wh L$. The bold numbers correspond to the three excluded predictors, $X_1^2, X_1X_2$ and $X_2X_3$ in \cite{bunea2012}.}
		\label{tab_1}
		\begin{tabular}{l|rrrrrrrrr}
			\hline
			& $X_1$ & $X_2$ & $X_3$ & $X_1^2$ & $X_2^2$ & $X_3^2$ & $X_1X_2$ & $X_1X_3$ & $X_2X_3$ \\ 
			\hline
			$\wt L$ & 1.33 & 0.60 & 1.05 & {\bf 0.28} & 0.44 & 0.64 & {\bf 0.14} & 0.71 & {\bf 0.35} \\ 
			$\wh\Psi + \wh L$ & 1.61 & 0.94 & 1.15 & {\bf 0.59} & 0.56 & 0.78 & {\bf 0.23} & 0.88 & {\bf 0.54} \\ 
			\hline
		\end{tabular}
	\end{table}
	
	\paragraph{Yeast cross dataset}
	
	The yeast cross dataset\footnote{The dataset is downloaded from \url{http://genomics-pubs.princeton.edu/YeastCross_BYxRM/home.shtml}.} consists of 1,008 prototrophic haploid segregants from a cross between a laboratory strain and a wine strain of yeast. This dataset was collected via high-coverage sequencing and consists of genotypes at 30,594 high-confidence single-nucleotide polymorphisms (SNPs) that distinguish the strains and densely cover the genome. There are 46 traits in this dataset corresponding to the measured growth under multiple conditions, including different temperatures, pHs and carbon sources, as well as addition of metal ions and small molecules \citep{yeast_data}. The goal is to study the relationship between genotypes and traits, which could be used for predicting traits or selecting significant genotypes for further scientific investigation \citep{yeast_data}. A multivariate linear regression by regressing traits on genotypes could be suitable for this purpose. However, it is likely that there exist hidden factors that also affect traits. We thus fit our model (\ref{model}) for prediction and variable selection. After removing the segregants with missing values in traits and SNPs which have Pearson correlations above $0.97$, we end up with  $n = 303$ segregants with $m = 46$ traits and $p = 571$ SNPs. 
	
	To evaluate the prediction performance, we randomly split the data into $70\%$ training set and $30\%$ test set. We center and normalize the SNPs in the training set to zero mean and unit variance. Traits in the training set are also centered. The corresponding means and scales from the training set are used to standardize the test set. We then apply the HIVE Algorithm \ref{alg_1} together with group-lasso (Lasso) and Ridge to the training set and evaluate the fitted model on the test set. The test mean square errors (MSE) of Lasso and Ridge are 7.29 and 6.29, respectively, while HIVE has a smaller test MSE 5.92. This suggests that HIVE has better prediction performance than Lasso and Ridge. We then refit the model to the whole dataset and apply HIVE, H-HIVE and Lasso for variable selection. Lasso and HIVE select, respectively, 261 and 259 SNPs with 205 common ones. For H-HIVE, it selects 263 SNPs in which 222 SNPs are identical to those selected from Lasso. The difference of the selected SNPs between HIVE (H-HIVE) and Lasso is due to the fact that Lasso does not account for the potential hidden variables. We expect that the results from HIVE (H-HIVE) may provide new insight on understanding how SNPs are associated with different traits. For instance, further confirmatory analysis such as controlled experiments can be conducted by the investigators to study the effect of the selected SNPs. 
	
	%Finally, we mentioned that the estimated number of hidden factors is $2$ by using (\ref{select_K}) and $10$ by using the permutation test. To take a closer look at this low-dimensional structure, after removing $\X\wh \Theta$ from $\Y$, the rank $2$ reduced-rank estimator by regressing $\Y - \X \wh \Theta$ on $\X$ explains $61.5\%$ variation of $\tr((\Y-\X\wh\Theta)^T\X(\X^T\X)^+\X^T(\Y-\X\wh\Theta))$\footnote{We use $M^+$ to denote the Moore-Penrose inverse of any matrix $M$.} while the rank $10$ reduced-rank estimator explains $83.6\%$ which is not a significant improvement. From this perspective, it is plausible that there exist approximately two hidden variables. 

	\section{Discussion}
	
	In this paper, we study the high-dimensional multivariate response regression model with hidden variables. We establish sufficient and necessary conditions for model identifiability. We propose the HIVE algorithm for estimating the coefficient matrix $\Ttheta$, which is adaptive to the unknown sparsity of $\Ttheta$. The algorithm is further extended to settings with heteroscedastic noise. Theoretically, we establish non-asymptotic upper bounds of the estimation error of our estimator, which are valid for any finite $n$, $p$, $m$ and $K$. 
	
	There are several future directions that are worthy of further investigation. First, it is appealing to study the variable selection property of the proposed algorithm. In this paper, we focus on the adaptive estimation of the coefficient matrix. To establish the variable selection consistency property, a different set of conditions (e.g., minimum signal strength condition) are required. %We aim to investigate this property in some future work. 
	Second, it is of great interest to construct confidence intervals or hypothesis tests for the high-dimensional matrix $\Ttheta$ \citep{guo2020doubly}. The inference results can be further used to control the false discovery rate (FDR) in multiple testing, which is of central importance in many biological applications. We refer to the SVA literature for discussions on the FDR control. Third, 
	%the proposed estimation procedures can be extended to handle a variety of different sparsity patterns of $\Ttheta$. Since we assume that there exists a small subset of common features associated with the responses, we apply the group-lasso penalty to estimate the row-sparse matrix $\Ttheta$.  In other applications, a different sparsity structure of $\Ttheta$ may be more suitable (for instance, column-sparsity or block-sparsity). One can modify our procedure to account for the specific sparsity structures.
	our model (\ref{model_1}) also covers a particular yet important model, the confounding model \citep{cevid2018spectral}, which, in addition to (\ref{model_1}), assumes $X = D'Z + W'$ for some random noise $W'\in\RR^p$ independent of $Z$. As a result, the coefficient $\Ps$ represents the causal effect of $X$ on $Y$. Both our methods and analysis are directly applicable to this case. But it is also worth mentioning that assuming this extra structure brings the advantage of predicting $Z$ by using $X$ when $p$ is large. The predicted $Z$ in turn could potentially be used to improve the estimation of $\Ps$, especially when $m$ is small. We do not pursue this direction in this paper and leave it to future research.
	
	%as mentioned in the Introduction, our model (\ref{model_1}) covers the structural equation model as a particular case. The latter assumes (\ref{modelZ}) additionally and as a result, the coefficient $\Ps$ represents the causal effect of $X$ on $Y$. Both our methods and analysis are directly applicable to this case. But it is also worth mentioning that assuming the extra structure in (\ref{modelZ}) brings the advantage of predicting $Z$ by using $X$ when $p$ is large. The predicted $Z$ in turn could potentially be used to improve the estimation of $\Ps$. We do not pursue this direction in this paper and leave it to future research.

	\section*{Acknowledgement}
	We thank the Associate Editor and two anonymous referees for their many insightful and helpful suggestions.
	%Yang Ning was partially supported by NSF-DMS 1854637. Xin Bing was supported in part by NSF-DMS 1712709.
	
	The first author was supported in part by NSF Grant DMS 1712709 and NSF Grant DMS 2015195. 
	
	The second author was supported in part by NSF Grant DMS-1854637 and CAREER Award DMS-1941945.
	
	\begin{supplement}
		\sname{Supplement to ``Adaptive Estimation in Multivariate Response Regression with Hidden Variables''}\label{supp}
		%\stitle{supplementary proofs}
		%\slink[url]{http://www.e-publications.org/ims/support/dowload/imsart-ims.zip}
		\slink[doi]{TBA}
		\sdescription{The supplementary document includes the proofs, the comparison with reduced-rank estimator and additional numerical results.}
	\end{supplement}
	
	%\vspace{.5cm}
	\bibliography{ref}

	\newpage

	\appendix
	
	Appendix \ref{sec_sintheta} states a variant of the robust $\sin\Theta$ Theorem and its proof. Appendix \ref{sec_main_proof} contains the proofs of the main results. The proofs of auxiliary results and technical lemmas are collected in Appendix \ref{sec_proof_auxiliary}. Appendix \ref{sec_comp_rr} contains both theoretical and empirical comparisons of the multivariate ridge estimator and the reduced-rank estimator. Supplementary simulation results and an alternative way of selecting $(\lambda_1, \lambda_2)$ can be found in Appendix \ref{sec_supp_sim}.

	\section{A variant of the robust $\sin\Theta$ theorem}\label{sec_sintheta}	
	We state a variant of Theorem 3 (Robust $\sin\Theta$ theorem) in \cite{hpca}. Let $N, M$ and $Z$ be $m\times m$ deterministic symmetric matrices satisfying  
	\begin{equation}\label{model_decomp}
	N = M + Z
	\end{equation}
	where $N$ is the observation matrix while $M$ is the matrix of interest with $\rank(M) = K$. Let $V\in \RR^{m\times K}$ denote the first $K$ eigenvectors of $M$ with non-increasing eigenvalues denoted as $\lambda_1(M) \ge \lambda_2(M) \ge \cdots \ge \lambda_K(M)$. Further let $\wh V$ be the output of Algorithm \ref{alg_hpca} applied to $N$. Define $\sin\Theta(\wh V, V) := \wh V_{\perp}^T V$ where $\wh V_{\perp} \in \RR^{m \times (m-K)}$ is the orthogonal complement of $\wh V$ such that $[\wh V, \wh V_{\perp}]$ is the complete orthogonal matrix. The following theorem provides upper bounds for the Frobenius norm of $\sin\Theta(\wh V, V)$. For any matrix $D$, let $G(D)$ be the matrix with diagonal entries equal to those of $D$ and off-diagonal entries equal to zero. Let $\Gamma(D)= D-G(D)$.  %This is different from bounding from above the operator norm of $\sin\Theta(\wh V, V)$ as studied in \cite{hpca}. 
	
	\begin{thm}\label{thm_rsintheta}
		Suppose $M \in \RR^{m\times m}$ is a rank-$K$ symmetric matrix and $V$ consists of its first $K$ eigenvectors. Assume there exists a universal constant $c > 0$ such that 
		\begin{equation}\label{cond_ic_prime}
		\max_{1\le j\le m}\left\|e_j^TV\right\|_2^2\cdot {\lambda_1(M) \over \lambda_K(M)} \le c
		\end{equation}
		and $\|\Gamma(Z)\|_F \le c \sqrt{K}\lambda_K(M)$. 
		Then the output $\wh V$ of HeteroPCA($N, K, T$) in Algorithm \ref{alg_hpca} with $T = \Omega\bigl(1 \vee \log {\sqrt{K}\lambda_K(M)\over \|\Gamma(Z)\|_F}\bigr)$ satisfies 
		\[
		\|\sin\Theta(\wh V, V)\|_F  ~ \lesssim ~ {\|\Gamma(Z)\|_F \over \lambda_K(M)} \wedge \sqrt K.
		\]
	\end{thm}
	
	\begin{proof}[Proof of Theorem \ref{thm_rsintheta}]
		The proof follows the same arguments as that of Theorem 7 (General robust $\sin\Theta$ theorem) in \cite{hpca} by taking the sparsity set  $\mathcal{G} = \{(i, i): 1\le i\le m\}$ and $b = \eta = 1$. We will use the same notations as \cite{hpca}. 
		
		Define $T_0 = \|\Gamma(N-M)\|_F = \|\Gamma(Z)\|_F$ and $K_t = \|N^{(t)} - M\|_F$ for $t = 0, 1, \ldots,$. Note that  $\|H\|_F \le \|\Gamma(H)\|_F + \|G(H)\|_F$ for all matrix $H \in \RR^{m\times m}$. 
		We then revisit the three steps in the proof of Theorem 7 in \cite{hpca} by only restating the main results and differences from \cite{hpca}. We use $V$ and $V^{(t)}$ in lieu of $U$ and $U^{(t)}$ in the original proofs of \cite{hpca} for all $t\ge 0$. 
		
		Step 1: The initial errors satisfy
		\begin{align*}
		K_0 &\le \|\Gamma(Z)\|_F + \|G(P_VMP_V)\|_F\\
		& \le \|\Gamma(Z)\|_F + \|M\|_F \max_i \|e_i^TV\|_2\\
		& \le \|\Gamma(Z)\|_F + c\sqrt{K}\lambda_K(M)
		\end{align*}
		for some sufficiently small constant $c>0$
		where the first inequality follows from the same arguments in \cite{hpca}, the second inequality uses the modified Lemma \ref{lem_1_new} and the third inequality uses $\|M\|_F \le \sqrt{K}\|M\|_{op} = \sqrt{K}\lambda_1(M)$ together with condition (\ref{cond_ic_prime}). 
		
		Step 2:  From the proof of \cite{hpca}, one needs to upper bound the two terms on the right hand side of the following display
		\[
		K_t \le \|\Gamma(N^{(t)}) - M\|_F + \|G(N^{(t)} - M)\|_F.
		\]
		The first term satisfies 
		\begin{equation}\label{bd_t0}
		\|\Gamma(N^{(t)}) - M\|_F = \|\Gamma(N-M)\|_F = T_0
		\end{equation}
		for all $t\ge 0$.  The second term can be upper bounded by (see display (6.15) in \cite{hpca})
		\begin{align*}
		\|G(N^{(t)} - M)\|_F &~ \le ~ \|G(P_V(N^{(t-1)} - M))\|_F  + \|G(P_{V^{(t-1)_{\perp}}} M)\|_F\\
		&\qquad  + \|\left(P_{V^{(t-1)}} - P_V\right)(N^{(t-1)} - M)\|_F.
		\end{align*}
		To bound them separately, we have 
		\begin{align}\label{bd_t1}\nonumber
		\|G(P_V(N^{(t-1)} - M))\|_F  &~ \le~  \|N^{(t-1)} - M\|_F \max_i \|e_i^TV\|_2\\
		&~= ~K_{t-1}\max_i \|e_i^TV\|_2
		\end{align}
		by Lemma \ref{lem_1_new}, and 
		\begin{align}\label{bd_t2}\nonumber
		\|G(P_{V^{(t-1)_{\perp}}} M)\|_F  &~ \le ~ \|G(P_{V^{(t-1)_{\perp}}} MP_V)\|_F\\ \nonumber
		& ~ \le ~ \|P_{V^{(t-1)_{\perp}}}M\|_F \max_i \|e_i^TV\|_2\\\nonumber
		&~ \le  ~2\|N^{(t-1)} - M\|_F  \max_i \|e_i^TV\|_2 \\
		&~ = ~ 2K_{t-1}\max_i \|e_i^TV\|_2
		\end{align}
		by using Lemma \ref{lem_1_new} in the second inequality and using Lemma 7 in \cite{hpca} in the third inequality. Finally, the proof of Theorem 7 in \cite{hpca} shows that 
		\begin{align*}
		\|P_{V^{(t-1)}} - P_V\|_{op} & \le  {4\|N^{(t-1)} - M\|_{op} \over \lambda_K(M)} \wedge 1 \le {4K_{t-1} \over \lambda_K(M)}
		\end{align*}
		which further gives 
		\begin{equation}\label{bd_t3}
		\|\left(P_{V^{(t-1)}} - P_V\right)(N^{(t-1)} - M)\|_F \le \|P_{V^{(t-1)}} - P_V\|_{op} \|N^{(t-1)} - M\|_F \le {4K^2_{t-1} \over \lambda_K(M)}.
		\end{equation}
		Collecting (\ref{bd_t0}) -- (\ref{bd_t3}) yields 
		\[
		K_t \le T_0+ 3K_{t-1}\max_i \|e_i^TV\|_2 + {4K^2_{t-1}\over \lambda_K(M)}
		\]
		
		Step 3:  Finally, by using the same induction arguments in \cite{hpca}, under (\ref{cond_ic_prime}) and $\|\Gamma(Z)\|_F = T_0 \le c\sqrt{K}\lambda_K(M)$ with sufficiently small $c$,  one can show that for all $t\ge 0$, 
		\[
		K_t \le 2T_0  + {\sqrt{K}\lambda_K(M) \over 2^{t+4}}.
		\]
		Therefore, for all $t \ge \Omega(1 \vee \log({\sqrt{K}\lambda_K(M) / T_0}))$, one has $K_t \le 3T_0$. The proof is completed by invoking the variant of Davis-Kahan's $\sin\Theta$ theorem \citep{Davis-Kahan-variant} applied to $N^{(t)}$ and $M$ with $d = s = K$, $r = 1$ and $\lambda_{K+1}(M) = 0$.
	\end{proof}

	\medskip 
	The following lemma is a variant of Lemma 1 in \cite{hpca} and it provides upper bounds for the Frobenius norms of some diagonal matrices. Recall that $G(A)$ has the same diagonal elements with $A$ but all off-diagonal entries equal to zero. 
	\begin{lemma}\label{lem_1_new}
		For any two orthogonal matrices $U, V\in \RR^{m\times K}$, let $P_U = UU^T$ and $P_V = VV^T$. Then for any matrix $A \in \RR^{m\times m}$, we have 
		\begin{align*}
		&\|G(P_UA)\|_F^2 \le \|A\|_F^2 \max_{1\le i\le m} \left\|e_i^TU\right\|_2^2,\\
		&\|G(AP_V)\|_F^2 \le \|A\|_F^2 \max_{1\le j\le m} \left\|e_j^TV\right\|_2^2,\\
		& \|G(P_UAP_V)\|_F^2 \le  \|A\|_F^2\left(\max_{1\le i\le m} \left\|e_i^TU\right\|_2^2  \wedge \max_{1\le j\le m} \left\|e_j^TV\right\|_2^2\right).
		\end{align*}
	\end{lemma}
	\begin{proof}
		Since $G(A)$ has non-zero entries only on the diagonal, by writing $A = (A_1, \ldots, A_m)$, one has 
		\begin{align*}
		\|G(P_UA)\|_F^2  = \sum_{i=1}^{m}\left(e_i^T UU^T A_i \right)^2 &\le \max_{1\le i\le m} \|e_i^TU\|_2^2 \cdot \sum_{i=1}^{m} A_i^TUU^T A_i\\
		&= \max_{1\le i\le m} \|e_i^TU\|_2^2\cdot \tr(A^TUU^T A)\\
		&\le \max_{1\le i\le m} \|e_i^TU\|_2^2\cdot \|A\|_F^2
		\end{align*}
		where we used Cauchy-Schwarz inequality in the first line and the inequality $\tr(A^TUU^TA) \le \|A\|_F^2 \|UU^T\|_{op} \le \|A\|_F^2$ in the last line. Since $G(P_UA)$ is symmetric to $G(AP_V)$, the second result follows. The last result follows by using the first two results together with $\|P_V\|_{op}\le 1$ and $\|P_U\|_{op} \le 1$.
	\end{proof}
	\medskip

	\section{Main proofs}\label{sec_main_proof}
	
	\subsection{Proof of Proposition \ref{prop_ident_nece}: $\Ps$ is not identifiable}\label{sec_proof_ident_nece}
	Under model (\ref{model_1}), we will show that there exists $\Psi_1\ne \Psi_2$ such that they yield the same likelihood. 
	For simplicity, we suppress the super script $*$ for related parameters and use lower case letters $x, y, z$ to denote the realizations of random variables $X, Y, Z$. Since the joint density of $(X,Y)$ is factorized as 
	\begin{align*}
	f_{(Y,X)}(y, x) = f_{Y|X}(y|x) f_X(x) = f_X(x)  \int f_{Y|X,Z}(y|x,z)f_{Z|X}(z|x)dz, 
	\end{align*}
	under model (\ref{model_1}), we can write the likelihood of $(Y,X)$ as
	\begin{align}\label{eq_density}
	f_{(Y,X)}\left(y,x; \Psi, B, \mu, \nu\right) &= f_X(x) \int f_E\left(y-\Psi^Tx-B^Tz; \mu\right)f_{Z|X}(z; \nu)dz
	\end{align}
	where $f_E$ denotes the density of $E$ parametrized by $\mu$ and $f_{Z|X}$ is the conditional p.d.f. of $Z=z|X=x$ parametrized by $\nu$. We emphasize that $\nu = \nu(x)$ depends on $X=x$ as $Z$ is correlated with $X$ as $A^*\neq 0$. For any $\Psi$ such that $\Psi P_{B} \ne 0$ (this choice exists as $\rank(B) = K$),  write 
	$$
	\wt\Psi = \Psi - \Psi P_B = \Psi - \Psi B^T(BB^T)^{-1}B,\qquad
	\wt Z = Z + (BB^T)^{-1}B\Psi^T X.
	$$
	Note that $\wt \Psi \ne \Psi$ as $\Psi P_{B} \ne 0$. We prove the following to conclude the proof
	\[
	f_{(Y,X)}\left(y,x; \Psi, B, \mu, \nu\right) = f_{(Y,X)}\left(y,x; \wt\Psi, B, \mu, \wt\nu\right)
	\]  
	where $\wt \nu$ parametrizes the p.d.f of $\wt Z | X$. 
	
	To proceed, write $\Delta = \Psi B^T(BB^T)^{-1}$ such that $\wt Z = Z + \Delta^T X$. Observe that the density functions of $Z|X$ and $\wt Z|X$ only differ by a shift of the mean. We then deduce 
	\[
	f_{\wt Z|X}\left(z; \wt \nu\right) = f_{Z|X}\left(
	z - \Delta^T x; \nu
	\right),
	\]
	from which we further obtain 
	\begin{align*}
	&f_{(Y,X)}\left(y,x; \wt \Psi, B, \mu, \wt \nu\right)\\ 
	&= f_X(x) \int f_E\left(y-\wt \Psi^Tx-B^Tz; \mu\right)f_{\wt Z|X}\left(z; \wt \nu\right)dz\\
	&= f_X(x) \int f_E\left(y-\wt \Psi^Tx-B^Tz; \mu\right)f_{Z|X}\left(z-\Delta^Tx; \nu\right)dz\\
	& = f_X(x) \int f_E\left(y-\wt \Psi^Tx-B^Tt-B^T\Delta^Tx; \mu\right)f_{Z|X}\left(t; \nu\right)dt\\
	& = f_X(x) \int f_E\left(y- \Psi^Tx-B^Tt; \mu\right)f_{Z|X}\left(t; \nu\right)dt\\
	& =  f_{(Y,X)}\left(y,x; \Psi, B, \mu, \nu\right).
	\end{align*}
	We use (\ref{eq_density}) in the first and last equality, set $t = z - \Delta^T x$ to derive the third equality and use 
	$$\wt \Psi^T x = \Psi^T x - B^T(BB^T)^{-1}B\Psi^Tx = \Psi^T x - B^T\Delta^T x$$
	to arrive at the fourth equality. This completes the proof.\qed

	\subsection{Proofs of Propositions \ref{prop_ident} and \ref{prop_ident_hetero}: identifiability for both homoscedastic and heteroscedastic cases}\label{sec_proof_prop_ident}
	\begin{proof}[Proof of Proposition \ref{prop_ident}]
		When $\Ps P_{\B} + A^*\B = 0$, the uniqueness of $\Ttheta$ follows trivially. We only prove case (2). Model (\ref{model}) implies 
		\[
		F^* = \Ps + A^*B^* = [\Cov(X)]^{-1}\Cov(X,Y),
		\]
		from which we can identify 
		$$
		\Sigma_{\eps} := \Cov(\eps) = \Cov(Y - (F^*)^TX).
		$$
		By further using $\Sigma_E = \tau^2\bI_m$, we  have 
		\[
		\Sigma_{\eps} =(B^*)^T\Sigma_W B^* +\Sigma_E =  (B^*)^T\Sigma_W B^* + \tau^2\bI_m
		\]
		Recall that $P_{\B}$ denotes the projection onto the row space of $B^*$. One can identify it from
		\[
		P_{\B} = U_K U_K^T
		\]
		where $U_K$ are the eigenvectors of $\Sigma_\eps$ corresponding to the largest $K$ eigenvalues. Projecting $Y$ onto $P_{\B}^{\perp} = \bI_m - P_{\B}$ gives 
		\[
		P_{\B}^{\perp}Y = (\Ps P_{\B}^{\perp})^T X + P_{\B}^{\perp}E = (\Ttheta)^T X + P_{\B}^{\perp}E.
		\]
		We then conclude 
		$$
		\Ttheta = [\Cov(X)]^{-1}\Cov(X,P_{\B}^{\perp}Y)
		$$
		which is uniquely defined. 
	\end{proof}
	
	\begin{proof}[Proof of Proposition \ref{prop_ident_hetero}]
		The proof of Proposition \ref{prop_ident_hetero} uses the same arguments as that of Proposition \ref{prop_ident} except the identifiability of $P_{\B}$, the projection onto the row space of $B^*$. After identifying $\Sigma_{\eps}$, note that 
		\[
		\Sigma_{\eps} = (B^*)^T\Sigma_W B^*  + \diag(\tau_1^2,\ldots, \tau_m^2).
		\]
		Under (\ref{cond_ic}), applying Theorem \ref{thm_rsintheta} with $N = \Sigma_{\eps}$, $M =  (B^*)^T\Sigma_W B^*$, $Z = \diag(\tau_1^2,\ldots, \tau_m^2)$ and $T \to \infty$ identifies $P_{\B}$. Hence the identifiability of $\Ttheta$ follows from the same arguments in the previous proof. 
	\end{proof}

	\subsection*{Proof of Proposition \ref{prop_ident_W}: the necessity of $\rank(\Sigma_W) = K$}
	We drop the superscripts $*$ for simplicity. 
	Under model
	\[
	Y = (\Theta + \Psi P_{B} + AB)^T X + B^TW+ E,
	\]
	with $\Theta = \Psi P_{B}^{\perp}$, suppose $\EE[W|X] = 0$ and $\Psi P_{B} + AB \ne 0$.
	We will show that $\Theta$ is not identifiable if $\rank(\Sigma_W)<K$.   Towards this end, we shall construct  $(\Psi, A, B)  \ne (\wt \Psi, A, \wt B)$ such that 
	\[
	\wt \Psi + A\wt B = \Psi + AB,\quad  \wt B^T W = B^T W, \text{ a.s.}
	\]
	meanwhile 
	\begin{equation}\label{eq_target}
	\wt \Psi P_{\wt B} + A\wt B \ne 0, \quad \wt \Theta = \wt \Psi P_{\wt B}^{\perp }\ne \Psi P_B^{\perp} = \Theta.
	\end{equation}
	Suppose $\rank(\Sigma_W) < K$. Then there exists a subspace $\mathcal{S} \subseteq \RR^{K}$ with $\rank(\mathcal{S}) = K - \rank(\Sigma_W)$ such that
	\begin{equation}\label{disp_D_Sigma_W}
	v^T\Sigma_W v = 0,\qquad \text{for any }v\in \mathcal{S}.
	\end{equation}
	Let $P_{\mathcal{S}}\in \RR^{K\times K}$ denote the projection matrix onto $\mathcal{S}$ and $P_{\mathcal{S}}^{\perp} = \bI_K - P_{\mathcal{S}}$. 
	
	Write
	\[
	\wt B = P_{\mathcal{S}}^{\perp}B = B - P_{\mathcal{S}}B,\qquad \wt \Psi = \Psi + A P_{\mathcal{S}}B.
	\]
	Note that (\ref{disp_D_Sigma_W}) implies $\wt B^TW = B^T W$ a.s. and 
	$$
	\wt \Psi + A \wt B = \Psi +  A P_{\mathcal{S}}B + A P_{\mathcal{S}}^{\perp}B = \Psi + AB.
	$$
	It remains to show (\ref{eq_target}). We choose $\Psi$ such that $\Psi P_{B} = 0$ which also implies $\Psi P_{\wt B} = 0$ and $A B\ne 0$ from the constraint $\Psi P_{B} + AB \ne 0$. Then
	\begin{align*}
	\wt \Psi P_{\wt B} + A\wt B &= \Psi  P_{\wt B}+  A P_{\mathcal{S}} B P_{\wt B} + A\wt B\\
	&=A P_{\mathcal{S}} B P_{\wt B} + A\wt B  P_{\wt B}\\
	& =AB  P_{\wt B}.
	\end{align*}
	Furthermore, $\Psi P_{B} = \Psi P_{\wt B} = 0$ gives
	\begin{align*}
	\wt \Psi P_{\wt B}^{\perp}  - \Psi P_B^{\perp} &= \left(\Psi+ AP_{\mathcal{S}}B \right)P_{\wt B}^{\perp} - \Psi = AP_{\mathcal{S}}B P_{\wt B}^{\perp} = ABP_{\wt B}^{\perp}.
	\end{align*}
	Since there exists various choices of $A$ such that $AB  P_{\wt B} \ne 0$ and $AB  P_{\wt B}^{\perp}\ne 0$ (for instance, any $A$ with $\rank(A) = K$ suffices), this concludes that $\Ttheta$ is not identifiable, hence completes the  proof.
	\qed

	\subsection{Proof of Theorem \ref{thm_pred}: in-sample prediction risk}\label{sec_proof_thm_pred}
	Before proving Theorem \ref{thm_pred}, we first prove Lemma \ref{lem_solution} as it provides a simpler analytical  expression  for the later proof. 
	
	\begin{proof}[Proof of Lemma \ref{lem_solution}]
		From (\ref{est_F}), for any fixed $\Psi$, we have 
		\begin{align*}
		\wh L(\Psi) &= \arg\min_{L} {1\over n}\|\Y - \X\Psi - \X L\|_F^2 + \lambda_2\|L\|_F^2\\
		&= (\X^T\X + n\lambda_2\bI_p)^{-1}\X^T(\Y - \X \Psi).
		\end{align*}
		Plugging this into (\ref{est_F}) yields
		\begin{align*}
		\wh\Psi &= \arg\min_{\Psi} {1\over n}\|\Y - \P\Y + \P \X\Psi - \X\Psi\|_F^2 +\lambda_1 \|\Psi\|_{\l12}+ \lambda_2\|\wh L(\Psi)\|_F^2\\
		& = \arg\min_{\Psi} {1\over n}\|\Q(\Y - \X\Psi)\|_F^2 +\lambda_1 \|\Psi\|_{\l12}\\
		&\hspace{2cm}+ \lambda_2\|(\X^T\X + n\lambda_2\bI_p)^{-1}\X^T(\Y - \X \Psi)\|_F^2.
		\end{align*}
		Since 
		\begin{align*}
		&{1\over n}\|\Q(\Y - \X\Psi)\|_F^2 + \lambda_2\|(\X^T\X + n\lambda_2\bI_p)^{-1}\X^T(\Y - \X \Psi)\|_F^2 \\
		&= \tr \left\{
		(\Y-\X\Psi)^T\left[ {1\over n}\Q^2 + \lambda_2\X(\X^T\X + n\lambda_2 \bI_p)^{-2}\X^T
		\right](\Y - \X\Psi)
		\right\}
		\end{align*}
		and 
		\begin{align*}
		& {1\over n}\Q^2 + \lambda_2\X(\X^T\X + n\lambda_2 \bI_p)^{-2}\X^T \\
		&= {1\over n}\biggl[
		\bI_p + \X(\X^T\X + n\lambda_2 \bI_p)^{-1}\X^T\X(\X^T\X + n\lambda_2 \bI_p)^{-1}\X^T\\&
		\qquad -2\X(\X^T\X + n\lambda_2 \bI_p)^{-1}\X^T+ n\lambda_2 \X(\X^T\X + n\lambda_2 \bI_p)^{-2}\X^T
		\biggr]\\
		&= {1\over n}\biggl[
		\bI_p -\X(\X^T\X + n\lambda_2 \bI_p)^{-1}\X^T
		\biggr]\\
		&= {1\over n}\Q
		\end{align*}
		by using $n\lambda_2 \X(\X^T\X + n\lambda_2 \bI_p)^{-2}\X^T =\X(\X^T\X + n\lambda_2 \bI_p)^{-1}(n\lambda_2\bI_p)(\X^T\X + n\lambda_2 \bI_p)^{-1}\X^T$ in the last line, 
		we obtain 
		\[
		\wh\Psi = \arg\min_{\Psi} {1\over n}\|\Q^{1/2}(\Y - \X\Psi)\|_F^2 +\lambda_1 \|\Psi\|_{\l12}.
		\]
		This proves (\ref{crit_Theta}). As a result, we have 
		\[
		\X\wh L(\wh \Psi) = \X (\X^T\X + n\lambda_2\bI_p)^{-1}\X^T(\Y - \X \wh \Psi) = \P\Y - \P \X \wh \Psi
		\]
		hence 
		\[
		\X\wh F = \X\wh L(\wh\Psi) + \X\wh\Psi = \P\Y + \Q\X\wh\Psi.
		\]
		The proof is complete. 
	\end{proof}
	
	To prove Theorem \ref{thm_pred}, recall that, for any $(\Psi_0, L_0)$ such that $F^* = \Psi_0 + L_0$, 
	\begin{align}\label{error_decomp}\nonumber
	\X\wh F - \X F^* &=   \X \wh L - \X L_0 + \X\wh \Psi - \X \Psi_0\\
	&=P_{\lambda_2}(\Y- \X\Psi_0) - \X L_0+ Q_{\lambda_2}(\X\wh \Psi - \X \Psi_0).
	\end{align}
	The result of Theorem \ref{thm_pred} follows by invoking Lemmas \ref{lem_ridge} -- \ref{lem_lasso} and noting that 
	\[
	{1\over n}\|\X\wh F- \X F^*\|_F^2 \le {2\over n}\|P_{\lambda_2}(\Y- \X\Psi_0) - \X L_0\|_F^2 + {2\over n}\| Q_{\lambda_2}(\X\wh \Psi - \X \Psi_0)\|_F^2
	\]
	from the basic inequality $(a+b)^2 \le 2(a^2 + b^2)$. We then proceed to upper bound the two terms on the right hand side separately. 
	\begin{lemma}\label{lem_ridge}
		Under conditions in Theorem \ref{thm_pred},  with probability $1-\epsilon$, 
		\begin{align*}
		{1\over n}\|\P(\Y - \X\Psi_0) - \X L_0\|_F^2 &\le 2\|\Q\|_{op} \cdot \lambda_2 {\rm tr}\left[
		L_0^T\wh\Sigma (\wh \Sigma + \lambda_2\bI_p)^{-1}L_0
		\right]\\
		& \quad + {2V_\eps \over n}\left[
		\sqrt{{\rm tr}(\P^2)}+\sqrt{2\|\P^2\|_{op}\log(m/\epsilon)}
		\right]^2
		\end{align*}
		where $V_\eps$ is defined in (\ref{def_V_eps}).
	\end{lemma}
	\begin{proof}
		By $\Y = \X\Psi_0 + \X L_0 + \Eps$ and the basic inequality $(a+b)^2 \le 2(a^2 + b^2)$, we have 
		\[
		\|\P(\Y - \X\Psi_0) - \X L_0\|_F^2 \le 2\|\P\Eps\|_F^2 + 2\|\Q \X L_0\|_F^2.
		\]
		Note that the second term satisfies
		\begin{align}\label{disp_QXL}
		\|\Q \X L_0\|_F^2 \le \|\Q\|_{op}\|\Q^{1/2}\X L_0\|_F^2 =\|\Q\|_{op} \cdot n\lambda_2 \tr\left[
		L_0^T\wh\Sigma (\wh \Sigma + \lambda_2\bI_p)^{-1}L_0
		\right]
		\end{align}
		by using Fact \ref{fact_QX} in the second equality. The result follows by  invoking Lemma \ref{lem_P_eps} for the term $\|\P\Eps\|_F^2$. 
	\end{proof}
	
	\medskip
	
	\begin{lemma}\label{lem_lasso}
		Under conditions in Theorem \ref{thm_pred},  with probability $1-\epsilon'$, 
		\begin{align*}
		&{1\over n}\|\Q\X(\wh \Psi - \Psi_0)\|_F^2\\
		& \le 4\|\Q\|_{op}\cdot \max\left\{
		4\lambda_2 {\rm tr}\left[
		L_0^T\wh\Sigma (\wh \Sigma + \lambda_2\bI_p)^{-1}L_0
		\right], ~ {\lambda_1^2 \over [\kappa_1(1/2, \Psi_0, \lambda_1, \lambda_2]^2}
		\right\}
		\end{align*} 
	\end{lemma}
	\begin{proof}
		Write $\wt \Y = \Q^{1/2}\Y$ and $\wt \X = \Q^{1/2}\X$. Starting with  (\ref{crit_Theta}),  we have
		\[
		{1\over n}\|\wt \Y - \wt\X \wh\Psi\|_F^2 + \lambda_1 \|\wh \Psi\|_{\l12} \le {1\over n}\|\wt \Y - \wt\X \Psi_0\|_F^2 + \lambda_1 \| \Psi_0\|_{\l12}.
		\] 
		Let $\langle A, B\rangle  = \tr(A^TB)$ for any commensurate matrices $A$ and $B$.  
		By writing $\Delta := \wh\Psi -\Psi_0$ and noting that $\wt \Y = \wt \X \Psi_0 + \wt \X L_0 + \wt \Eps$ with $\wt \Eps = \Q^{1/2}\Eps$, standard arguments yield 
		\begin{align}\nonumber
		&{1\over n}\|\wt \X \Delta\|_F^2\\\nonumber
		&\le {2\over n}\left|
		\langle \wt \X L_0 + \wt \Eps, \wt \X \Delta\rangle
		\right|+ \lambda_1 \left(
		\|\Psi_0\|_{\l12} - \|\wh \Psi\|_{\l12}
		\right)\\\nonumber
		&\le {2\over n}\|\wt \X L_0\|_F\|\wt \X \Delta\|_F + {2\over n}\left| \langle \wt \Eps , \wt \X\Delta \rangle \right| +  \lambda_1 \left(
		\|\Psi_0\|_{\l12} - \|\Delta + \Psi_0\|_{\l12}
		\right)\\\label{disp_Xtd_Delta}
		&\le {2\over n}\|\wt \X L_0\|_F\|\wt \X \Delta\|_F + {2\over n}\max_{1\le j\le p}\|\wt \X_j^T\wt \Eps\|_2 \|\Delta\|_{\l12}+  \lambda_1 \left(
		\|\Psi_0\|_{\l12} - \|\Delta + \Psi_0\|_{\l12}
		\right)
		\end{align}
		where we use Cauchy-Schwarz in the second inequality and the following display to derive the third inequality,
		\[
		\left| \langle \wt \Eps , \wt \X\Delta \rangle \right| = \left|
		\sum_{j=1}^p \wt \X_j^T\wt\Eps\Delta_{j\cdot}
		\right|\le \max_{1\le j\le p}\|\wt \X_j^T\wt \Eps\|_2 \|\Delta\|_{\l12}.
		\]
		On the event
		\begin{equation}\label{def_event}
		\cE := \left\{
		\max_{1\le j\le p} \|\X_j^T\Q\Eps\|_2  \le {n\over 4}\cdot \lambda_1
		\right\},
		\end{equation}
		by $\|\wt \X_j^T\wt \Eps\|_2 = \|\X_j^T\Q \Eps\|_2$, we further have 
		\begin{align}\label{disp_XDelta}
		{1\over n}\|\wt \X \Delta\|_F^2\left(
		1 - {2\|\wt \X L_0\|_F \over \|\wt \X\Delta\|_F}
		\right) \le \lambda_1\left(
		\|\Psi_0\|_{\l12} - \|\Delta + \Psi_0\|_{\l12}+{1\over 2}\|\Delta\|_{\l12}
		\right)
		\end{align}
		Notice that 
		\begin{equation}\label{disp_XL0}
		\|\wt \X L_0\|_F^2 = \tr\left[
		L_0^T \X^T \Q\X  L_0
		\right] = n\lambda_2 \tr\left[
		L_0^T\wh\Sigma (\wh \Sigma + \lambda_2\bI_p)^{-1}L_0
		\right]
		\end{equation}
		by using Fact \ref{fact_QX}. 
		When $\|\wt \X\Delta\|_F \le 4\|\wt \X L_0\|_F$, we obtain the desired result from (\ref{disp_XL0}). It suffices to consider the case $\|\wt \X\Delta\|_F \ge 4\|\wt \X L_0\|_F$. Display (\ref{disp_XDelta}) then implies
		\[
		{1\over n}\|\wt \X \Delta\|_F^2 \le 2\lambda_1\left(
		\|\Psi_0\|_{\l12} - \|\Delta + \Psi_0\|_{\l12}+{1\over 2}\|\Delta\|_{\l12}
		\right),
		\]
		from which we conclude $\Delta \in \R(1/2, \Psi_0, \lambda_1, \lambda_2)$ defined in (\ref{def_R}).
		Invoking condition (\ref{RE_pred}) with $c = 1$ gives
		\[
		{1\over \sqrt n}\|\wt \X \Delta\|_F \le {2\lambda_1 \over \kappa_1(1/2, \Psi_0, \lambda_1, \lambda_2)}.
		\] 
		Therefore, on the event $\cE$, by combining with (\ref{disp_XL0}), we have 
		\begin{equation}\label{bd_pred_X_tilde}
		{1\over n}\|\wt \X \Delta\|_F^2 \le \max\left\{
		16\lambda_2 \tr\left[
		L_0^T	\wh\Sigma (\wh \Sigma + \lambda_2\bI_p)^{-1}L_0
		\right], {4\lambda_1^2 \over \kappa_1^2(1/2, \Psi_0, \lambda_1, \lambda_2)}
		\right\}.
		\end{equation}
		Since the choice of $\lambda_1$ in (\ref{rate_lbd1}) together with Lemma \ref{lem_XQeps} implies 
		$\PP(\cE) = 1-\epsilon'$, we conclude the proof by invoking (\ref{rate_lbd1}) in the above display. 
	\end{proof}

	\subsection{Proof of Theorem \ref{thm_Theta}: convergence rate of $\|\wt\Theta - \Ttheta\|_{\l12}$}\label{sec_proof_thm_Theta}
	We work on the event $\cE_{\B} := \{\|\wh P - P_{\B}\|_F\le c\xi_n\}$  for some constant $c>0$. Define 
	\begin{equation}\label{def_R5}
	Rem_5 = C\left\{{1\over \sqrt n}\|\X F^*\|_{op} +  \Lambda_1^{1/2}\left(1+ \sqrt{K \over n}\right)\right\}\xi_n
	\end{equation}
	for some constant $C>0$.
	We will prove that for any $\lambda_3 \ge \bar \lambda_3$, the solution $\wt\Theta$ from (\ref{est_Theta}) satisfies
	\begin{align}\label{rate_fit}
	&{1\over \sqrt{n}}\|\X\wt \Theta - \X\Ttheta\|_F  
	\le\max\left\{4Rem_5, ~ {3\lambda_3\sqrt{s_*}\over \kappa(s_*, 3)} \right\},\\\label{rate_Theta}
	&\|\wt \Theta - \Ttheta\|_{\l12}
	\lesssim \max\left\{{Rem_5^2\over \lambda_3}, ~ {\lambda_3 s_*\over \kappa^2(s_*, 4)} \right\}
	\end{align}
	with probability $1- \eps - 2e^{-c''K}$ for some constant $c''>0$. Then the result of Theorem \ref{thm_Theta} follows immediately by noting that $Rem_5 = \wt \lambda_3 \sqrt{s_*}/\kappa(s_*,4)$. 
	
	We proceed  to prove (\ref{rate_fit}) and (\ref{rate_Theta}). Pick any $\lambda_3 \ge \bar \lambda_3$.
	Starting from (\ref{crit_Theta}), by writing $\Delta := \wt \Theta - \Ttheta$, standard arguments yield
	\begin{align*}
	{1\over n}\|\X\Delta \|_F^2
	&\le {2\over n}\left|
	\langle \Y\wh P^{\perp} -\X\Ttheta, \X\Delta\rangle
	\right| + \lambda_3\left(
	\|\Ttheta\|_{\l12} - \|\wt\Theta\|_{\l12}
	\right)\\
	&\le {2\over n}\left|
	\langle \E\wh P^{\perp}, \X\Delta \rangle \right| + {2\over n}\left|\langle (\X F^* + \W B^*)\wh P^{\perp} -\X\Ttheta, \X\Delta\rangle
	\right|\\
	&\qquad+\lambda_3\left(
	\|\Delta_{S_*\cdot}\|_{\l12} - \|\Delta_{S_*^c \cdot}\|_{\l12}
	\right)
	\end{align*}
	with $S_*:= \{1\le j\le p: \|\Ttheta_{j\cdot}\|_2 \ne 0\}$. Since 
	\[
	\left|
	\langle \E\wh P^{\perp}, \X\Delta \rangle \right| \le \|\Delta\|_{\l12}\max_{1\le j\le p}\|\X_j^T \E \wh P^{\perp}\|_2,
	\]
	on the event $\cE'$ defined as
	$$
	\left\{\max_{1\le j\le p} \|\X_j^T\E \wh P^{\perp}\|_2 \le n\lambda_3 / 4\right\} \bigcap \left\{
	{1\over \sqrt{n}}\left\| (\X F^* + \W B^*)\wh P^{\perp} -\X\Ttheta\right\|_{F} \le Rem_5
	\right\},
	$$
	by using $|\langle M, N\rangle| \le \|M\|_F \|N\|_F$ for any commensurate matrices, we obtain
	\begin{equation}\label{disp_XDelta_prime}
	{1\over n}\|\X\Delta\|_F^2
	\le {2\over \sqrt n}\|\X\Delta\|_F\cdot Rem_5 + {\lambda_3 \over 2 }\left(
	3\|\Delta_{S_*\cdot}\|_{\l12} - \|\Delta_{S_*^c \cdot}\|_{\l12}
	\right).
	\end{equation}
	By rearranging terms, we have 
	\begin{align*}
	{1\over n}\|\X\Delta\|_F^2\left(
	1 - {2Rem_5 \over \|\X\Delta\|_F/\sqrt{n}}
	\right) \le{\lambda_3 \over 2 }\left(
	3\|\Delta_{S_*\cdot}\|_{\l12} - \|\Delta_{S_*^c \cdot}\|_{\l12}
	\right).
	\end{align*}
	When $\|\X\Delta\|_F/\sqrt{n} \le 4Rem_5$,   (\ref{rate_fit}) holds. When $\|\X\Delta\|_F/\sqrt{n} \ge 4Rem_5$, we have 
	\[
	{1\over n}\|\X\Delta\|_F^2\le {\lambda_3 }\left(
	3\|\Delta_{S_*\cdot}\|_{\l12} - \|\Delta_{S_*^c \cdot}\|_{\l12}
	\right).
	\]
	Hence $\Delta \in \C(S_*, 3) \subseteq \C(S_*,4)$. Invoking (\ref{RE_X}) with $s = s_*$ and $\alpha = 3$ yields 
	\begin{align*}
	{1\over n}\|\X\Delta\|_F^2\le 3\lambda_3
	\|\Delta_{S_*\cdot}\|_{\l12}
	& \le 3\lambda_3\sqrt{s_*}\|\Delta_{S_*\cdot}\|_F\le {3\lambda_3\sqrt{s_*} \over \kappa(s_*, 3)}{1\over \sqrt{n}}\|\X\Delta\|_F,
	\end{align*}
	which implies the first result on the event $\cE' \cap \cE_{\B}$. 
	
	To show (\ref{rate_Theta}), note that $\kappa(s_*,4)>0$ and consider two cases:\\
	(1) When $\Delta\in \C(S_*, 4)$, from the definition of $\kappa(s_*, 4)$, one has 
	\begin{align*}
	\|\Delta\|_{\l12} & = \|\Delta_{S_*\cdot}\|_{\l12} + \|\Delta_{S_*^c \cdot}\|_{\l12}\\
	&\le 5\|\Delta_{S_*\cdot}\|_{\l12}\\
	&\le 5\sqrt{s_*} \|\Delta_{S_*\cdot}\|_F\\
	&\le{5\sqrt{s_*}\over \kappa(s_*, 4)}{1\over \sqrt{n}}\|\X\Delta\|_F.
	\end{align*}
	(2) When $\Delta \notin \C(S_*, 4)$,  we have 
	$
	4 \|\Delta_{S_*\cdot}\|_{\l12} < \|\Delta_{S_*^c \cdot}\|_{\l12}
	$
	by definition.  Plugging this into display (\ref{disp_XDelta_prime}), we have 
	\begin{align*}
	{1\over n}\|\X\Delta\|_F^2
	\le {2\over \sqrt n}\|\X\Delta\|_F\cdot Rem_5 -{\lambda_3 \over 8} \|\Delta_{S_*^c \cdot}\|_{\l12}.
	\end{align*}
	It implies $\|\X\Delta\|_F/\sqrt{n} \le 2Rem_5$ and 
	\[	
	\lambda_3\|\Delta_{S_*^c \cdot}\|_{\l12} \le   {16\over \sqrt{n}}\|\X\Delta\|_F\cdot Rem_5
	\]
	We thus obtain
	\begin{align*}
	\|\Delta\|_{\l12} & \le {5 \over 4}\|\Delta_{S_*^c \cdot}\|_{\l12}\le 20{Rem_5 \over \lambda_3} {1\over \sqrt n}\|\X\Delta\|_F \le 40{Rem_5^2 \over \lambda_3}.
	\end{align*}
	Combining these two cases and invoking (\ref{rate_fit}) give the desired result on the event $\cE' \cap \cE_{\B}$. Finally, invoking Lemmas \ref{lem_U} and \ref{lem_XE} yield $\PP(\cE') \ge 1-\epsilon - 2m^{-c''K}$. This completes the proof. \qed

	\bigskip
	
	The following lemma gives that the probability of $\cE'$ on the event $\cE_{\B}$.  Recall that $Rem_5$ is defined in (\ref{def_R5}).
	\begin{lemma}\label{lem_U}
		Under conditions of Theorem \ref{thm_Theta}, on the event $\cE_{\B}= \{\|\wh P - P_{\B}\|_F\le c\xi_n\}$ for some constant $c>0$, the following holds with probability $1-2e^{-c'K}$ for some constant $c'>0$, 
		\[
		{1\over \sqrt{n}}\left\| (\X F^* + \W B^*)\wh P^{\perp} -\X\Ttheta\right\|_F \le Rem_5.
		\]
	\end{lemma}
	\begin{proof}[Proof of Lemma \ref{lem_U}]
		By noting that $(\X F^* +\W B^*)P^{\perp}_{\B} = \X\Ttheta$, we have 
		\begin{align*}
		&{1\over \sqrt{n}}\left\| (\X F^* + \W B^*)\wh P^{\perp} -\X\Ttheta\right\|_F \\
		& =	{1\over \sqrt{n}}\left\| (\X F^* + \W B^*)(\wh P^{\perp} -P_{\B}^{\perp})\right\|_F\\
		&\le	{1\over \sqrt{n}}\| (\X F^* + \W B^*)\|_{op}\left\|\wh P^{\perp} -P_{\B}^{\perp}\right\|_F\\
		&\le {1\over \sqrt{n}}\Bigl(
		\|\X F^*\|_{op}+ \|\W B^*\|_{op}
		\Big)c\xi_n,
		\end{align*}
		where we have used $\|MN\|_F\le \|M\|_{op}\|N\|_F$ for any commensurate matrices in the third line and the triangle inequality in the last line.
		For $\W B^*$, invoking Lemma \ref{lem_W^TW} gives 
		\[
		{1\over n}\|\W B^*\|_{op}^2 \le \|(B^*)^T\Sigma_W B^*\|_{op}\left[1 + C_{\g_w}\left(\sqrt{K\over n} \vee {K \over n}\right)\right]
		\]
		with probability $1-2e^{-c'K}$ for some constants $c'>0$ and $C_{\g_w}$ depending on $\g_w$ only. This completes the proof. 
	\end{proof}

	\subsection{Proof of Theorem \ref{thm_U}: convergence rate of $\|\wh P_{\B} - P_{\B}\|_F$ for homoscedastic case}\label{sec_proof_thm_U}
	Recall that $\wh P_{\B} = \wh U \wh U^T$ and $P_{\B} = UU^T$.  Write 
	$$
	M = (B^*)^T\Sigma_W B^*.
	$$ 
	Applying Theorem 2 in \cite{Davis-Kahan-variant} to $\Sigma_{\eps}$ and $\wh \Sigma_{\eps}$ with $d = s = K$ and $r = 1$ yields 
	\begin{equation}\label{rate_U}
	\|\wh U Q - U \|_F \le {2^{3/2} 
		\|\wh \Sigma_{\eps} - \Sigma_{\eps}\|_{F}
		\over \lambda_K(\Sigma_{\eps}) - \lambda_{K+1}(\Sigma_{\eps})} = {2^{3/2} 
		\|\wh \Sigma_{\eps} - \Sigma_{\eps}\|_{F}
		\over \lambda_K(M)}
	\end{equation}
	for some orthogonal matrix $Q$. We also use the fact that $\Sigma_E = \tau^2\bI_m$ and $\lambda_{K+1}(M) = 0$. 
	Note that, for this $Q$, 
	\[
	\|\wh U\wh U^T - UU^T\|_F \le \|(\wh U Q - U)Q^T\wh U^T\|_F + \|U(\wh U Q - U)^T\|_F =  2\|\wh UQ - U\|_F.
	\]
	It then suffices to upper bound $\|\wh \Sigma_{\eps} - \Sigma_{\eps}\|_F$. Notice that
	\begin{align}\label{eq_diff_wh_Sigma_eps_eps_eps}\nonumber
	&\wh \Sigma_{\eps} - {1\over n}\Eps^T\Eps\\\nonumber 
	&= {1\over n}(\Y - \X \wh F)^T(\Y - \X \wh F) - {1\over n}\Eps^T\Eps\\
	& = {1\over n}(\wh F -F^*)^T\X^T\X(\wh F -F^*)+{1\over n}(F^* -\wh F)^T\X^T\Eps + {1\over n}\Eps^T\X(F^* -\wh F).
	\end{align}
	Recalling that (\ref{error_decomp}), we have 
	\begin{align*}
	\X \wh F - \X F^* &= \P(\Y - \X\Psi_0) - \X L_0 + \Q \X (\wh \Psi - \Psi_0)\\
	& = \P\Eps - \Q \X L_0 + \Q \X(\wh \Psi - \Psi_0)
	\end{align*}
	for any $\Psi_0 + L_0 = F^*$. 
	By using formula $\|H^TH\|_F\le \|H\|_F^2$ for any matrix $H$ and triangle inequality, we thus have 
	\begin{align*}
	\|\wh \Sigma_{\eps} - \Sigma_{\eps}\|_{F}
	&\le 
	{1\over n}\left\|\X\wh F-\X F^*\right\|_{F}^2 + {2\over n}\left\|\Eps^T\X(\wh F-F^*)\right\|_{F}+ \left\|{1\over n}\Eps^T\Eps- \Sigma_{\eps}\right\|_{F}\\
	&\le{1\over n}\left\|\X\wh F-\X F^*\right\|_{F}^2 + {2\over n}\left\|\Eps^T\P\Eps\right\|_F + {2\over  n}\left\|\Eps^T\Q\X L_0\right\|_F \\
	&\qquad + {2\over  n}\left\|\Eps^T\Q\X(\wh \Psi-\Psi_0)\right\|_{F}+ \left\|{1\over n}\Eps^T\Eps- \Sigma_{\eps}\right\|_{F}.
	\end{align*}
	We then study each terms on the right hand side. From Lemma \ref{lem_P_eps}, we have 
	\begin{align*}
	&{1\over n}\|\Eps^T\P\Eps\|_F \le {1\over n}\left\|\P^{1/2}\Eps\right\|_F^2\le 
	{V_\eps \over n}\left(
	\sqrt{\tr(\P)} + \sqrt{2\|\P\|_{op}\log(m/\epsilon)}
	\right)^2 \\
	&{1\over  n}\|\Eps^T\Q\X L_0\|_F  \le \sqrt{V_\eps\log(m/\epsilon) \over n}\sqrt{\|\Q\|_{op}\cdot Rem_2(L_0)}
	\end{align*}
	with probability $1-2\epsilon$. To bound the fourth term, first notice that 
	\[
	{1\over  n}\|\Eps^T\Q\X(\wh \Psi-\Psi_0)\|_{F} \le  \max_{1\le j\le p}{1\over n}\|\X_j^T\Q \Eps\|_2\cdot \|\wh \Psi - \Psi_0\|_{\l12}.
	\]
	Indeed, by writing $\Delta = \wh \Psi - \Psi_0$, one has 
	\begin{align*}
	\|\Eps^T\Q\X\Delta\|_{F}^2 & = \sum_{\ell =1}^m \Delta_{\ell}^T\X^T\Q\Eps\Eps^T\Q\X \Delta_{\ell}\\
	&\le \sum_{\ell =1}^m \sum_{i=1}^p|\Delta_{i\ell}|\sum_{j=1}^p|\Delta_{j\ell}| \max_{i, j}|\X_i^T\Q\Eps\Eps^T\Q\X_j|\\
	&\le \sum_{i=1}^p \sum_{j=1}^p \|\Delta_{i\cdot}\|_2 \|\Delta_{j\cdot}\|_2  \max_{1\le j\le p}\|\X_j^T\Q\Eps\|_2^2\\
	& = \|\Delta\|_{\l12}^2   \max_{1\le j\le p}\|\X_j^T\Q\Eps\|_2^2.
	\end{align*}
	Note that, on the event $\cE$ defined in (\ref{def_event}), 
	\[
	\|\wh \Psi - \Psi_0\|_{\l12} \lesssim    {\lambda_1s_0\over \wt \kappa^2(s_0, 4)}+ {Rem_2(L_0) \over \lambda_1}
	\]
	from (\ref{rate_Theta_hat_12}). Invoking (\ref{def_event}) yields 
	\[
	{1\over  n}\|\Eps^T\Q\X(\wh \Psi-\Psi_0)\|_{F}\lesssim Rem_2(L_0) +  {\lambda_1^2s_0\over \wt \kappa^2(s_0, 4)}
	\]
	with probability $1-\epsilon'$.
	Finally, the last term can be upper bounded by invoking Lemma \ref{lem_eps_eps} as 
	\begin{equation}\label{rate_epseps}
	\PP\left\{\left\|{1\over n}\Eps^T\Eps- \Sigma_{\eps}\right\|_{F} \le cV_\eps\left(\sqrt{\log m\over n} \vee {\log m \over n}\right)\right\} \ge 1-2m^{-c'}
	\end{equation}
	for some constant $c, c'>0$. Collecting terms and invoking Theorem \ref{lem_pred_theta} for $\|\X \wh F - \X F^*\|_F^2/n$ yield, after using $\|\P\|_{op}\le 1$, $\|\Q\|_{op}\le 1$ and Lemma \ref{lem_RE} to simplify the results,
	\begin{align}\label{rate_Sigma_eps}\nonumber
	&\|\wh \Sigma_{\eps} - \Sigma_{\eps}\|_{F}\\\nonumber 
	&\lesssim  Rem_2(L_0) +  {\lambda_1^2s_0\over \wt \kappa^2(s_0, 4)}+{V_\eps \over n}\left(
	\sqrt{\tr(\P)} + \sqrt{2\|\P\|_{op}\log(m/\epsilon)}
	\right)^2\\\nonumber
	&\quad + \sqrt{V_\eps\log(m/\epsilon) \over n}\sqrt{Rem_2(L_0)} + V_\eps\left(\sqrt{\log m\over n} \vee {\log m \over n}\right)\\
	&\lesssim 
	Rem_2(L_0) +{\lambda_1^2s_0\over \wt \kappa^2(s_0, 4)}+{ \tr(\P)V_\eps\over n}+ V_\eps\left(\sqrt{\log (m/\epsilon)\over n} \vee {\log(m/\epsilon) \over n}\right)
	\end{align}
	with probability $1-3\epsilon-\epsilon'-2m^{-c'}$. Recall the eigen-decomposition of $\wh \Sigma = U\diag(\sigma_1,\ldots, \sigma_p)U^T$ with $U = (u_1,\ldots, u_K)$. We have
	\begin{align}\nonumber
	&\tr(\P) = \tr\left[
	{1\over n} \X(\wh \Sigma+ \lambda_2 \bI_p)^{-1}\X^T
	\right] = \sum_{k=1}^q{\sigma_k \over \sigma_k+\lambda_2},\\\label{bd_rem_2}
	&Rem_2(L_0) = \sum_k {\lambda_2\sigma_k \over \sigma_k + \lambda_2} u_k^TL_0L_0^Tu_k \le  {\lambda_2\sigma_1 \over \sigma_1 + \lambda_2}\|L_0\|_F^2.
	\end{align}
	By invoking Lemma \ref{lem_RE} together with the choice of $\lambda_1$ as (\ref{rate_lbd1}), we further have 
	\begin{equation}\label{bd_rem_3}
	{\lambda_1^2s_0\over \wt \kappa^2(s_0, 4)} \lesssim {\sigma_1 + \lambda_2 \over \lambda_2}{\lambda_1^2s_0 \over \kappa^2(s_0,4)} \lesssim {\lambda_2(\sigma_1 + \lambda_2) \over (\sigma_q + \lambda_2)^2} {s_0V_\eps \log N \over \kappa^2(s_0,4)n} 
	\end{equation}
	Take $\eps = m^{-c'}$ and use $\log m \le n$ in (\ref{rate_Sigma_eps}) to complete the proof. 
	\qed 
	
	\subsection{Lemma \ref{lem_pred_theta} used in the proof of Theorem \ref{thm_U}}\label{sec_proof_lem_pred_theta}
	The following lemma provides the rate of $\|\wh \Psi -\Psi_0\|_{\l12}$ where $\wh \Psi$ is obtained in (\ref{est_F}). Furthermore, the proof also reveals that Lemma \ref{lem_pred_theta} holds if $\lambda_1$ is replaced by any $\wt \lambda_1 \ge \lambda_1$.
	
	\begin{lemma}\label{lem_pred_theta}
		Under conditions of Corollary \ref{cor_pred}, choose $\lambda_1$ as (\ref{rate_lbd1}) and  any $\lambda_2 \ge 0$ such that $\P$ exists. With probability $1-\epsilon - \epsilon'$, 
		\begin{equation}\label{rate_Theta_hat}
		\|\wh\Psi - \Psi_0\|_{\l12} \lesssim    {\lambda_1s_0\over \wt \kappa^2(s_0, 4)}+ {Rem_2(L_0) \over \lambda_1}
		\end{equation}
		where $L_0 = F^* - \Psi_0$ and $Rem_2(L_0)$  is defined  in Theorem \ref{thm_pred}.
	\end{lemma}
	
	\begin{proof}
		We prove (\ref{rate_Theta_hat}) by working on the event $\cE$ defined in (\ref{def_event}). From (\ref{bd_pred_X_tilde}),
		\begin{align}\nonumber
		{1\over n}\|\wt \X (\wh \Psi - \Psi_0)\|_F^2 &\le \max\left\{
		16\lambda_2 \tr\left[
		L_0^T	\wh\Sigma (\wh \Sigma + \lambda_2\bI_p)^{-1}L_0
		\right], {4\lambda_1^2 \over \kappa_1^2(1/2, \Psi_0, \lambda_1, \lambda_2)}
		\right\}\\\label{disp_Xtd_Delta_true}
		&\lesssim  Rem_2(L_0) +  {\lambda_1^2 s_0 \over \wt \kappa^2(s_0, 3)}.
		\end{align}
		by invoking the first result of Lemma \ref{lem_RE} with $c = 1/2$. Write $\Delta := \wh \Psi - \Psi_0$ and  consider two cases.\\
		(1) When $\Delta \in \C(S_0, 4)$ with $S_0 = \{1\le j\le p: \|[\Psi_0]_{j\cdot}\|_2 \ne 0\}$, it follows from the definitions of $\C(S_0, 4)$ and $\wt\kappa(s_0, 4)$ that
		\begin{align*}
		\|\Delta\|_{\l12} & \le 5\|\Delta_{S_0\cdot}\|_{\l12} \le 5\sqrt{s_0}\|\Delta_{S_0\cdot}\|_F\le 5{\sqrt{s_0} \over \wt\kappa(s_0,4)} {1\over \sqrt{n}}\|\wt \X\Delta\|_F.
		\end{align*}
		(2) When $\Delta \notin \C(S_0, 4)$, it implies 
		$
		\|\Delta_{S_0^c\cdot}\|_{\l12} > 4\|\Delta_{S_0\cdot}\|_{\l12}.
		$
		From (\ref{disp_Xtd_Delta}), by invoking $\cE$, we obtain
		\begin{align*}
		{1\over n}\|\wt \X \Delta\|_F^2 &\le {2\over n}\|\wt \X L_0\|_F\|\wt \X \Delta\|_F +  {\lambda_1 \over 2} \left(
		3\|\Delta_{S_0\cdot}\|_{\l12} - \|\Delta_{S_0^c\cdot}\|_{\l12}
		\right)\\
		&\le {2\over n}\|\wt \X L_0\|_F\|\wt \X \Delta\|_F -  {\lambda_1 \over 8} \|\Delta_{S_0^c\cdot}\|_{\l12}.
		\end{align*}
		This implies $\|\wt \X\Delta\|_F \le 2\|\wt \X L_0\|_F$ and 
		\[
		\|\Delta_{S_0^c\cdot}\|_{\l12}\le {16\over \lambda_1}\cdot {1\over n}\|\wt \X L_0\|_F\|\wt X\Delta\|_F \le {32 \over \lambda_1} {1\over n}\|\wt \X L_0\|_F^2.
		\]
		Using $	\|\Delta_{S_0^c\cdot}\|_{\l12} > 4\|\Delta_{S_0\cdot}\|_{\l12}$ again yields
		\begin{align*}
		\|\Delta\|_{\l12} \le {5\over 4}\|\Delta_{S_0^c\cdot}\|_{\l12} \le {40 \over \lambda_1}{1\over n}\|\wt \X L_0\|_F^2.
		\end{align*}
		Combining these two cases and invoking (\ref{disp_XL0}) with $L_0$ and (\ref{disp_Xtd_Delta_true}) give 
		\begin{align}\label{rate_Theta_hat_12}\nonumber
		\|\Delta\|_{\l12} &\lesssim {\sqrt{s_0}\over \wt \kappa(s_0, 4)}\left(\sqrt{Rem_2(L_0)} +  {\lambda_1 \sqrt {s_0}\over \wt \kappa(s_0, 3)}\right) + {Rem_2(L_0) \over \lambda_1}\\
		&\lesssim  {\lambda_1s_0\over \wt \kappa^2(s_0, 4)}+ {Rem_2(L_0) \over \lambda_1}
		\end{align}
		where we also use $\wt \kappa(s_0, 3) \ge \wt \kappa(s_0,4)$ to derive the last inequality. We then conclude the proof by
		recalling that $\PP(\cE) = 1-\epsilon'$.
	\end{proof}
	
	\subsection{Proof of Theorem \ref{thm_U_hetero}: convergence rate of $\|\wt P_{\B} - P_{\B}\|_F$ for heteroscedastic case}\label{sec_proof_thm_U_hetero}
	Let $M = (B^*)^T\Sigma_W B^*$. From (\ref{eq_diff_wh_Sigma_eps_eps_eps}) and by using $\Sigma_{\eps} = M + \Sigma_E$, one has 
	\begin{align*}
	\wh\Sigma_{\eps} = M + \Delta_1 + \Delta_2 
	\end{align*}
	where $	\Delta_1 = \Sigma_E = \diag(\tau_1^2, \ldots, \tau_m^2)$ and 
	\begin{align*}
	\Delta_2 &= {1\over n}(\wh F -F^*)^T\X^T\X(\wh F -F^*)+{1\over n}(F^* -\wh F)^T\X^T\Eps \\
	&\quad + {1\over n}\Eps^T\X(F^* -\wh F)+ {1\over n}\Eps^T\Eps - \Sigma_{\eps}.		
	\end{align*}
	We aim to apply Theorem \ref{thm_rsintheta} with $N = \wh\Sigma_{\eps}$, $M = M$, $Z = \Delta_1 + \Delta_2$, $\wh V = \wt U$ and $V=U$. Observe that $\|\Gamma(\Delta_1)\|_F = 0$, that is, the off-diagonal elements of $\Delta_1$ are zero, and $\|\Delta_2\|_F =  \|\wh\Sigma_{\eps} - \Sigma_{\eps}\|_F$ from (\ref{eq_diff_wh_Sigma_eps_eps_eps}). Invoking (\ref{rate_Sigma_eps}) yields
	\begin{equation}\label{rate_Delta_2}
	\PP\left\{\|\Delta_2\|_F \le Rem(P_{\B})\cdot \lambda_K(M)\right\} \ge 1-\epsilon' - 5m^{-c'}
	\end{equation}
	with $Rem(P_{\B})$ defined in (\ref{def_RU}). We then work on the event 
	that the above display holds. Recall that $\Gamma(\Delta_2)$ denotes the matrix with off-diagonal elements equal to $\Delta_2$ and diagonal elements equal to zero. 
	Since $Rem(P_{\B}) \le c\sqrt{K}$ implies $\|\Gamma(\Delta_2)\|_F \le \|\Delta_2\|_F \le c\sqrt{K}\lambda_K(M)$, in conjunction with condition (\ref{cond_ic}), an application of  Theorem \ref{thm_rsintheta} with $N = \wh\Sigma_{\eps}$, $M = M$, $Z = \Delta_1 + \Delta_2$, $\wh V = \wt U$ and $V=U$ gives
	\[
	\|\sin\Theta(\wt U, U)\|_F  \lesssim {\|\Gamma(\Delta_2)\|_F \over \lambda_K(M)} \wedge \sqrt K \le {\|\Delta_2\|_F \over \lambda_K(M)}.
	\]
	Finally, using the inequality
	\[
	\|\wt P_{\B} -  P_{\B}\|_F = \|\wt U \wt U^T - UU^T\|_F \le 2\|\sin\Theta(\wt U, U)\|_F.
	\]
	and (\ref{rate_Delta_2}) again concludes the proof.\qed

	\subsection{Proof of Theorem \ref{thm_robust_PCA}: consistency of using PCA to estimate $P_{\B}$ in the presence of heteroscedasticity}\label{sec_proof_thm_robust_PCA}	
	
	The proof follows the same arguments as that of Theorem \ref{thm_U}. 
	The first difference is to apply Theorem 2 in \cite{Davis-Kahan-variant} to $\wh\Sigma_{\eps}$ and $\Pi := B^{*T}\Sigma_W\B + \bar\tau^2\bI_m$ with $d = s= K$ and $r=1$ to obtain 
	\[
	\|\wh U Q - U \|_F \le {2^{3/2} 
		\|\wh \Sigma_{\eps} - \Pi\|_{F}
		\over \lambda_K(\Pi) - \lambda_{K+1}(\Pi)} = {2^{3/2} 
		\|\wh \Sigma_{\eps} - \Pi\|_{F}
		\over \lambda_K(\Pi)}.
	\]
	The second difference from the proof of Theorem \ref{thm_U} is to upper bound the numerator as 
	\[
	\|\wh\Sigma_{\eps} - \Pi\|_F\le \|\wh\Sigma_{\eps} - \Sigma_{\eps}\|_F + \|\Sigma_{\eps} - \Pi\|_{F} = \|\wh\Sigma_{\eps} - \Sigma_{\eps}\|_F + \|\Sigma_E-\bar\tau^2\bI_m\|_F
	\]
	by adding and subtracting $\Sigma_{\eps}$ and using $\Sigma_{\eps} = \Pi + \Sigma_E - \bar\tau^2\bI_m$. 
	Since the results in Theorem \ref{thm_pred} still hold in the heteroscedasticity case, $\|\wh\Sigma_{\eps} - \Sigma_{\eps}\|_F $ can be bounded by (\ref{rate_Sigma_eps}). Then the proof is completed by using 
	\[
	\|\Sigma_E- \bar\tau^2\bI_m\|_F^2 = \sum_{j=1}^m \left(\tau_j^2-\bar\tau^2\right)^2.
	\]

	\subsection{Proof of Theorem \ref{thm_K}: selection of $K$}\label{sec_proof_thm_K}
	Write $M = (B^*)^T\Sigma_W B^*$ and its eigenvalues as $\Lambda_1 \ge \Lambda_2 \ge \cdots \ge \Lambda_K$ and $\Lambda_j = 0$ for $j>K$. Recall that $\wh \lambda_1\ge \wh\lambda_2\ge \cdots \ge  \wh \lambda_m$ are the eigenvalues of $\wh \Sigma_{\eps}$.  By Weyl's inequality, we have 
	\[
	|\wh \lambda_j - \Lambda_j| \le \|\wh \Sigma_{\eps} - M\|_{op}
	\]
	for all $1\le j\le m$. We work on the event $\{\|\wh \Sigma_{\eps} - \Sigma_{\eps}\|_{F}  \lesssim Rem(P_{\B}) \Lambda_K\}$ which, from the proof of Theorem \ref{thm_U}, holds with probability $1-\epsilon' - 5m^{-c}$ for some constant $c>0$. Note that 
	\begin{align}\label{bd_op}
	\|\wh \Sigma_{\eps} - M\|_{op}  & \le \|\wh \Sigma_{\eps} - \Sigma_{\eps}\|_{F} + \|\Sigma_E\|_{op}\lesssim Rem(P_{\B}) \Lambda_K+ \max_{1\le j\le m}\tau_j^2.
	\end{align}
	We thus conclude 
	\[
	|\wh \lambda_j - \Lambda_j| \lesssim \max_{1\le j\le m}\tau_j^2 +Rem(P_{\B}) \Lambda_K 
	\]
	for $1\le j\le m$. Since $(b)$ of Assumption \ref{ass_rates} implies $\Lambda_j \asymp m$ for $1\le j\le K$, by also using $Rem(P_{\B})=o(1)$ and $\max_{j}\tau_j^2 = O(1)$, we have $\wh \lambda_j \asymp m$.  This concludes $\wh \lambda_{j+1} / \wh \lambda_{j} \asymp 1$ for $1\le k\le {K-1}$.  On the other hand, since $\wh \lambda_{K+1} = O(\max_j\tau_j^2 +Rem(P_{\B}) \Lambda_K )$, we further obtain 
	$\wh \lambda_{K+1} / \wh \lambda_K = O(\max_j\tau_j^2/m +Rem(P_{\B}))$. This completes the proof. \qed

	\medskip
	
	\section{Auxiliary proofs and technical lemmas}\label{sec_proof_auxiliary}
	
	\subsection{Lemmas used in Remark \ref{rem_design_impact_factor}}
	We establish the connection between the impact factor defined in (\ref{RE_pred}) and the RE conditions of $\wt \X$ and $\X$.  Recall that  $M = n^{-1}\X^T\Q^2\X$ and $\sigma_1 \ge \sigma_2\ge \cdots \ge \sigma_q > 0$ are the non-zero eigenvalues of $\wh\Sigma$ with $q := \rank(\X)$.
	\begin{lemma}\label{lem_RE}
		For any given $\Psi_0$ with row-sparsity $s_0$ and any constant $c\in (0,1)$, one has 
		\[
		\kappa_1(c, \Psi_0, \lambda_1, \lambda_2) \ge{\wt\kappa(s_0, \alpha_{c}) \over  (1+c)\sqrt{s_0}}  \ge
		\sqrt{\lambda_2 \over  \sigma_1+ \lambda_2}\cdot {\kappa(s_0, \alpha_{c})\over (1+c)\sqrt{ s_0}}
		\]
		with $\alpha_c = (1+c)/(1-c)$, and 
		\[
		\max_{1\le j\le p}M_{jj} \le \max_{1\le j\le p}\wh\Sigma_{jj}  \left(\lambda_2 \over \sigma_q + \lambda_2\right)^2.
		\]
		As a result,  we have 
		\[
		{\max_{1\le j\le p}M_{jj} \over 
			\kappa_1^2(1/2, \Psi_0, \lambda_1, \lambda_2)} \le	{9s_0\over 4\kappa^2(s_0, 3)}\max_j \wh\Sigma_{jj} \cdot {\lambda_2(\sigma_1 + \lambda_2) \over (\sigma_q+\lambda_2)^2}.
		\]
	\end{lemma}
	\begin{proof}
		We first prove 
		\[
		\kappa_1(c, \Psi_0, \lambda_1, \lambda_2) \ge {\wt\kappa(s_0, \alpha_{c}) \over  (1+c)\sqrt{s_0}}.
		\]
		Observe that $\R(c, \Psi_0, \lambda_1, \lambda_2) \subseteq \C(S, \alpha_{c})$ for any $|S| \le s_0$, $c\in (0,1)$ and $\alpha = (1+c)/(1-c)$. Indeed, for any $\Delta \in \R(c, \Psi_0, \lambda_1, \lambda_2)$, $|S| \le s_0$ and $c\in (0,1)$, 
		\begin{align*}
		0 &\le \|\Psi_0\|_{\l12} - \|\Psi_0 + \Delta\|_{\l12} + c\|\Delta\|_{\l12}\\
		&\le \|\Delta_{S\cdot}\|_{\l12} - \|\Delta_{S^c \cdot}\|_{\l12} + c\|\Delta_{S\cdot}\|_{\l12}+c\|\Delta_{S^c \cdot}\|_{\l12}\\
		& = (1+c)\|\Delta_{S\cdot}\|_{\l12} - (1-c)\|\Delta_{S^c \cdot}\|_{\l12},
		\end{align*}
		which implies $\Delta \in \C(S, \alpha_{c})$. Note the above display also implies 
		\[
		\|\Psi_0\|_{\l12} - \|\Psi_0 + \Delta\|_{\l12} + c\|\Delta\|_{\l12} \le \sqrt{s_0}(1+c)\|\Delta_{S\cdot}\|_F.
		\]
		We thus have 
		\begin{align*}
		\wt \kappa(s_0, \alpha_{c}) &= \min_{\substack{S\subseteq \{1, 2, \ldots, p\}\\ |S|\le s_0}}\min_{\Delta \in \C(S, \alpha)} {\|\wt \X\Delta\|_F/\sqrt{n} \over \|\Delta_{S\cdot}\|_F}\\
		&\le \min_{\substack{S\subseteq \{1, 2, \ldots, p\}\\ |S|\le s_0}}\min_{\Delta \in \C(S, \alpha)} {(1+c)\sqrt{s_0}
			\cdot \|\wt \X\Delta\|_F/\sqrt{n} \over \|\Psi_0\|_{\l12} - \|\Psi_0 + \Delta\|_{\l12} + c\|\Delta\|_{\l12} }\\
		&\le\min_{\Delta \in \R(c, \Psi_0, \lambda_1, \lambda_2)} {(1+c)\sqrt{s_0}
			\cdot \|\wt \X\Delta\|_F/\sqrt{n} \over \|\Psi_0\|_{\l12} - \|\Psi_0 + \Delta\|_{\l12} + c\|\Delta\|_{\l12} }\\
		&= (1+c)\sqrt{s_0}\kappa_1(c, \Psi_0, \lambda_1, \lambda_2).
		\end{align*}
		We then prove the second inequality of the first statement. Since
		\begin{align*}
		{1\over n}\|\wt \X \Delta \|_F^2 
		&= \tr\left[
		\Delta^T {1\over n}\X^T\Q \X \Delta 
		\right]\\
		& = \lambda_2\tr\left[
		\Delta^T \wh \Sigma(\wh \Sigma + \lambda_2 \bI_p)^{-1}\Delta 
		\right]\\
		&\ge{\lambda_2 \over \|\wh \Sigma\|_{op}+\lambda_2}\tr\left( 
		\Delta^T\wh \Sigma \Delta \right)
		\end{align*}
		by using  Fact \ref{fact_QX} in the second line, it follows that 
		\[
		\wt \kappa^2(s, \alpha) \ge {\lambda_2 \over \sigma_1+\lambda_2} \kappa^2(s, \alpha).
		\]
		
		We then show the second statement. From Fact \ref{fact_QX} and $\wh \Sigma = U\diag(\sigma_1, \ldots, \sigma_q)U^T$, we have 
		\[
		M= \lambda_2^2 (\wh \Sigma + \lambda_2 \bI_p)^{-1}\wh \Sigma (\wh \Sigma + \lambda_2 \bI_p)^{-1} = \lambda_2^2 U D U^T = \lambda_2^2\wh\Sigma^{1/2}(\wh \Sigma + \lambda_2\bI_p)^{-2}\wh\Sigma^{1/2}
		\]
		with $D$ being diagonal and $D_{kk} =  \sigma_k / (\sigma_k + \lambda_2)^2$ for $1\le k\le q$. This implies
		\[
		\max_{1\le j\le p} M_{jj} \le \max_{1\le j\le p} \wh\Sigma_{jj}\left({\lambda_2 \over \sigma_q + \lambda_2}\right)^2
		\]
		which, in conjunction with the previous result of $\kappa_1(c, \Psi_0, \lambda_1, \lambda_2)$ with $c = 1/2$, gives 
		\[
		{\max_{1\le j\le p}M_{jj} \over 
			\kappa_1^2(1/2, \Psi_0, \lambda_1, \lambda_2)}  \le  
		{9s_0\over 4\kappa^2(s_0, 3)}\max_j \wh\Sigma_{jj} \cdot {\lambda_2(\sigma_1 + \lambda_2) \over (\sigma_q+\lambda_2)^2}.
		\]
		This completes the proof. 
	\end{proof}

	\subsection{Proof of Corollary \ref{cor_pred} and Remark \ref{rem_orth_design}}\label{sec_proof_cor_thm_pred}

	We first prove the following lemma for any pair $\Psi_0 + L_0 = F^*$ from which Corollary \ref{cor_pred} follows immediately. By taking $q=p$, $\sigma_k = 1$ for all $1\le k\le q$ and $\kappa(s_0, 3) = 1$, the bound in (\ref{rate_fit_orth_design}) of Remark \ref{rem_orth_design} follows from Corollary \ref{cor_pred}.
	
	\begin{lemma}\label{lem_Rem23}
		Let $Rem_1$, $Rem_2(L_0)$ and $Rem_3(\Psi_0)$ be defined in Theorem \ref{thm_pred}. One has 
		\begin{align*}
		& Rem_1 ~\lesssim~ \left[
		\sum_{1\le k\le q}\left({\sigma_k \over \sigma_k + \lambda_2}\right)^2+
		\max_{1\le k\le q}\left( {\sigma_k \over \sigma_k + \lambda_2}\right)^2\log(m/\epsilon)
		\right]{V_\eps \over n},\\
		&Rem_2(L_0)~\lesssim~  {\lambda_2\sigma_1 \over \sigma_1 + \lambda_2}\|L_0\|_F^2,\\			
		&Rem_3(\Psi_0) ~\lesssim~  \max_{1\le j\le p} \wh \Sigma_{jj}{\lambda_2(\sigma_1 + \lambda_2) \over (\sigma_q + \lambda_2)^2}\left(
		1+ {\log(p/\epsilon') \over r_e(\Gamma_\eps)}\right)
		{s_0  \over  \kappa^2(s_0, 3)} {V_{\eps} \over n}
		\end{align*}
		and $\|\Q\|_{op} \le \lambda_2 / (\sigma_q + \lambda_2)$.
	\end{lemma}
	\begin{proof}
		Recall that $\wh\Sigma = U\diag(\sigma_1, \ldots, \sigma_q)U^T$ with $U = (u_1,\ldots, u_K)$ and $\P = \X(\X^T\X + n\lambda_2\bI_p)^{-1}\X^T$. The first result follows by observing
		\begin{align}\label{tr_P_square}
		\tr(\P^2) &= \tr\left[
		{1\over n} \X(\wh \Sigma+ \lambda_2 \bI_p)^{-1}\wh\Sigma (\wh \Sigma+ \lambda_2 \bI_p)^{-1}\X^T
		\right] = \sum_{k=1}^q\left({\sigma_k \over \sigma_k+\lambda_2}\right)^2,\\\nonumber
		\|\P\|_{op} &= \left\|
		(\wh \Sigma+ \lambda_2 \bI_p)^{-1/2} {1\over n}\X^T\X (\wh \Sigma+ \lambda_2 \bI_p)^{-1/2}
		\right\|_{op} \le  {\sigma_1 \over \sigma_1+\lambda_2}.
		\end{align}
		By noting that  
		\[
		\tr\left(L_0^T \wh \Sigma (\wh \Sigma + \lambda_2\bI_p)^{-1}L_0\right) = \sum_{k=1}^q {\sigma_k \over \sigma_k + \lambda_2} L_0^Tu_k u_k^TL_0 \le  {\sigma_1 \over \sigma_1 + \lambda_2}\|L_0\|_F^2,
		\]
		with $U = (u_1, \ldots, u_q)$, the second result follows. The bound of $Rem_3(\Psi_0)$ can be derived from Lemma \ref{lem_RE}. Finally, since $\|\P\|_{op} \ge \sigma_q / (\lambda_2 + \sigma_q)$,  we immediately have
		\begin{equation}\label{bd_Q_op}
		\|\Q\|_{op} \le 1 - \|\P\|_{op} \le   {\lambda_2 \over \sigma_q + \lambda_2}.
		\end{equation}		
		The proof is complete.
	\end{proof}

	\subsection{Proof of Corollary \ref{cor_rate}, Remarks \ref{rem_rate_Theta} and \ref{rem_rate_theta_hat}}\label{sec_proof_cor_rate}
	
	\begin{proof}[Proof of Corollary \ref{cor_rate}]
		By inspecting the proof of Theorem \ref{thm_pred}, we can change the logarithmic factors in Theorem \ref{thm_Theta} and Theorem \ref{thm_U} to $\log(N)$ with $N = n\vee m\vee p$. The resulting probabilities will tend to $1$ as $n\to \infty$. 
		
		We first upper bound $Rem(P_{\B}; \lambda_2) := Rem(P_{\B})$ defined in (\ref{def_RU}). Here we write the dependency on $\lambda_2$ explicitly. By replacing $s_0$ and $\|L_0\|_F^2$ by $s_*$ and $R_*$, respectively, we have
		\begin{align*}
		Rem(P_{\B}; \lambda_2)  &\le {1\over \Lambda_K}\Bigg\{V_\eps\sqrt{\log m\over n} +  {\lambda_2\sigma_1 \over \lambda_2 + \sigma_1}{R_*}+\sum_{k=1}^q {\sigma_k \over \sigma_k + \lambda_2}{V_\eps \over n}\\
		&\hspace{0.5cm} + {\lambda_2(\sigma_1 + \lambda_2) \over (\sigma_q + \lambda_2)^2}\cdot \max_{1\le j\le p} \wh \Sigma_{jj}\left(1 + {\log(p/\epsilon') \over r_e(\Gamma_{\eps})}\right){s_*\over \kappa^2(s_*,4)}{V_\eps \over n}\Bigg\}.
		\end{align*}
		Further recalling that (\ref{def_V_eps}), we have $
		V_\eps \asymp (K\Lambda_1\g_w^2 + m\g_e^2),
		$ under part $(b)$ of Assumption \ref{ass_rates}. This implies 
		\begin{equation}\label{rate_V_Lambda_K}
		{V_\eps \over \Lambda_K}  \asymp K\g_w^2 + \g_e^2 = O(K).
		\end{equation}
		From $\kappa^{-1}(s_*,4)=O(1)$ in part $(a)$ of Assumption \ref{ass_rates} and $\wh\Sigma_{jj} = 1$ for all $1\le j\le p$, we conclude 
		\begin{align}\label{simp_RU}
		Rem(P_{\B}; \lambda_2) & \lessapprox {\lambda_2\sigma_1 \over \sigma_1 + \lambda_2}{R_* \over m}+{K\over n} \sum_{k=1}^q{\sigma_k \over \sigma_k+\lambda_2}+{\lambda_2(\sigma_1 + \lambda_2) \over (\sigma_q + \lambda_2)^2} {Ks_*\over n}  + {K \over \sqrt n}.
		\end{align}
		We then prove
		\begin{align}\label{bd_Err_PB}\nonumber
		&\min_{\lambda_2}Rem(P_{\B}; \lambda_2)  \\
		&\lessapprox  \min\Bigg\{
		{\sigma_1R_* \over m} + {Ks_*\over n},~ {qK\over n} ,~\sqrt{ {(p+\sigma_1s_*)KR_*\over nm}} + {Ks_*\over n}
		\Bigg\} + {K\over \sqrt n}.
		\end{align}
		To prove the first bound, by choosing $\lambda_2 \to \i$ in (\ref{simp_RU}), we have
		\[
		\min_{\lambda_2} Rem(P_{\B})  ~ \lessapprox ~  {\sigma_1R_* \over m} + {s_*K \over n} + {K\over \sqrt n}.
		\]
		To prove the second bound, take $\lambda_2 \to 0$ in (\ref{simp_RU}) to obtain
		\[
		\min_{\lambda_2} Rem(P_{\B})  ~ \lessapprox  ~ {qK\over n} + {K\over \sqrt n}.
		\]
		Finally, from $\sum_k \sigma_k = \tr(\wh\Sigma) = p$, display (\ref{simp_RU}) yields
		\begin{align*}
		Rem(P_{\B}; \lambda_2)  &~ \lessapprox ~  \lambda_2 {R_* \over m} + {pK\over \lambda_2 n} + {\sigma_1 + \lambda_2 \over \lambda_2}{s_*K \over n} + {K\over \sqrt n}\\
		&~ = ~ \lambda_2 {R_* \over m} + {(p+\sigma_1 s_*)K\over \lambda_2 n} + {s_*K \over n} + {K\over \sqrt n}.
		\end{align*}
		Optimizing the above display over $\lambda_2$ yields 
		\[
		\lambda_2^2 ~ =  ~ \left({pK \over n} + {\sigma_1}{s_*K\over n}\right){m\over R_*}
		\]
		such that 
		\[
		\min_{\lambda_2} Rem(P_{\B}; \lambda_2) ~ \lessapprox ~ \sqrt{\left({pK \over n} + {\sigma_1}{s_*K\over n}\right){R_*\over m}} + {s_*K\over n}+{K\over \sqrt n}.
		\]
		We thus have proved (\ref{bd_Err_PB}). 
		
		Finally, using $K = O(n)$ and Assumption \ref{ass_rates} concludes 
		\begin{equation}\label{rate_prefix}
		{1\over \sqrt n}\|\X F^*\|_{op}+  \Lambda_1^{1/2}\left(1+ \sqrt{K\log m \over n}\right) \lessapprox \sqrt{m + s_*}.
		\end{equation}
		On the other hand, we have 
		\begin{equation}\label{simp_lbd3}
		{\bar\lambda_3s_* \over \kappa^2(s_*,4)} \lessapprox s^*\sqrt{ m \over  n}.
		\end{equation}
		This completes the proof of Corollary \ref{cor_rate}. 
	\end{proof}
	
	\medskip
	
	\begin{proof}[Proof of Remark \ref{rem_rate_Theta}]
		When  $p<n$, $s_*\asymp p$ and $\sigma_1 = O(1)$, from Corollary \ref{cor_rate} and $q = \min\{n,p\} \le p$, we have
		\begin{align*}
		{\rm Err}(P_{\B}) &= \min\Bigg\{
		{\sigma_1 R_* \over m} + {Ks_*\over n},~ {qK\over n} ,~\sqrt{ {(p+\sigma_1s_*)K R_*\over nm}} + {Ks_*\over n}
		\Bigg\} + {K\over \sqrt n}\\
		&\lesssim \min\Bigg\{ {R_* \over m} + {pK\over n},~ {pK\over n} ,~\sqrt{ {R_* \over m} \cdot {pK\over n}} + {pK\over n}
		\Bigg\} + {K\over \sqrt n}\\
		&\le {pK\over n} + {K \over \sqrt{n}}.
		\end{align*}
		Using $p\asymp s_*$ gives 
		\[
		\sqrt{s_*(m + s_*)\over  m}\cdot {\rm Err}(P_{\B}) \lessapprox \sqrt{p(m+p)\over m}\left(
		{pK\over n} + {K \over \sqrt{n}}
		\right).
		\]
		The result of $p<n$ then follows by (\ref{simp_lbd3}) and $K = O(\sqrt{p}\wedge \sqrt{n/p})$.
		
		To prove the result of $p\ge n$, note that
		\begin{align*}
		{\rm Err}(P_{\B}) &\le  \min\Bigg\{
		\sigma_1 {R_* \over m} ,~\sqrt{ {(p+\sigma_1s_*)K R_*  \over nm}}
		\Bigg\} + {Ks_*\over n}+ {K\over \sqrt n}.
		\end{align*}
		Condition $K = O(\sqrt{s_*} \wedge \sqrt{n/s_*} \wedge m)$ implies
		\[
		\sqrt{s_*(m+s_*) \over m}\left(
		{K\over \sqrt n} + {Ks_*\over n}
		\right)= \left\{
		\begin{array}{ll}
		O\left(s_*/\sqrt{n}\right)   & \text{if }s_* = O(m); \\
		O\left(s_*/\sqrt n + (s_*/\sqrt{n})^2\right) & \text{if }m = O(1).
		\end{array}
		\right.
		\]
		The desired result then follows. 
	\end{proof}
	
	\medskip
	
	\begin{proof}[Proof of Remark \ref{rem_rate_theta_hat}]
		Note that $\Ttheta = \Ps$ when $\Ps P_{\B} = 0$. As a result, both $\Ps$ and $L^*$ are identifiable. 
		From Lemma \ref{lem_pred_theta}, we have 
		\begin{align*}
		\|\wh\Psi - \Ps\|_{\l12} = O_p\left(    {\wt\lambda_1s_*\over \wt \kappa^2(s_*, 4)}+ {Rem_2(L^*) \over \wt\lambda_1}\right)
		\end{align*}
		for any $\wt\lambda_1 \ge \lambda_1$ with $\lambda_1$ defined in (\ref{rate_lbd1}). 
		For a suitable choice of $\wt \lambda_1$, we can deduce that
		\[
		\|\wh\Psi - \Ps\|_{\l12} = O_p\left(   
		{\lambda_1s_*\over \wt \kappa^2(s_*, 4)} + \sqrt{{s_*Rem_2(L^*)\over \wt \kappa^2(s_*, 4)}}\right)
		\]
		By Lemma \ref{lem_RE} and (\ref{bd_rem_2}) together with $V_\eps = O(Km)$ and $[\kappa(s_*,4)]^{-1}=O(1)$ under Assumption \ref{ass_rates}, we have
		\[
		{1\over \sqrt m} \|\wh\Psi - \Ps\|_{\l12}  
		\lessapprox {\sigma_1 + \lambda_2 \over \sigma_q + \lambda_2}\cdot {s_*\sqrt{K} \over \sqrt n} + \sqrt{{s_* \sigma_1R_*\over m}}.
		\]
		The result then follows by taking $\lambda_2 \gtrsim \sigma_1$. 
	\end{proof}
	
	\subsection{Lemma used in Section \ref{sec_dis_cond}}
	
	The following lemma provides the upper bound of $n^{-1}\|\X F^*\|_{op}^2$. Recall that $L^* = A^*B^*$ with $A^*$ defined in (\ref{def_A}). Without loss of generality, we can assume $Y$, $X$ and $Z$ are centered such that 
	\[
	A^* = \Sigma^{-1}\Sigma_{XZ}
	\]
	by writing $\Sigma_{XZ} = \Cov(X,Z)$. Recall that $\Sigma = \Cov(X)$, $\Sigma_Z = \Cov(Z)$ and $S_* = \{1\le j\le p: \|\Ps_{j\cdot}\|_2 \ne 0\}$.

	\begin{lemma}\label{lem_cond}
		Suppose $\Sigma^{-1/2}\X_{i\cdot}$ are i.i.d. $\g_X$ sub-Gaussian random vectors for $1\le i\le n$ where $\g_X$ is some positive constant. Further assume $\|\Sigma_{S_*S_*}\|_{op} = O(1)$, $\|\Sigma_Z\|_{op} = O(1)$, $K=O(n)$, $s_* = O(n)$ and $\|\B\|_{op}^2 = O(m)$. Then 
		\[
		{1\over n}\|\X F^*\|_{op}^2 = O_p(s_* + m).
		\]
	\end{lemma}
	\begin{proof}
		First, we have 
		$$\|\X F^*\|_{op} \le \|\X \Ps \|_{op} + \|\X L^*\|_{op}.$$ 
		To bound the first term, one has 
		\[
		\|\X \Ps \|_{op}  = \|\X_{S_*}\|_{op}\|\Ps_{S_*\cdot}\|_{op}.
		\]
		Since $\Sigma^{-1/2}\X_{i\cdot}$ are i.i.d. $\gamma_X$ sub-Gaussian, Theorem 5.39 in \cite{vershynin_2012} yields
		\[
		{1\over n}\|\X_{S_*}\|_{op}^2 = {1\over n}\|\X_{S_*}^T\X_{S_*}\|_{op} = O_p\left( \|\Sigma_{S_*S_*}\|_{op}\left(1 + {s_* \over n}\right)\right).
		\]
		Invoking $\|\Sigma_{S_*S_*}\|_{op} = O(1)$ and $s_* = O(n)$ together with $\|\Ps\|_{op}^2 = O(s_*+m)$ yields 
		\[
		{1\over n}\|\X_{S_*}\|_{op}^2 = O_p(m + s_*).
		\]
		To bound $\|\X L^*\|_{op} \le \|\X A^*\|_{op} \|\B\|_{op}$, recall  that 
		$
		A^*  = \Sigma^{-1}\Sigma_{XZ}.
		$
		The fact that $\Sigma^{-1/2}\X_{i\cdot}$ is  $\gamma_X$ sub-Gaussian implies
		$
		A^{*T}\X_{i\cdot}
		$
		is $\g_X\sqrt{\|\Sigma_{ZX}\Sigma^{-1}\Sigma_{XZ}\|_{op}}$ sub-Gaussian (see, for instance, \cite{vershynin_2012}), hence $\g_X \sqrt{\|\Sigma_Z\|_{op}}$ sub-Gaussian. By using Theorem 5.39 in \cite{vershynin_2012} again, one has 
		\[
		\PP\left\{
		{1\over n}\|\X A^*\|_{op}^2 \lesssim \g_X^2\left( 1 +  {K\over n}\right)
		\right\}\ge 1-2e^{-cK}.
		\]
		In conjunction with $\|\B\|_{op}^2 = O(m)$, we then conclude
		$$
		{1\over n}\|\X L^*\|_{op}^2 = O(m).
		$$
		This completes the proof. 
	\end{proof}

	\subsection{Proofs of Lemmas \ref{lem_unif_B} and \ref{lem_sparse_psi} in \cite{bing2020adaptive}}\label{sec_proofs_lemma_Theta}
	
	\begin{proof}[Proof of Lemma \ref{lem_unif_B} in \cite{bing2020adaptive}]
		Let $G\in \RR^{K\times m}$ be a random matrix whose entries are i.i.d. $N(0,1)$. Write its SVD as $G = V_GD_GU_G^T$ where $U_G\in \RR^{m\times K}$ contains the $K$ right singular vectors. Since $GQ$ has the same distribution as $G$ for any orthogonal matrix $Q\in\RR^{m\times m}$, columns of $U_G$ are uniformly distributed over the families of $K$ orthonormal vectors. Hence $U_G$ has the same distribution as $U$ and it remains to show
		\[
		\left\|U_G U_G^T \Ps_{j\cdot}\right\|_2^2  =  O_p\left(\left\|\Ps_{j\cdot}\right\|_2^2{K\over m}\right).
		\]
		To this end, we have 
		\begin{align*}
		\left\|U_G U_G^T \Ps_{j\cdot}\right\|_2^2 &=  \left\|U_G^T \Ps_{j\cdot}\right\|_2^2 =  \left\|D_G^{-1}V_G^T G \Ps_{j\cdot}\right\|_2^2\le {\|G \Ps_{j\cdot}\|_2^2 \over \lambda_K(GG^T)}.
		\end{align*}
		Concentration inequalities for random matrices with i.i.d. Gaussian entries yield (see, for instance, \cite[Corollary 3.35]{vershynin_2012})
		\[
		\PP\left\{\lambda_K(GG^T) \ge (\sqrt{m} - \sqrt{K}-t)^2
		\right\} \ge 1-2e^{-t^2/2},\quad \forall t>0.
		\]
		By noting that $G\Ps_{j\cdot} \sim N(0, \|\Ps_{j\cdot}\|_2^2\cdot {\bm I}_K)$, invoking Lemma \ref{lem_quad} yields 
		\[
		\PP\left\{
		{\Ps_{j\cdot}}^T GG^T \Ps_{j\cdot}  > \|\Ps_{j\cdot}\|_2^2\left(
		\sqrt{K} + \sqrt{2t}
		\right)^2
		\right\}\le e^{-t},\quad \forall t>0. 
		\]
		Combining these two displays and using $K=o(m)$ complete the proof. The proof of 
		\[
		\left|e_\ell^TU_G U_G^T \Ps_{j\cdot}\right| =  O_p\left(\left\|\Ps_{j\cdot}\right\|_2{K\over m}\right),
		\]
		for any $1\le \ell \le m$, follows immediately by noting that 
		\[
		\left|e_\ell^TU_G U_G^T \Ps_{j\cdot}\right| \le \left\|U_G U_G^T e_\ell\right\|_2 \left\|U_G U_G^T\Ps_{j\cdot}\right\|_2.
		\]
		The second statement follows immediately from the first result and $\|\Ps_{j\cdot}\|_2^2\le m\|\Ps\|_\i^2 = O(m)$
	\end{proof}
	
	\begin{proof}[Proof of Lemma \ref{lem_sparse_psi} in \cite{bing2020adaptive}]
		Pick any $1\le j\le p$. We have
		\begin{align*}
		\left\|P_{\B}\Ps_{j\cdot}\right\|_2^2  &= \left\|{\B}^T(\B{\B}^T)^{-1}\B\Ps_{j\cdot}\right\|_2^2\\
		&\le \|{\B}^T(\B{\B}^T)^{-1}\|_{op}^2\|\B\Ps_{j\cdot}\|_2^2\\
		& = \lambda_K^{-1}(\B{\B}^T) \sum_{k=1}^K |{\B_{k\cdot}}^T \Ps_{j\cdot}|^2.
		\end{align*}
		The result follows from $\lambda_K(\B{\B}^T) \ge c' m$ and 
		$|{\B_{k\cdot}}^T \Ps_{j\cdot}| \le d \|\Ps\|_\i\|\B\|_\i$. 
		
		Pick $1\le \ell \le m$. Similar arguments yield 
		\begin{align*}
		\left|e_\ell^T P_{\B}\Ps_{j\cdot}\right|  &= \left|e_\ell^T{\B}^T(\B{\B}^T)^{-1}\B\Ps_{j\cdot}\right|\\
		&\le \|{\B} e_\ell\|_2 \|(\B{\B}^T)^{-1}\|_{op}^2\|\B\Ps_{j\cdot}\|_2\\
		&\le  K d ~ {\|\Ps\|_\i\|\B\|_\i^2 \over \lambda_K(\B{\B}^T)}.
		\end{align*}
		The proof is complete.
	\end{proof}

	\subsection{Technical lemmas for controlling the stochastic terms}\label{sec_proof_auxiliary_lemma}
	
	Recall that $\Eps = \W B^* + \E$ and $\P$, $\Q$ are defined in (\ref{def_P_Q_lbd2}). Further recall that 
	$$
	V_\eps = \tr(\Gamma_{\eps}) = \g_w^2 \|\Sigma_W^{1/2} B^*\|_F^2 + m\g_e^2,\qquad \Gamma_{\eps} = \g_w^2 B^{*T}\Sigma_W \B + \g_e^2 \bI_m.
	$$
	We first state a lemma which studies the tail behaviour of $\Eps$. 
	\begin{lemma}\label{lem_eps_subgaussian}
		Under Assumption \ref{ass_error},  $\Eps_{ij}$ is $\g_{\Eps_j}$ sub-Gaussian for any $i\in [n]:=\{1,\ldots,n\}$ and $j\in [m]$ with
		\[
		\g_{\Eps_j}^2 =  \g_w^2\left[(B_j^*)^T\Sigma_W B_j^*\right]+ \g_e^2.
		\]
		Furthermore, the random vector $\Eps_j$ is $\g_{\Eps_j}$ sub-Gaussian for any $j\in [m]$ and the random vector $\Gamma_{\eps}^{-1/2}\Eps_{i\cdot}$ is sub-Gaussian with sub-Gaussian constant equal to $1$ for any $i\in [n]$
	\end{lemma}
	\begin{proof}
		Fix any $i\in [n]$ and $j\in[m]$. For any $t\ge 0$, by the independence of $\E$ and $\W$, we have
		\begin{align*}
		\EE[\exp(t\Eps_{ij})] & = \EE[\exp(t\W_{i\cdot}^TB_j^*)]\cdot \EE[\exp(t\E_{ij})]\\
		&\le \exp\left(
		t^2\g_w^2\|\Sigma_W^{1/2}B_j^*\|_2^2/2
		\right)\exp(t^2\g_e^2/2)
		\end{align*}
		where we use Assumption \ref{ass_error} in the second line. This proves the first claim. The second claim follows immediately from the i.i.d. property over $1\le i\le n$. To prove the third result, for any $i\in [n]$ and any fixed $u\in \RR^m$, we have
		\begin{align*}
		&\EE\left[
		\exp\left(u^T\Gamma_{\eps}^{-1/2}\Eps_{i\cdot}\right)
		\right]\\
		& = \EE\left[
		\exp\left(u^T\Gamma_{\eps}^{-1/2}B^{*T}\W_{i\cdot}\right)
		\right]\cdot \EE\left[\exp\left(u^T \Gamma_{\eps}^{-1/2}\E_{i\cdot}\right)\right]\\
		& \le \exp\left\{{1\over 2}\left[
		\g_w^2 u^T\Gamma_{\eps}^{-1/2}B^{*T}\Sigma_W\B  \Gamma_{\eps}^{-1/2}u
		+
		\g_e^2u^T\Gamma_{\eps}^{-1/2}\bI_m \Gamma_{\eps}^{-1/2}u
		\right]
		\right\}\\
		& = \exp\left(\g_w^2 \|u\|_2^2 / 2\right).
		\end{align*}
		We used the independence between $\W$ and $\E$ in the first equality and used Assumption \ref{ass_error} to derive the second line. This completes the proof. 
	\end{proof}
	
	We present several lemmas which control different terms related with $\Eps$, $\W$ and $\E$. The following lemma states the deviation inequality of $\|\P\Eps\|_F^2$, $\|\P^{1/2}\Eps\|_F^2$ and $\|\Eps^T\Q\X L^*\|_F$.
	
	\begin{lemma}\label{lem_P_eps}
		Under Assumption \ref{ass_error},  one has  
		\begin{align*}
		&\PP\left\{\left\|\P\Eps\right\|_F^2 \le V_\eps \left(\sqrt{\tr(\P^2)} + \sqrt{2\|\P^2\|_{op}\log(m / \epsilon)} \right)^2\right\} \ge 1-\epsilon\\
		&\PP\left\{\left\|\P^{1/2}\Eps\right\|_F^2 \le V_\eps \left(\sqrt{\tr(\P)} + \sqrt{2\|\P\|_{op}\log(m / \epsilon)} \right)^2\right\} \ge 1-\epsilon\\
		&\PP\left\{\left\|\Eps^T\Q\X L^*\right\|_F \le \sqrt{nV_\eps} \sqrt{\|\Q\|_{op}Rem_2(L^*)\log(m/\epsilon)}\right\} \ge 1-\epsilon.
		\end{align*}
	\end{lemma}
	\begin{proof}
		We now prove the first result. 
		Note that $\|\P\Eps\|_F^2 = \sum_{j=1}^m \|\P\Eps_j\|_2^2$. Pick any $j\in [m]$. Since $\Eps_j$ is $\g_{\Eps_j}$ sub-Gaussian from Lemma \ref{lem_eps_subgaussian}, applying Lemma \ref{lem_quad} with $\g_{\xi} = \g_{\Eps_j}$ and $K = \P^2$ yields 
		\[
		\PP\left\{
		\Eps_j^T \P^2\Eps_j > \g_{\Eps_j}^2\left(
		\sqrt{\tr(\P^2)} + \sqrt{2\|\P^2\|_{op}t}
		\right)^2
		\right\} \le e^{-t},
		\]
		for all $t\ge 0$. Since $\sum_{j=1}^m \g_{\Eps_j}^2 = V_\eps$, choosing $t = \log(m/\epsilon)$ and taking the union bounds over $1\le j\le m$ complete the proof of the first result. By the same arguments, the second result follows immediately, and we also have
		\[
		\PP\left\{
		\left\|\Eps^T\Q\X L^*\right\|_F^2 \le {V_\eps} \left(\sqrt{\tr(D)} + \sqrt{2\|D\|_{op}\log(m / \epsilon)} \right)^2
		\right\} \ge 1-\epsilon
		\]
		where $D = \Q\X L^*(L^*)^T \X^T\Q$. The third result then follows by observing that 
		\[
		\tr(D) \le \|\Q\X L^*\|_F^2 \le n \|\Q\|_{op}\cdot Rem_2(L^*)
		\]
		from (\ref{disp_QXL}) and $\|D\|_{op}\le \tr(D)$.
	\end{proof}
	
	The following lemma provides the deviation inequality of $\max_{1\le j\le p}\|\X_j^T\Q \Eps\|_2$. Recall that $M = n^{-1}\X^T\Q^2 \X$ and $\Gamma_{\eps} = \g_w^2B^{*T}\Sigma_W \B + \g_e^2\bI_m$.
	\begin{lemma}\label{lem_XQeps}
		Under Assumption \ref{ass_error}, with probability $1-\epsilon$, one has 
		\[
		\max_{1\le j\le p} \|\X_j^T\Q \Eps\|_2^2 \le  \left(\sqrt{{\rm tr}(\Gamma_{\eps})}+\sqrt{2 \|\Gamma_{\eps}\|_{op} \log(p / \epsilon)}\right)^2 n\max_{1\le j\le p}M_{jj}.
		\]
	\end{lemma}
	\begin{proof}
		Fix any $1\le j\le p$. Notice that 
		$$
		\|\X_j^T\Q\Eps\|_2^2 =  \X_j^T \Q\Eps \Gamma_{\eps}^{-1/2} \Gamma_{\eps} \Gamma_{\eps}^{-1/2}\Eps^T\Q \X_j.
		$$
		We first show that 
		$\Gamma_{\eps}^{-1/2}\Eps^T\Q \X_j$ is $\sqrt{nM_{jj}}$ sub-Gaussian. By independence across $i\in [n]$, we have, for any $v\in \RR^m$, 
		\begin{align*}
		\EE\left[
		\exp\left(v^T\Gamma_{\eps}^{-1/2}\Eps^T\Q \X_j\right)
		\right] & = \prod_{i=1}^n \EE\left[
		\exp\left(v^T\Gamma_{\eps}^{-1/2}\Eps_{i\cdot} e_i^T\Q \X_j\right)\right]\\
		&\le \prod_{i=1}^n 
		\exp\left(\|v\|_2^2 e_i^T\Q \X_j\X_j^T \Q e_i / 2\right)\\
		& = \exp(\|v\|_2^2 nM_{jj} / 2).
		\end{align*}
		We used Lemma \ref{lem_eps_subgaussian} to derive the second line. Then invoke Lemma \ref{lem_quad} with $\g_{\xi}= \sqrt{nM_{jj}}$ and $K = \Gamma_{\eps}$ gives 
		\[
		\PP\left\{
		\|\X_j^T\Q\Eps\|_2^2 > nM_{jj}\left(
		\sqrt{\tr(\Gamma_{\eps})} + \sqrt{2\|\Gamma_{\eps}\|_{op} t} 
		\right)^2 
		\right\}\le e^{-t},\quad \text{for all }t\ge0.
		\]
		Choose $t = \log(p / \epsilon)$ and take the union bounds over $j\in[p]$ to complete the proof. 
	\end{proof}
	
	The next tail inequality is for $\max_{1\le j\le p}\|\X_j^T\E\|_2$, derived based on the quadtratic form of a sub-Gaussian random vector. 
	\begin{lemma}\label{lem_XE}
		Under Assumption \ref{ass_error}, with probability $1-\epsilon$, one has 
		\[
		\max_{1\le j\le p} \|\X_j^T\E\|_2^2 \le  \g_e^2\left(\sqrt{m}+ \sqrt{2\log(p/\epsilon)}\right)^2n\max_{1\le j\le p}\wh \Sigma_{jj}.
		\]
	\end{lemma}
	\begin{proof}
		Pick any $1\le j\le p$. Note that $\E^T\X_j$ is $\g_e\sqrt{n\wh\Sigma_{jj}}$ sub-Gaussian. Indeed, for any $u\in \RR^m$, we have 
		\begin{align*}
		\EE[\exp(\langle u, \E^T\X_j\rangle)] &= \prod_{t=1}^n \prod_{i=1}^m\EE[\exp(u_i \X_{tj}\E_{ti})]\\
		&\le \prod_{t=1}^n \prod_{i=1}^m\exp(u_i^2\X_{tj}^2\g_e^2/2)]\\
		&= \exp(\|u\|_2^2 \X_j^T\X_j \g_e^2/2) 
		\end{align*}
		by using the independence of entries of $\E$ and  $\E_{ti}$ is $\g_e$ sub-Gaussian. Applying Lemma \ref{lem_quad} with $K = \bI_m$ and $\g_{\xi} = \g_e\sqrt{n\wh\Sigma_{jj}}$ gives 
		\[
		\PP\left\{
		\X_j^T\E\E^T\X_j > \g_e^2 n\wh\Sigma_{jj}\left(
		\sqrt{m}+ \sqrt{2t}
		\right)^2
		\right\} \le e^{-t}.
		\]
		Choosing $t = \log(p/\epsilon)$ and taking the union bounds over $1\le j\le p$ conclude the proof. 
	\end{proof}
	
	%A similar tail inequality can also be derived for $\max_{1\le j\le p}\|\X_j^T\W B^*\|_2$. Let $V_W = \tr((B^*)^T\Sigma_{W}B^*)$.
	%\begin{lemma}\label{lem_XW}
	%	Under Assumption \ref{ass_error}, with probability $1-\epsilon$, one has 
	%	\[
	%	\max_{1\le j\le p} \|\X_j^T\W B^*\|_2^2 \le  \g_e^2\left(\sqrt{V_W}+ \sqrt{2\Lambda_1\log(p/\epsilon)}\right)^2n\max_{1\le j\le p}\wh \Sigma_{jj}.	\]
	%\end{lemma}
	%\begin{proof}
	%	Pick any $1\le j\le p$ and write $\X_j^T \W B^* = \X_j^T \wt \W (\Sigma_{W}^{1/2}B^*)$ with $\wt \W = \W \Sigma^{-1/2}$. Note that $\wt \W^T\X_j$ is $\g_w\sqrt{n\wh\Sigma_{jj}}$ sub-Gaussian by using the similar argument in the proof of Lemma \ref{lem_XE}. Applying Lemma \ref{lem_quad} with $K = \Sigma_{W}^{1/2}B^*(B^*)^T\Sigma_{W}^{1/2}$ and $\g_{\xi} = \g_w\sqrt{n\wh\Sigma_{jj}}$ gives 
	%	\[		\PP\left\{
	%	\X_j^T\W B^*(B^*)^T\W^T\X_j > \g_w^2 n\wh\Sigma_{jj}\left(
	%	\sqrt{V_W}+ \sqrt{2\Lambda_1 t}
	%	\right)^2	\right\} \le e^{-t}.\]
	%	by recalling that $V_W = \tr((B^*)^T\Sigma_{W}B^*)$ and  $\Lambda_1 = \|(B^*)^T\Sigma_{W}B^*\|_{op}$.	Choosing $t = \log(p/\epsilon)$ and taking the union bounds over $1\le j\le p$ conclude the proof. 
	%\end{proof}

	The following lemma states the tail behaviour of the operator norm of $\Sigma_W^{-1/2}\W^T\W\Sigma_W^{-1/2}$.

	\begin{lemma}\label{lem_W^TW}
		Under Assumption \ref{ass_error},  with probability $1- 2e^{-cK}$, one has
		\[
		{1\over n}\left\|\Sigma_W^{-1/2}\W^T\W\Sigma_W^{-1/2}\right\|_{op} \le 1 + C\left(\sqrt{K\over n} \vee {K \over n}\right)
		\]
		where $c=c(\g_w)$ and $C = C(\g_w)$ are positive constants. 
	\end{lemma}
	\begin{proof}
		The result follows directly from the proof of Theorem 5.39 in \cite{vershynin_2012}.
	\end{proof}
	
	The following lemma states the deviation inequality of $\|n^{-1}\Eps^T\Eps - \Sigma_{\eps}\|_F$.
	\begin{lemma}\label{lem_eps_eps}
		Under Assumption \ref{ass_error}, one has
		\[
		\PP\left\{\left\|{1\over n}\Eps^T\Eps - \Sigma_{\eps}\right\|_F \le cV_\eps\left(\sqrt{\log(m) \over n} \vee {\log(m) \over n}\right) \right\}\ge 1-2m^{-c'}
		\]
		for some absolute constants $c, c'>0$.
	\end{lemma}
	\begin{proof}
		Fix any $j,\ell \in [m]$. We first upper bound $|n^{-1}\Eps_j^T\Eps_{\ell} - (\Sigma_{\eps})_{j\ell}|$. Since entries of $\Eps_j$ and $\Eps_{\ell}$ are $\g_{\Eps_j}$ and $\g_{\Eps_\ell}$ sub-Gaussian, respectively, from Lemma \ref{lem_eps_subgaussian}, invoking Lemma 15 in \cite{essential} with $ t = \min\{\sqrt{\log(m)/n} \vee \log(m)/n\}$  gives 
		\[
		\PP\left\{
		\left| {1\over n}\Eps_j^T\Eps_{\ell} - (\Sigma_{\eps})_{j\ell}\right| \le c \g_{\Eps_j} \g_{\Eps_\ell} \left(\sqrt{\log(m) \over n} \vee {\log(m) \over n}\right)
		\right\}\ge 1- 2m^{-c'}
		\]
		for some absolute constants $c, c'>0$.  The result then follows by taking the union bounds over $1\le j,\ell \le m$ and noting that 
		$\sum_j\sum_\ell \g_{\Eps_j}^2\g_{\Eps_\ell}^2 = V_\eps^2$. 
	\end{proof}
	
	\subsection{An algebraic fact and one auxillary lemma}
	We first state an algebraic fact that is used in our analysis.
	\begin{fact}\label{fact_QX}
		Let $\Q$ be defined in (\ref{def_P_Q_lbd2}). Then 
		\[
		\Q \X = \lambda_2 \X(\wh \Sigma + \lambda_2 \bI_p)^{-1}.
		\]
	\end{fact}
	\begin{proof}
		The proof follows by noting that
		\[
		\Q \X = \X - \X (\X^T\X + n\lambda_2\bI_p)^{-1}\X^T\X = n\lambda_2\X(\X^T\X + n\lambda_2\bI_p)^{-1}
		\]
		and the definition $\wh\Sigma = \X^T\X/n$.
	\end{proof}
	
	The following lemma is used in our analysis. The tail inequality is for a quadratic form of sub-Gaussian random vector. It is a slightly simplified version of Lemma 8 in \cite{Hsu2014}. 
	\begin{lemma}\label{lem_quad}
		Let $\xi\in \RR^d$ be a $\gamma_\xi$ sub-Gaussian random vector. For all symmetric positive semidefinite matrices $K$, and all $t\ge 0$, 
		\[
		\PP\left\{
		\xi^T K\xi > \gamma_\xi^2\left(
		\sqrt{{\rm tr}(K)}+ \sqrt{2\|K\|_{op}t}
		\right)^2
		\right\} \le e^{-t}.
		\] 
	\end{lemma}
	\begin{proof}
		From Lemma 8 in \cite{Hsu2014}, one has
		\[
		\PP\left\{
		\xi^T K\xi > \gamma_\xi^2\left(
		\tr(K) + 2\sqrt{\tr(K^2)t}+2\|K\|_{op}t
		\right)
		\right\} \le e^{-t},
		\] 
		for all $t\ge 0$. The result then follows from $\tr(K^2) \le \|K\|_{op}\tr(K)$.
	\end{proof}
	
	%The next tail probability is for a sum of independent centered sub-exponential random variables. 
	
	%\begin{lemma}[Corollary 5.17 in \cite{vershynin_2012}]
	%	Let $X_1, \ldots, X_n$ be independent centered sub-exponential random variables, and let $\zeta= \max_i \|X_i\|_{\psi_1}$. Then, for every $t\ge 0$, we have 
	%	\[
	%		\PP\left\{				{1\over n}\left|\sum_{i=1}^n X_i\right| \ge \zeta t
	%		\right\} \le 2\exp \left[
	%				-c \min\left( t^2, t\right)n
	%		\right]		\]
	%	where $c>0$ is an absolute constant. 
	%\end{lemma}
	
	\section{Comparison of the multivariate ridge estimation and the reduced-rank estimation}\label{sec_comp_rr}
	
	\subsection{Comparison with the reduced-rank estimator}
	
	In this section, we state our reasoning for using the ridge penalty in (\ref{est_F}) rather than the commonly used low-rank approach. In particular, we compare our estimator (\ref{est_F}), or equivalently the multivariate ridge regression, with the reduced-rank estimator under our model (\ref{model}), $Y= (L^*)^TX + (B^*)^TW + E$ when $\Ps = 0$. Since $L^* = A^*B^*$ exhibits a low-rank structure when both $p$ and $m$ are relatively large comparing to $K$, one could estimate $L^*$ by the reduced-rank estimator \citep{i75,i08,bunea2011}
	\begin{equation}\label{est_rr}
	\wh L^{(RR)} = \arg\min_{L} \left\|\Y - \X L\right\|_F^2 +  \mu\cdot \rank(L)
	\end{equation}
	for some tuning parameter $\mu > 0$. 
	
	There have been extensive research on the estimation of a low-rank matrix in both regression and matrix completion settings, for instance, \cite{bing2019, bunea2011,bunea2012,cp, ct, rt, rv98,giraud2011,GiraudBook,ob, nw, koltchinskii2011,yuan2007}, a list that is far from exhaustive. In regression setting, the reduced-rank estimator $\wh L^{(RR)}$ is shown to be optimal for both prediction and estimation in high-dimensional settings \citep{bunea2011}. Despite of its popularity, we give several reasons why the multivariate ridge-type estimator is more suitable in our setting.
	
	First, we focus on the comparison of the convergence rate of prediction error. Suppose 
	part (b) of Assumption \ref{ass_rates} holds. From Corollary \ref{cor_pred}, by using $s_*=0$, $V_\eps = O(Km)$ and $\sum_k\sigma_k \le q \sigma_1$ with $q = \rank(\X)$, 
	one can deduce that our estimator $\X \wh F$ satisfies (notice that $F^* = L^*$ as $\Ps = 0$)
	\[
	{1\over n}\left\|\X \wh F - \X L^*\right\|_F^2  \lesssim 
	{\sigma_1qKm\over \lambda_2 n} + {\sigma_1 Km\log(m/\epsilon) \over \lambda_2 n}+ \lambda_2 \|L^*\|_F^2.
	\]
	When $\log(m) = O(q)$, choosing $\epsilon = m^{-1}$ and optimizing the above rate over $\lambda_2$ yield
	\[
	{1\over n}\left\|\X \wh F - \X L^*\right\|_F^2	\lesssim
	\sqrt{\sigma_1 Kqm\over n}\|L^*\|_F.
	\]
	On the other hand, \cite{bunea2011} showed that the reduced-rank estimator $\wh L^{(RR)}$ in (\ref{est_rr}) has the following prediction error 
	\begin{align*}
	{1\over n}\left\|\X \wh L^{(RR)} - \X L^*\right\|_F^2  &\lesssim  {K\over n}\left\|\W B^*+\E\right\|_{op}^2\\ &\lesssim  {Km\over n}\left\|(\W B^*+\E)\Sigma_{\eps}^{-1/2}\right\|_{op}^2 \lesssim {K(n+m)m\over n}
	\end{align*}
	where the first inequality can be proved in the same way as \cite{koltchinskii2011,bunea2011}, the second inequality holds because the operator norm of $\Sigma_{\eps}= (B^*)^T\Sigma_W B^* + \tau^2\bI_m$ is of order $m$ under part (b) of Assumption \ref{ass_rates}, and the last inequality is due to the deviation bounds of the operator norm of random matrices whose rows are i.i.d. sub-Gaussian random vectors \citep{vershynin_2012}. Clearly, by $\|L^*\|_F^2 \le K\|L^*\|_{op}^2$ and recalling that $q = \rank(\X) \le \min\{n,p\}$, the rate of our estimator $\X \wh F$ is potentially faster if 
	\[
	{\sigma_1\|L^*\|_{op}^2 } \le {nm \over q}\left(1 + {m\over n}\right)^2
	\]
	which usually holds, especially when $m$ is large (see the discussion on $\|L^*\|_F$ in Section \ref{sec_dis_cond}).

	To the best of our knowledge, the reduced-rank estimator only achieves the minimax  rate when the error covariance matrix has a bounded operator norm \citep{koltchinskii2011,rt}. However, as shown above, the error in our model (\ref{model}) has covariance matrix $(B^*)^T\Sigma_W B^* + \tau^2\bI_m$ whose operator norm is order of $m$.  In this case, the reduced-rank estimator is no longer optimal. Instead, our ridge type estimator leverages the fact that $\|L^*\|_F$ is small and leads to a much faster prediction rate. This theoretical comparison is further corroborated by simulation studies in Appendix \ref{sec_sim_comp_rr}.
	
	Finally, from computational perspective, if we replace $\lambda_2 \|L\|_F^2$ in (\ref{est_F}) by the rank penalty $\lambda_2\cdot \rank(L)$, the resulting optimization becomes non-convex and is computationally challenging. For computational convenience, one may instead penalize the nuclear norm of $L$ rather than its rank such that the resulting optimization problem can be solved by the ADMM algorithm. However, to establish statistical guarantees for the corresponding estimator, on top of the RE condition for the group-lasso penalty, this approach may require additional conditions on the design matrix due to the nuclear norm regularization \citep{koltchinskii2011}. In contrast, the proposed estimator  (\ref{est_F}) can be computed efficiently and has much weaker restrictions on the design matrix for provable guarantees.
	
	\subsection{Empirical comparison with the reduced-rank estimator}\label{sec_sim_comp_rr}
	
	In this section, we compare the multivariate ridge regression with the reduced-rank estimator under model $\Y = \X L + \W B+\E$ when $\Psi=\Theta = 0$. Comparing to the commonly studied low-rank regression in the literature,  the noise level here is much larger since it follows a factor structure with diverging eigenvalues. 
	
	In the following we will demonstrate  the advantage of using the multivariate ridge regression in (\ref{est_F}) with $\lambda_1 = 0$ over the reduced-rank estimator. We first show that the ridge-type estimator has smaller PMSE than the reduced-rank estimator when $\|L\|_F$ is small or moderate. Second, we further show that the ridge-type estimator is more robust to the noise level than the reduced-rank estimator.
	
	We follow the data generating process in simulation studies and choose $n = 80$, $p = 120$, $m = 30$ and $\rho = 0.3$.  To change the strength of $\|L\|_F$, we vary $\eta \in \{0.05, 0.1, 0.15, \ldots, 0.5, 0.55\}$. For each $\eta$, we randomly generate $100$ datasets and the averaged PMSEs of $\wh L^{(RR)}$ (RR) and $\wh L^{(Ridge)}$ (Ridge) are shown in Figure \ref{fig_rr}. Note that we provide the true $K$ for $\wh L^{(RR)}$. As seen in the first panel, Ridge has much smaller PMSE than RR when $\eta$ is small. Moreover, it seems that RR does not provide consistent prediction when the noise has diverging eigenvalues. 
	
	To show the robustness to the noise level, we use the same setting and fix $\eta = 0.2$. Recall that $W_{ik} \sim N(0, \sigma_W^2 = 1)$. We then change the noise level by varying $\sigma_W^2 \in \{0.6, 0.8, 1, \ldots, 2.8, 3\}$. The averaged PMSEs of RR and Ridge are shown in the second panel of Figure \ref{fig_rr}. It is easy to see that Ridge is much more robust to the magnitude of the noise level and outperforms RR by a large margin. 
	
	\begin{figure}[H]
		\begin{tabular}{cc}
			\includegraphics[width=.45\textwidth]{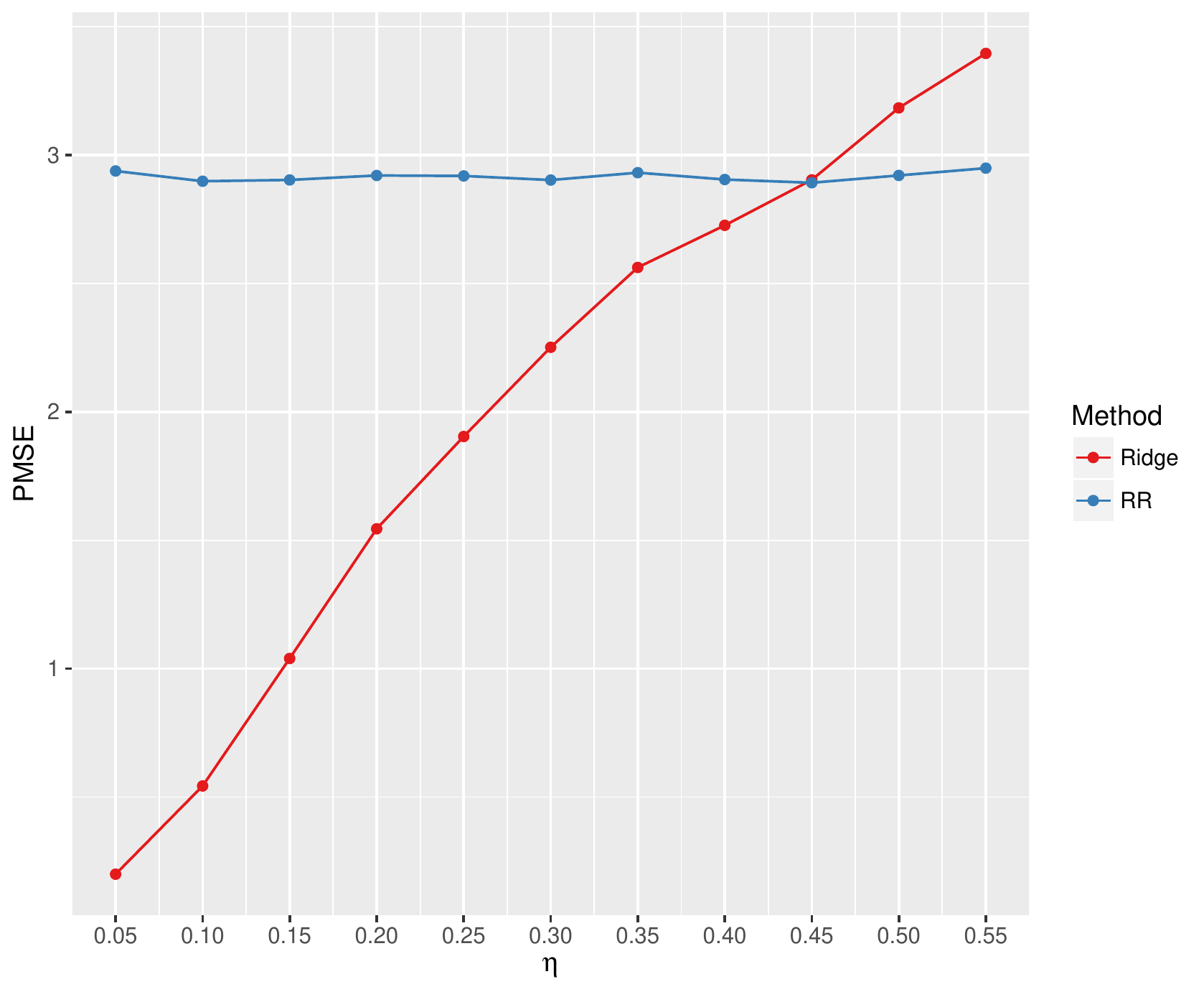}&
			\includegraphics[width=.45\textwidth]{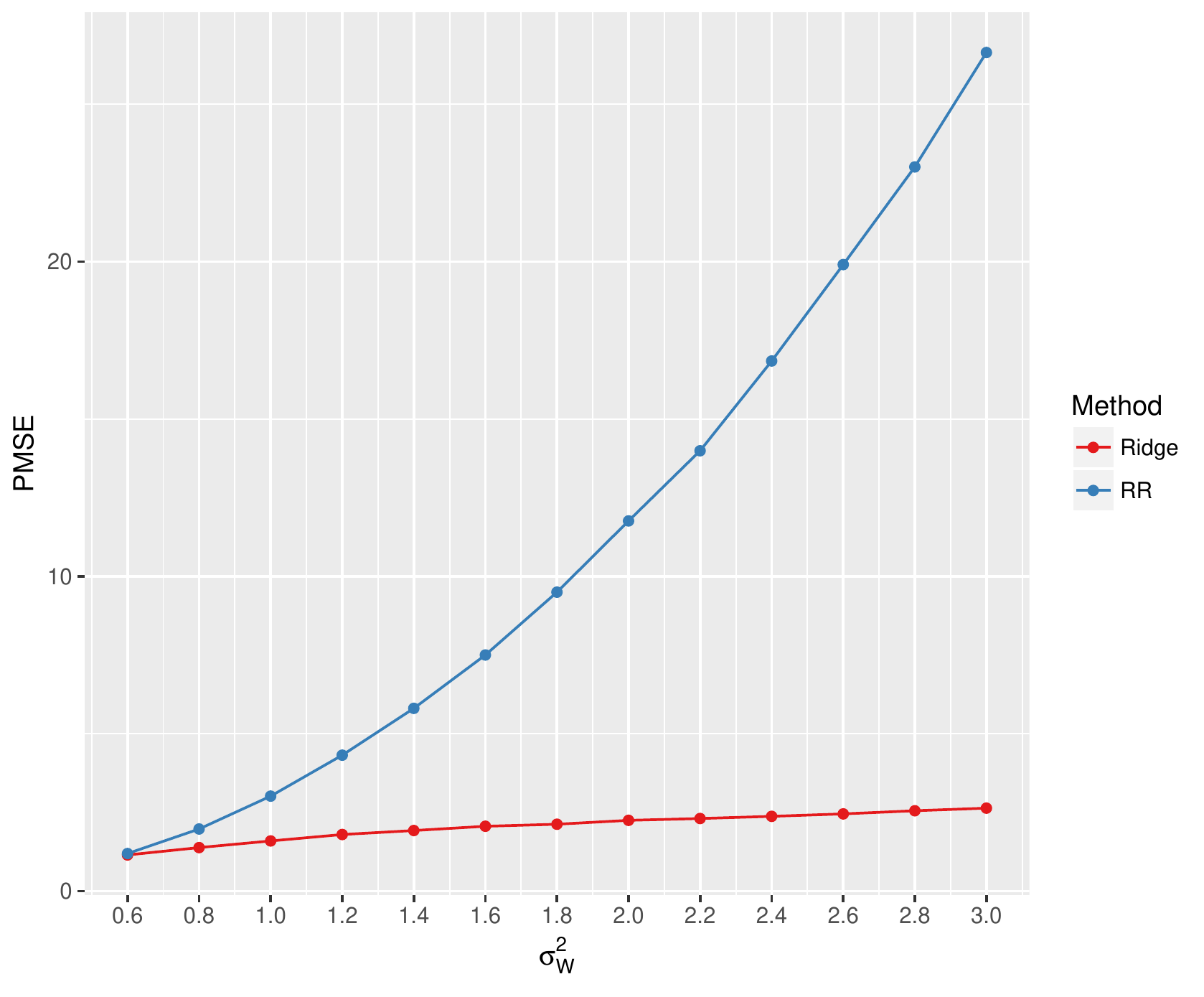}
		\end{tabular}
		\caption{The averaged PMSEs of RR and Ridge when we vary $\eta$ and $\sigma_W$ separately. }
		\label{fig_rr}
	\end{figure}

	\section{Supplementary simulation results}\label{sec_supp_sim}

	\subsection{Simulation results on comparison with other SVA-related algorithms}

	We conduct simulation studies to compare SVA \cite{Lee2017} with other related methods, such as CATE \cite{wang2017} and BCconf \cite{McKennan19}\footnote{We use the \textsf{cate} package and the code from \url{https://github.com/chrismckennan/BCconf.git} for CATE and BCconf, respectively, both implemented in \textsc{R}}.
	We use the same simulation setting described in Section \ref{sec_simulation}, and consider two scenarios: (1) small $p$ and small $m$ ($n=100$, $m = p = 20$); (2) small $p$ and large $m$ ($n=100$, $m = 150$, $p = 20$). Table \ref{tab_homo} contains the averaged RSSE for the two scenarios under homoscedastic errors. We provide the true $K$ for all three methods. SVA has the best performance while BCconf is only slightly worse than SVA. CATE tends to have much larger errors possibly due to the fact that the rows of $\Psi^*$ are not sufficiently sparse. For heteroscedastic errors, Table \ref{tab_hetero} shows the similar pattern, but the performance of BCconf deteriorates as the noise becomes more heteroscedastic (corresponding to a larger $\alpha$).

	\begin{table}[H]
		\centering
		\caption{Averaged RSSE of SVA, CATE and BCconf over 100 repetitions for homoscedastic cases. The numbers in parenthesis are standard errors.}
		\label{tab_homo}
		\begin{tabular}{c|ccc|ccc}
			\hline 
			&\multicolumn{3}{c|}{$n=100$, $p=m=20$} & \multicolumn{3}{c}{$n=100$, $p=20$, $m=150$}\\
			\hline
			$\eta$  & SVA & CATE & BCconf & SVA & CATE & BCconf  \\ 
			\hline
			$0.1$ & 2.92 (0.16) & 11.98 (1.12) & 3.25 (0.38) & 	7.89 (0.14) & 26.15 (0.63) & 8.04 (0.21)\\ 
			$0.3$ & 2.99 (0.19) & 12.06 (1.34) & 3.30 (0.39) & 8.07 (0.16) & 26.21 (0.64) & 8.20 (0.20)\\ 
			$0.5$ & 3.07 (0.21) & 12.10 (1.12) & 3.44 (0.45) & 8.28 (0.22) & 26.36 (0.61) & 8.43 (0.27)\\ 
			$0.7$ & 3.23 (0.24) & 12.12 (1.22) & 3.55 (0.39) & 8.68 (0.28) & 26.40 (0.71) & 8.83 (0.33)\\ 
			$0.9$ & 3.37 (0.26) & 12.14 (1.27) & 3.67 (0.43) & 9.15 (0.34) & 26.68 (0.63) & 9.31 (0.38)\\ 
			$1.1$ & 3.61 (0.31) & 12.10 (1.33) & 3.89 (0.44) & 9.78 (0.44) & 26.72 (0.88) & 9.91 (0.45)\\  
			$1.3$ & 3.81 (0.39) & 12.07 (1.56) & 4.07 (0.45) & 10.47 (0.63) & 27.12 (0.66) & 10.61 (0.64)\\ 
			\hline
		\end{tabular}
	\end{table}
	
	\begin{table}[H]
		\centering
		\caption{Averaged RSSE of SVA, CATE and BCconf over 100 repetitions for heteroscedastic cases. The numbers in parenthesis are standard errors.}
		\label{tab_hetero}
		\begin{tabular}{r|ccc|ccc}
			\hline 
			&\multicolumn{3}{c|}{$n=100$, $p=m=20$} & \multicolumn{3}{c}{$n=100$, $p=20$, $m=150$}\\
			\hline
			$\alpha$  & SVA & CATE & BCconf & SVA & CATE & BCconf  \\ 
			\hline
			$0$ & 3.11 (0.22) & 12.01 (1.37) & 3.44 (0.43) &  8.30 (0.19) & 26.33 (0.72) & 8.45 (0.22) \\ 
			$1$ & 3.12 (0.28) & 13.10 (2.13) & 6.26 (2.64) & 8.32 (0.19) & 26.21 (1.23) & 17.66 (5.80) \\ 
			$2$ & 3.13 (0.30) & 14.10 (2.83) & 8.87 (4.68) & 8.32 (0.23) & 25.91 (2.05) & 22.54 (8.42) \\ 
			$3$ & 3.26 (0.32) & 15.07 (3.17) & 10.82 (5.76) 
			& 8.33 (0.28) & 27.24 (2.97) & 19.57 (7.61) \\  
			$4$ & 3.29 (0.39) & 15.74 (4.05) & 10.85 (5.95) 
			& 8.32 (0.29) & 28.56 (3.96) & 20.99 (6.85) \\ 
			$5$ & 3.32 (0.49) & 17.14 (4.31) & 11.56 (5.32) 
			& 8.25 (0.29) & 28.14 (3.61) & 20.26 (7.64) \\
			\hline
		\end{tabular}
	\end{table}

	\subsection{Simulation results when the first $K$ eigenvalues of ${\B}^T\Sigma_W{\B}$ are moderate}
	
	Part $(b)$ of Assumption \ref{ass_rates} requires all the eigenvalues of ${\B}^T\Sigma_W \B$ to grow with order $m$. In practice, when $m$ is large, some eigenvalues of ${\B}^T\Sigma_W \B$ might be moderate. We thus evaluate the performance of the proposed method in this scenario and compare it with other SVA-related methods, such as SVA \cite{Lee2017}, BCconf \cite{McKennan19} and CATE \cite{wang2017}. We consider the scenario $n = 100$, $p=20$ and $m = 150$ and use the same setting as described in Section \ref{sec_simulation} except that entries $B_{jk}$ of $B$ for $1\le j\le m$ and $1\le k\le K$ are independently drawn from $N(0, [\sigma_B]_k^2)$ where 
	\[
	[\sigma_B]_k^2 = {1\over m} m^{K-k\over K-1}, \quad \textrm{for }k = 1,\ldots, K.
	\]
	From the concentration inequalities of the operator norm of $m^{-1}{\B}^T \B - \textrm{diag}([\sigma_B]_1^2, \ldots, [\sigma_B]_K^2)$ together with Weyl's inequality, it is easy to show that, with probability tending to one as $K=o(m)$, 
	\[
	\Lambda_k \asymp m\cdot [\sigma_B]_k^2 = m^{K-k\over K-1},\quad \textrm{for }k = 1,\ldots, K.
	\]
	We set $\eta = 0.5$ and choose $\alpha \in \{0,1,\ldots,5\}$ to vary the heteroscedasticity. The left plot in Figure \ref{fig_small_B} depicts the averaged RSSE of Oracle, SVA, CATE, BCconf and HIVE, all of which use the true $K$. It shows that HIVE has worse performance than SVA when the eigenvalues of ${\B}^T \Sigma_W \B$ are moderate or small. However, since $K$ is unknown in practice and using the true $K$ also seems suboptimal when some eigenvalues of ${\B}^T \Sigma_W \B$ are small, we conduct a more informative comparison of SVA, CATE, BCconf and HIVE by using data-dependent estimates of $K$. Specifically, we use the ``eigenvalue distance''(ED) approach to estimate $K$ for SVA, CATE and BCconf.\footnote{This method is implemented in the \textsf{cate} package, see \cite{wang2017} and \cite{McKennan19} for more details. Another option of selecting $K$ in \textsf{cate} is the bi-cross validation. We did not consider it here due to its slow computation for large $n$ and $m$.} On the other hand, we use the criterion (\ref{select_K}) (Ratio), the permutation test (PA) and ED to estimate $K$ for HIVE. The right plot of Figure \ref{fig_small_B} shows the RSSE of all methods and HIVE-Ratio has the best performance in all settings. It is worth mentioning that HIVE-Ratio, HIVE-PA and HIVE-ED all tend to under-estimate $K$ in this scenario. This suggests that under-estimating $K$ may improve the performance of HIVE when  some eigenvalues of ${\B}^T \Sigma_W \B$ are moderate or small, in line with our discussion after Lemma \ref{lem_mis_K} in Section \ref{sec_remark_K}.
	
	\begin{figure}[ht]
		\centering
		\begin{tabular}{cc}
			\includegraphics[width=0.48\textwidth]{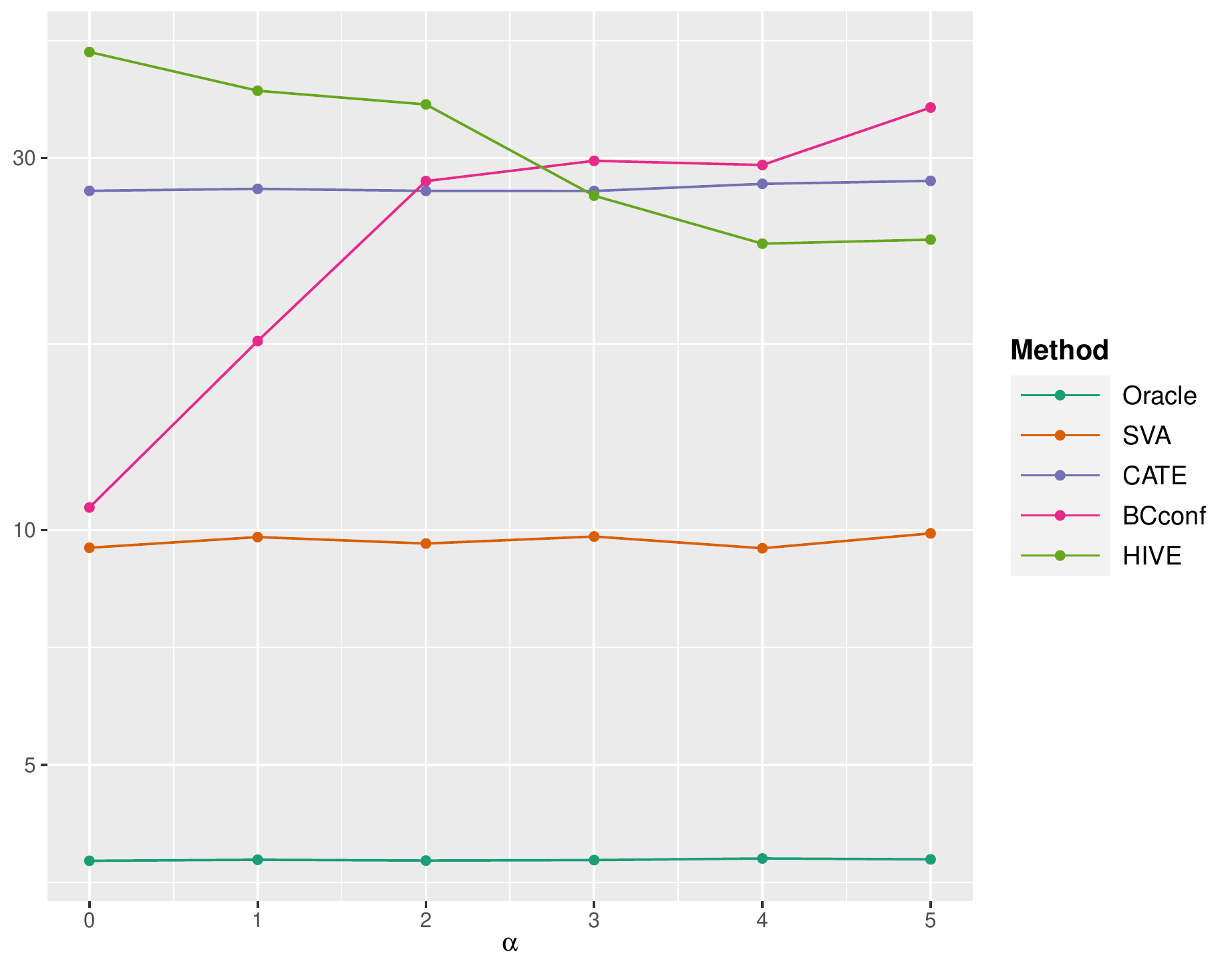}   & 
			\includegraphics[width=0.48\textwidth]{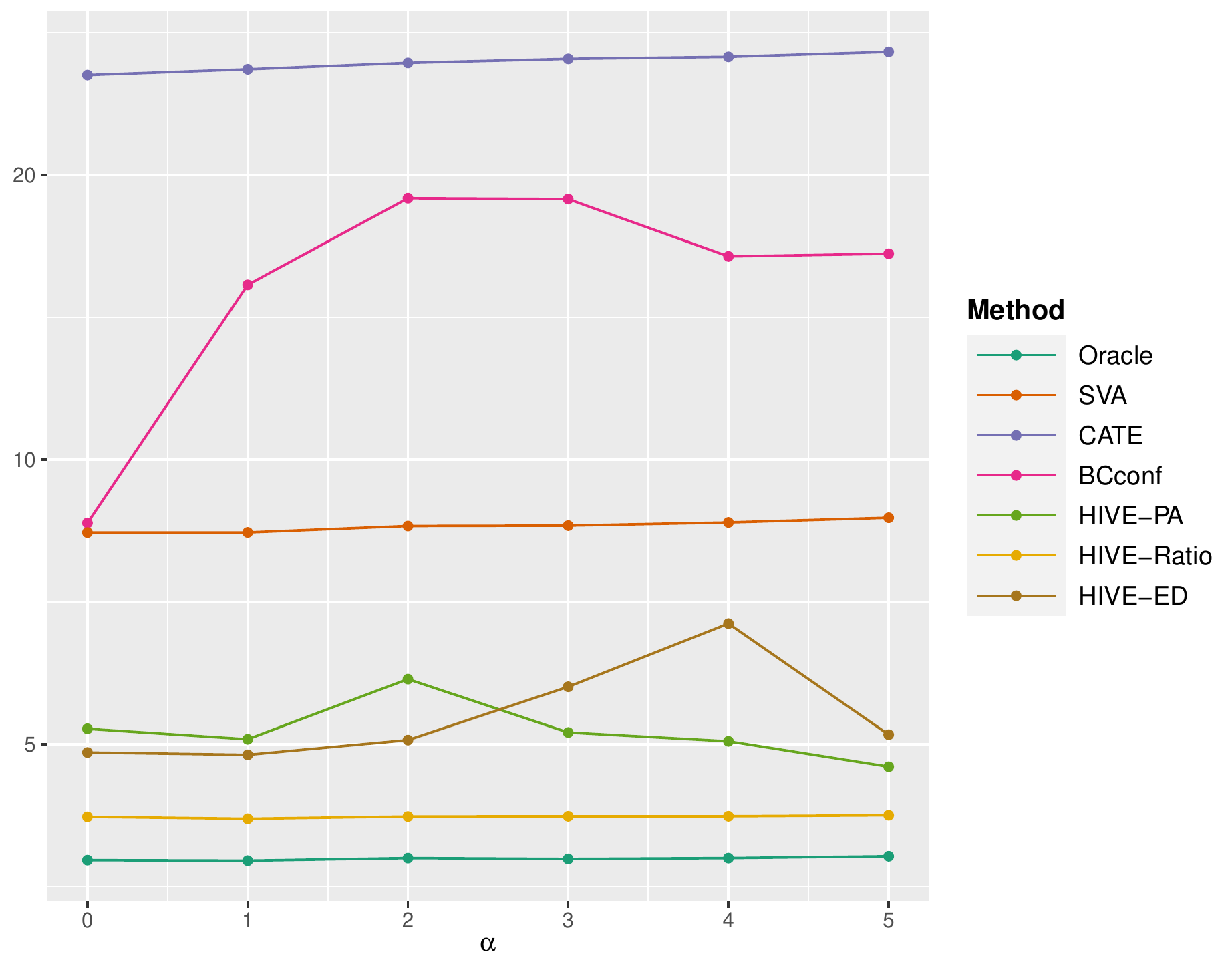}
		\end{tabular}
		\caption{The averaged RSSE of different methods. In the left plot, SVA, CATE, BCconf and HIVE are provided with the true $K$. In the right plot, all methods use data-dependent estimates of $K$.}
		\label{fig_small_B}
	\end{figure}

	\subsection{An alternative procedure of selecting $(\lambda_1,\lambda_2)$}\label{sec_supp_tuning}
	When selecting $(\lambda_1, \lambda_2)$ in the step of estimating $F^* = \Ps + L^*$ over a fine two-way grid is computationally burdensome, we suggest the following faster approach which, according to our simulation results, performs as well as the two-way grid search. Briefly speaking, it first searches $\lambda_2^{cv}$ which yields the smallest prediction error, over a grid of $\lambda_2$ with $\lambda_1$ choosing as a function of $\lambda_2$. Then by fixing $\lambda_2^{cv}$, it searches $\lambda_1^{cv}$ over a grid of $\lambda_1$ via a similar cross-validation criterion. This way of sequentially selecting $\lambda_2$ and $\lambda_1$ greatly reduces the computational cost comparing to the $k$-fold cross validation over a two-way grid of $(\lambda_1, \lambda_2)$.

	Specifically, we start with a grid $\mathcal{G}$ of $\lambda_2$ and for each $\lambda_2 \in \mathcal{G}$, we set 
	\[
	\lambda_1(\lambda_2) = c_0 \sqrt{\max_{1\le j\le p} M_{jj}(\lambda_2)}\left(\sqrt{m \over n} + \sqrt{2\log p \over n}\right)
	\]
	where 
	$M(\lambda_2) = n^{-1} {\bm X}^T Q^2_{\lambda_2} {\bm X}$ and 
	$c_0>0$ is some universal constant (our simulation reveals good performance with $c_0 = 4$). This choice of $\lambda_1(\lambda_2)$ is based on the theoretical rate in (\ref{rate_lbd1}) of Theorem \ref{thm_pred}. We then use $k$-fold cross validation to select $\lambda_2^{cv}$ which gives the smallest mean square error of the predicted values. Fixing $\lambda_2^{cv}$, the optimization problem in (\ref{crit_Theta}) becomes a group-lasso problem and we propose to select $\lambda_1$ via $k$-fold cross validation (for instance, the \textsf{cv.glmnet} package in \textsc{R}).

\end{document}